%% file: arXiv_Fast_Spawn_Prune_v1.tex
\documentclass[twoside,11pt]{article}

\usepackage[abbrvbib, preprint]{jmlr2e}
\usepackage{amsfonts}
\usepackage{xcolor}
\usepackage{algorithm,algpseudocode}
\usepackage{booktabs}
\usepackage{tabularx}
\input{maths_commands_jmlr}

\newtheorem{defi}{Definition}[section]
\newtheorem{rem}{Remark}[section]
\newtheorem{theo}{Theorem}[section]
\newtheorem{coro}{Corollary}[section]
\newtheorem{lem}{Lemma}[section]

 \newcommand{\bracket}[1]{\langle{#1}\rangle}
 \newcommand{\Tnua}{\mathsf{W}_{\nu,\alpha}}
 \newcommand{\Tnuka}{\mathsf{W}_{\nu_k,\alpha}}
 \newcommand{\Tnub}{\mathsf{T}_{\nu,\beta}}
 \newcommand{\sPF}{\hat{T}^{\alpha,\beta}_{\sharp}}
 \newcommand{\UU}{}
 \newcommand{\MM}{\mathbf{E}_{\infty}}
 \newcommand{\RR}{\mathfrak{R}}
 \newcommand{\CC}{\mathfrak{C}}
 \newcommand{\cc}{\mathfrak{c}}
 \newcommand{\CTV}{\mathfrak{C}_{TV}}

 \newcommand{\cX}{\mathcal X}
 \newcommand{\cH}{\mathcal H}
 
 \newcommand{\bbH}{\mathds H}

 \newcommand{\cM}{\mathcal M}
 \newcommand{\bbR}{\mathds R}

 \newcommand{\HTVC}{(\mathcal{H}_{\mathrm{TV}}^{\infty})}
 \newcommand{\Heps}{(\mathcal{H}_{\varepsilon})}
 \newcommand{\tnu}{\tilde{\nu}}
 \newcommand{\poschedule}{\widehat{C}_{k}}
 \newcommand{\negschedule}{\widehat{c}_{k}}

 \newcommand{\measet}{\cM(\cX)}

 \newcommand{\nuk}{\nu_k}
 
 \newcommand{\nukp}{\nu_{k^+}}
 \newcommand{\nukkp}{\nu_{k^{++}}}
 
 \newcommand{\nukm}{\nu_{k-1^+}}
 
 \newcommand{\nukpp}{\nu_{k+1}}
 
 \newcommand{\nuks}{\hat{\nu}_k}
 
 \newcommand{\nukps}{\hat{\nu}_{k^+}}
 
 \newcommand{\nukkpps}{\hat{\nu}^{++}_{k}}
 
 \newcommand{\nukpps}{\hat{\nu}_{k+1}}
 
 \newcommand{\nukms}{\hat{\nu}_{k-1^+}}
 
 \newcommand{\vkd}{v_{k-1^+}}
 \newcommand{\vks}{\hat{v}_{k-1^+}}
 
 \newcommand{\nutv}{\|\nu\|_{\mathrm{TV}}}
 \newcommand{\Lip}{\mathfrak{L}}
 \newcommand{\Pkp}{\widehat{\mathcal{P}}_{\nukp}}
 \newcommand{\Nkp}{\widehat{\mathcal{N}}_{\nukp}}
 \newcommand{\Thetakm}{\Theta_{k-1^+}}
 \newcommand{\Thetakms}{\widehat{\Theta}_{k-1^+}}

 \newcommand{\Fk}{\mathfrak{F}_k}
 \newcommand{\Fkm}{\mathfrak{F}_{k-1}}
 
 \newcommand{\Deltakp}{\Delta_{k^+}}
 \newcommand{\Deltakpp}{\Delta_{k+1}}

 \newcommand{\Deltakps}{\widehat{\Delta}_{k^+}}

 \newcommand{\epk}{\varepsilon_k}

 \newcommand{\epkm}{\varepsilon_{k-1}}

 \newcommand{\pos}{\bm{T}}
 \newcommand{\weights}{\bm{W}}

 \newcommand{\newtext}[1]{\noindent\textcolor{black}{#1}}

\definecolor{burgundy}{rgb}{0.5, 0.0, 0.13}
\definecolor{camel}{rgb}{0.76, 0.6, 0.42}
\definecolor{chamoisee}{rgb}{0.63, 0.47, 0.35}
\definecolor{grey1}{RGB}{128,128,128}
\definecolor{ydc-red}{HTML}{EE66AA}
\usepackage[colorinlistoftodos,textwidth=2.3cm]{todonotes}

\usepackage{lastpage}

\begin{document}

\title{Fast Spawn\&Prune (FS\&P): Global convergence of stochastic conic particle gradient descent via birth/death process}

\author{\name Yohann De Castro \email yohann.de-castro@ec-lyon.fr\\
       \addr École Centrale Lyon, CNRS UMR 5208, Institut Camille Jordan,\\ Écully, France.
       \AND
       \name Sébastien Gadat  \email sebastien.gadat@tse-fr.eu\\
       \addr Toulouse School of Economics, CNRS UMR 5314, TSE-R \\
       Toulouse, France.
         \AND
         \name Clément Marteau \email clement.marteau@math.univ-lyon1.fr\\
         \addr Université Lyon 1, CNRS UMR 5208, Institut Camille Jordan,\\ Villeurbanne, France.
         }

\maketitle

\begin{abstract}%
We investigate the global optimization of the objective function arising in continuous sparse regression, specifically the Beurling LASSO (BLASSO), over the space of measures. While Conic Particle Gradient Descent (CPGD) methods are computationally efficient, they may become trapped in local minima due to the non-convexity of the parameterization. To overcome this limitation, we introduce  \textsf{Fast Spawn\&Prune (FS\&P)}, a stochastic algorithm that extends \textsf{FastPart} introduced in \cite{FastPartv1} and combines CPGD with a birth–death process. The birth mechanism ensures asymptotic global exploration by introducing particles in regions where first-order optimality conditions are violated, while the death process preserves computational efficiency by pruning non-informative particles.
   We provide the first theoretical guarantee of global convergence for this class of discrete-time stochastic algorithms, without requiring exponentially large initializations.
    Furthermore, we derive explicit convergence rates for the excess risk, which scale as $\mathcal{O}\big(\left(\log K / K\right)^{\frac{1}{2(2+d)}}\big)$, where $K$ denotes the number of iterations and $d$ the dimension of the domain, thereby quantifying the trade-off between global exploration and local refinement.
    Moreover, the sample complexity is $\mathcal{O}\big(N^{-\frac{1}{4(2+d)}}\big)$ (up to logarithmic factors). We also propose a horizon-free variant that does not require prior knowledge of the iteration budget.    
\end{abstract}


\section{Introduction}
Continuous sparse regression has been at the core of numerous studies in statistics and signal processing. In particular, it encompasses a wide range of models and problems, including statistical \emph{mixture models}, \emph{deconvolution} problems, and \emph{neural networks}. We refer to \cite{candes2014towards,azais2015spike,de2021supermix,duval2015exact,giard2025gaussian}, among others.

In this paper, we do not focus on the statistical properties of the estimator $\mu^\star$, but rather on the underlying \emph{optimization problem}. Indeed, attaining the exact global minimum is not strictly necessary in practice; approximate solutions are often sufficient to achieve the desired statistical guarantees. While a thorough analysis of these properties lies beyond the scope of this paper, we refer the interested reader to Appendix~\ref{sec:statistical_guarantees} for a brief overview.

\paragraph{From Over-Parametrization to Global Convergence}
The optimization of a non-convex objective $J$ via particle discretization has undergone significant theoretical advances. The seminal works of \cite{chizat2018global,chizat2022sparse} established that, in the \emph{overparameterized} regime—where the number of particles is very large—gradient descent dynamics can benefit from a \emph{convex optimization landscape}. Specifically, in the mean-field limit, the gradient flow converges to the global optimum, provided that the initialization assigns strictly positive mass to every measurable subset of the domain $\mathcal{X}$. This regime effectively convexifies the problem by allowing mass to flow freely toward the optimal support.

However, the computational cost of deterministic particle gradient descent scales poorly with both the number of particles and the dataset size. To address this issue, recent works have investigated stochastic approximations \citep{FastPartv1} and sketching techniques \citep{poon2023geometry,de2025effective}. In previous work, the authors analyzed \textit{Stochastic Conic Particle Gradient Descent} (\textit{FastPart}) in the overparameterized regime. They showed that replacing exact gradients with unbiased stochastic estimators (via mini-batching and random features) significantly improves time complexity while maintaining strong stability guarantees, in particular the boundedness of the total variation norm along the trajectory and local convergence rates.

Despite these advances, a critical gap remains regarding \emph{global convergence} for discrete-time algorithms with sparse initialization (\textit{i.e.}, with few to a moderate number of particles). Standard gradient descent methods, including their stochastic variants, are primarily local search methods and struggle to transport mass to remote regions of the domain when the current support is far from the optimum. As a result, the discrete algorithm may remain near stationary points where the first-order optimality conditions—namely, the non-negativity of $J'_\nu$, the Fréchet derivative of the objective—are violated in regions devoid of particles.

\paragraph{Fast Spawn\&Prune and the Birth/Death Process}
In this paper, we introduce \textbf{FS\&P}, an algorithm that augments stochastic conic particle gradient descent with a \emph{Birth and Death} process. This mechanism is designed to bridge the gap between local descent and global exploration. 
The \textbf{Birth Process} acts as a global corrective mechanism. It detects violations of the first-order optimality conditions—specifically, regions~$\mathcal{N}_\nu$ where the so-called \emph{dual certificate} $J'_\nu$ satisfies $J'_\nu < 0$—and introduces new particles in these areas. This mechanism ensures asymptotic global exploration and prevents the dynamics from becoming trapped in local minima. Unlike greedy methods such as Frank--Wolfe algorithms, which require solving a global minimization problem for $J'_\nu$ to add a particle, our approach simply samples random points within the negative regions $\mathcal{N}_\nu$ of the certificate $J'_\nu$, making it significantly more computationally tractable. The \textbf{Death Process} maintains computational efficiency by pruning non-informative particles. Exploiting the convexity of the objective function over the space of signed measures, we show that removing a particle whose weight is small and for which the dual certificate $J'_\nu$ at its location is sufficiently large (i.e., in the regions $\mathcal{P}_\nu$) strictly decreases the objective value. This provides a rigorous criterion for reducing the number of particles—and consequently the computational cost—without compromising convergence guarantees.

\medskip
Our main theoretical contribution is to prove the \textbf{global convergence} of this scheme. The \emph{Stochastic Conic Particle Gradient Descent with Birth and Death} is a discrete-time algorithm that constructs a sequence of measures $(\nu_k)_{k \ge 0}$ from the discretized version of the objective function $J(\cdot)$ by encoding the measures as a finite sum of \emph{particles} (Dirac masses). Let $(\varepsilon_k)_k$ be a sequence of \emph{exploration parameters} controlling the intensity of the birth process at iteration $k$. At each iteration, the algorithm performs the following steps:

\begin{itemize}
   \item \textbf{Weight and Push-Forward Update:} Each particle's weight is updated via an exponential weighting scheme based on the local value of the dual certificate $J'_{\nu_k}$ (the so-called \emph{conic descent}). Simultaneously, the positions of the particles are adjusted using a so-called \emph{generalized gradient descent step} to ensure they remain within the domain $\mathcal{X}$. A descent lemma (Proposition~\ref{prop:incre}) quantifies the decrease in the objective function due to these updates, giving the intermediate update $\nu_k\longmapsto\nu_{k^+}$.
\item \textbf{Birth Process:} New particles are introduced by sampling from regions $\mathcal{N}_{\nukp}$ where the (updated) dual certificate~$J'_{\nu_{k^+}}$ is negative (which can be done in practice using \emph{rejection sampling}, for instance). The number of new particles added is proportional to the \emph{exploration schedule parameter} $\varepsilon_k$, which decays over time to balance exploration and exploitation.
\item \textbf{Death Process:} Particles located in regions $\mathcal{P}_{\nukp}$ where $J'_{\nu_{k^+}}$ is sufficiently large are safely removed (with theoretical guarantees) from the measure. This pruning step helps control the total number of particles, ensuring computational efficiency.
\item \textbf{Stochastic Gradient Estimation:} To further enhance scalability, the algorithm employs stochastic approximations of the gradient using \emph{mini-batches} of data and \emph{random feature} mappings (sketching). This reduces the computational burden associated with evaluating the full gradient at each iteration.
\end{itemize}
We start by presenting the continuous sparse regression framework and the key mathematical objects involved in the optimization problem. We then detail the \textbf{Fast Spawn\&Prune} algorithm, including the weight and push-forward updates, as well as the Birth and Death processes. Next, we derive explicit convergence rates that depend on the dimension of the domain $\mathcal{X}$, reflecting the cost of global exploration. We show that, under suitable assumptions on the exploration schedule $(\varepsilon_k)_k$, the sequence of iterates converges to the global optimum.

\subsection{Continuous sparse regression}
Let $\mathcal{X}\subset\R^d$ be a \emph{compact convex} set (equal to the closure of its interior) and consider $(\mathcal{M}(\mathcal{X}),\|\cdot\|_{\mathrm TV})$ the space of \emph{signed} measures, defined as the topological dual space of the space $(\mathcal C(\mathcal{X}),\|\cdot\|_{\infty})$, the \emph{continuous} functions endowed with the infinity norm. Let~$\mathbb H$ be a \emph{separable} Hilbert space and let $\Phi\,:\,\mathcal{M}(\mathcal{X})\to \mathbb H$ be a linear map, referred to as the (forward) measurement operator. We define the BLASSO problem \citep{candes2014towards,azais2015spike} as
\begin{align}
    \label{eq:blasso}
    \mu^\star \in  \argmin_{\mu \in \cM(\cX)} J(\mu)
    \quad\textnormal{where}\quad 
    J(\mu):=\frac{1}{2} \big\Vert {y} - \Phi \mu \big\Vert_{\mathbb{H}}^2 + \kappa \Vert \mu \Vert_{\mathrm TV}\,, 
    \tag{$\mathcal{P}$}
\end{align}
where $\kappa>0$ is a regularization parameter and ${y}\in\mathbb H$ is some observation. We assume that~$\Phi$ is a \emph{bounded linear} and \emph{weak-* continuous} operator and one can prove (see Lemma~\ref{lem:dual_Phi} in the appendix) that 
\begin{subequations}
\begin{equation}
\label{def:Phi}
    \Phi:\ \nu\in\measet\longmapsto\int_\cX\varphi_\vt\mathrm d\nu(\vt) \in\bbH\,,
\end{equation}
where $ \vt\in\cX\longmapsto\varphi_\vt\in\bbH$ denotes the \emph{feature map}. One can define the \emph{model kernel} $K(\cdot,\cdot)$~as 
\begin{equation}
\label{eq:kernel}
\forall \vs,\vt \in \mathcal{X}\,,\qquad
K(\vs,\vt) := \langle \varphi_s,\varphi_t\rangle_{\mathbb H}\,.
\end{equation}
We make the following assumption on the program~\eqref{eq:blasso} throughout this paper, which is satisfied for standard kernels on compact sets $\cX$, such as the Gaussian kernel for instance.
\end{subequations}

\medskip

\noindent
\textbf{Assumption}
($\mathcal{H}_{\mathcal{P}}$). There exist constants $\cc_{\mathcal{P}}>0$ and $\CC_{\mathcal{P}}>0$ such that the observation ${y}\in\mathbb H$ is {\bf bounded} in~$\mathbb{H}$, namely:
$$\|y\|_{\mathbb{H}}\leq \CC_{\mathcal{P}},$$ and the kernel $K(\cdot,\cdot)$ introduced in \eqref{eq:kernel} is 
\begin{subequations}
\label{eqs:hyp_HP}
\begin{itemize}
    \item {\bf Smooth:} Twice continuously differentiable
    \begin{align}
            \addtocounter{equation}{1}
        \tag{\theequation--$(\mathcal{H}_{\mathcal{P}})$}
    \label{eq:upper_bound_kernel_gradient_hessian}
        \max\Big\{\|K(\cdot,\cdot)\|_\infty,
        \|\nabla_sK(\cdot,\cdot)\|_\infty,
        \|\nabla^2_sK(\cdot,\cdot)\|_\infty
        \Big\}\leq\CC_{\mathcal{P}}\,,
    \end{align}
    where $\|\cdot\|_\infty$ is the infinity norm (each case over all $\vs,\vt\in\cX$) of the absolute kernel value, the Euclidean norm of the kernel gradient with respect to the first variable and the operator norm of the kernel Hessian with respect to the first variable, respectively.
    \item {\bf Normalized and positive:}
    \begin{equation}
                \addtocounter{equation}{1}
        \tag{\theequation--$(\mathcal{H}_{\mathcal{P}})$}
    \label{eq:normalization_kernel}
        \forall \vs,\vt\in\cX\,,\qquad K(\vs,\vt)\geq \cc_{\mathcal{P}} >0\quad\text{and}\quad K(\vt,\vt)=1\,.     
    \end{equation}
\end{itemize}
\end{subequations}    

\begin{rem}
\label{rem:K_leq_1}
    Note that Equation \eqref{eq:normalization_kernel} yields $K(\vs,\vt)\in[\cc_{\mathcal{P}},1]$, for any $\vs,\vt\in\cX$. Note also that (see Lemma~\ref{lem:lipschitz_feature_map}) the so-called \emph{kernel metric} $d_K(\cdot,\cdot)$ (see \cite[Eq.~(4.80)]{steinwart2008support}) satisfies
    \begin{equation}
        \label{eq:control_kernel_distance_euclidean_distance}
            \forall \vs,\vt\in\cX\,, \qquad d_K(s,t):=
            \|\varphi_t-\varphi_s\|_{\mathbb H}\leq  \sqrt{\CC_{\mathcal P}}\|t-s\|\,,
    \end{equation}
    where $\|\cdot\|$ is the Euclidean norm on $\mathbb{R}^d$. We will use the Lipschitz property throughout this~paper.
\end{rem}

\subsection{First order optimality condition}
\paragraph{Symmetrization trick}
Following \cite{chizat2022sparse}, we address the optimization problem over the space of signed measures $\mathcal{M}(\mathcal{X})$ by lifting it to the space of \emph{non-negative} measures, denoted by $\mathcal{M}_+(\mathcal{X})$. This standard argument is presented in Section~\ref{sec:symmetrization} (in the appendix) and, from now on, we assume that the feasible set of~\eqref{eq:blasso} is $\mathcal{M}_+(\mathcal{X})$.

\paragraph{Fr\'echet differentiation: the dual certificate}
\begin{subequations}
For any $\nu \in \mathcal{M}_+(\mathcal{X})$, the Fréchet derivative of $J(\cdot)$ is denoted by $J'_\nu\in\mathcal{C}(\mathcal{X})$ (referred to as the \emph{dual certificate}). By Lemmas~\ref{lem:1.2a} and~\ref{lem:1.2b} (in the appendix), the dual certificate enjoys the following equality, for all $\nu \in \cM_+(\cX)$ and $\sigma\in\cM(\cX)$ such that $\nu+\sigma \in \cM_+(\cX)$,
\begin{equation}
    \label{eq:Frechet_Jprime}
    J(\nu+\sigma) - J(\nu) = \langle J'_\nu,\sigma\rangle + \frac12\|\Phi(\sigma)\|_\mathbb{H}^2\,,
\end{equation}
and
\begin{equation}
\label{eq:Frechet_J} 
\forall{t}\in \mathcal{X}\,,\qquad
J_\nu'({t}) = \langle \varphi_{{t}}, \Phi(\nu)- {y} \rangle_\mathbb{H} + \kappa\,.
\end{equation}

\medskip

\begin{rem}[Lipschitz continuity of the dual certificate]
    \label{rem:lipschitz_J_prime}
    As established in Lemma~\ref{lem:lipschitz_J_prime} (see the appendix), the dual certificate $J'_{\nu}$ is Lipschitz continuous with respect to the spatial variable $t \in \mathcal{X}$. Specifically, its Lipschitz constant $\mathfrak{L}({\nu}) $ satisfies 
    \[
        \mathfrak{L}({\nu}) \leq \sqrt{\CC_{\mathcal{P}}} (\CC_{\mathcal{P}}+\|\nu\|_{\mathrm{TV}})\,,
    \]
    where $\CC_{\mathcal{P}}$ is the Lipschitz constant of the feature map $t \longmapsto \varphi_t$ (derived from the kernel smoothness in Assumption $(\mathcal{H}_{\mathcal{P}})$). This smoothness is crucial for the descent lemma (Proposition~\ref{prop:incre}), as it ensures the gradient does not vary too much between close particles. Note that this bound depends on the mass of the measure $\|\nu\|_{\mathrm{TV}}$, which we prove remains bounded by a constant $\CTV>0$ throughout the algorithm. We can define the \emph{uniform Lipschitz constant} $\Lip$ of the dual certificates $J'_{\nu}$ as 
    \begin{equation}
    \label{def:lip_uniform_J'}
            \Lip:=\sqrt{\CC_{\mathcal{P}}} (\CC_{\mathcal{P}}+\CTV)\,.
    \end{equation}
\end{rem}
\end{subequations}

\medskip

\begin{rem}[Bounds on the dual certificate]
    Note that the dual certificate~\eqref{eq:Frechet_J} satisfies, for any $\nu \in \cM_+(\cX)$ and any ${t}\in \mathcal{X}$,
    \begin{equation}
    \label{eq:lower_bound_DC}
        \|\nu\|_{\mathrm{TV}}+\CC_{\mathcal{P}}+\kappa
        \ge
        J_\nu'({t}) 
        =
        \int_\cX K(s,t)\mathrm{d}\nu(s) 
        - 
        \langle \varphi_{{t}}, {y} \rangle_\mathbb{H}
        +
        \kappa
        \geq \cc_{\mathcal{P}}\|\nu\|_{\mathrm{TV}}-\CC_{\mathcal{P}}+\kappa  
    \end{equation}
    under \eqref{eqs:hyp_HP}. The bound \eqref{eq:lower_bound_DC} will be used to prove that the measure updates $(\nu_k)_k$ have bounded total variation. 
\end{rem}

\medskip

\paragraph{Optimality condition}
The dual certificate $J'_\nu$ plays a central role in solving the problem. In particular, it enables the characterization of solutions to our optimization problem, thereby emphasizing its importance. We have the following proposition; see, for instance, \cite[Proposition 3.1]{chizat2022sparse}.
\begin{proposition}[Karush–Kuhn–Tucker (KKT) conditions]\label{prop:minimization}
A measure $\nu^\star$ is a minimizer of $\nu\longmapsto J(\nu)$ if and only if $J_{\nu^\star}'(\vt)\geq 0$ for all $\vt\in \mathcal{X}$ and $J_{\nu^\star}'(\vt)= 0$ when $\vt$ belongs to the support of $\nu^\star$.
\end{proposition}

\noindent
The KKT conditions in Proposition~\ref{prop:minimization} provide a natural rationale for the particle birth and death dynamics proposed in this algorithm. First, the requirement that $J'_{\nu^\star} \ge 0$ implies that any region where the current dual certificate is negative ($J'_{\nu} < 0$) corresponds to a local violation of optimality. This motivates a \textbf{Birth process} to inject mass into these under-represented areas. 
Second, the complementary slackness condition implies that the support of the optimal measure is contained within the zero level set of the dual certificate. Consequently, particles located in regions where $J'_{\nu}$ is strictly positive should be pruned or down-weighted to reduce the objective value, thereby justifying a \textbf{Death process}.

\subsection{Weight \& Push-Forward update}
\label{sec:proj_gradient}
To adopt an optimization perspective on the functional $J(\cdot)$, we can formulate an ideal algorithm that recursively generates a sequence of measures $(\nuk)_{k \ge 0}$ via a so-called conic gradient descent on $J(\cdot)$.

\paragraph{Generalized descent on $\cX$}
\begin{subequations}
While the dynamics of the weights of the measures are made explicit through an exponential-weight descent, we also need to update the positions of the support points of the measure. Since the measure is supported on $\cX$, we must therefore constrain the particles to remain in $\cX$. This is ensured by the following proximal approach, which we briefly describe below. We refer to \cite{Ghadimi-Lan-Zhang} for further details. Given a step size $\beta>0$, we define
\begin{equation}
    \forall \vt \in \cX\,, \quad \forall {v} \in \mathbb{R}^d\,, \qquad \vt^+_{t,v,\beta} := \argmin_{{u} \in \cX}\Big\{ \langle{u},{v}\rangle+\frac{1}{2 \beta} \|{u}-\vt\|^2\Big\}\,.  
\end{equation}
The \textit{generalized gradient descent step} associated with a descent vector ${v}$ is then defined as
\begin{equation}
\label{eq:gradient_generalise}
\pi_{\cX}(\vt,{v},\beta) := \frac{\vt- \vt^+_{t,v,\beta}}{\beta}
\quad
\text{so that }
\quad
\vt^+_{t,v,\beta}=\vt-\beta \pi_{\cX}(\vt,{v},\beta)\,.
\end{equation}

\end{subequations}

\medskip

\paragraph{Descent Property}
By considering step sizes $\alpha, \beta>0$, whose values will be specified and discussed throughout the paper, we introduce for any measure $\nu \in \cM_+(\cX)$  the mappings $\Tnua\,:\,\mathcal{X} \longrightarrow \mathbb{R}$ and $\Tnub\,:\,\mathcal{X} \longrightarrow \mathbb{R}^d$ defined as:
\[
    \forall \vt\in \mathcal{X}\,,\qquad
    \Tnua(\vt) = e^{-\alpha J_\nu'(\vt)} \quad \mathrm{and} \quad \Tnub(\vt) = \vt-\beta\, \pi_{\cX}(\vt,\nabla J_\nu'(\vt),\beta)\,.
\]

\medskip
\begin{defi}[Weight \& Push-Forward update]
\label{def:PF}
For any $\nu \in \cM_+(\cX)$, define the update as $\nu^+ := \Tnub^\sharp \Tnua\nu$ 
where the weight update $\Tnua\nu\in \cM_+(\cX)$ is given by, for any Borel set $\mathcal{B}\subseteq\mathcal{X}$,
\[
(\Tnua\nu)(\mathcal{B}) = \int_{\mathcal{B}} \Tnua(\vt) \mathrm{d}\nu(\vt) =  \int_{\mathcal{B}}e^{-\alpha J_\nu'(\vt)} \mathrm{d}\nu(\vt)\,,
\]
and the push-forward measure $\Tnub^\sharp \mu\in \cM_+(\cX)$, for any $\mu\in \cM_+(\cX)$, is defined by:
\[
\forall \psi \in \mathcal{C}(\mathcal{X})\,,\qquad\int_\mathcal{X} \psi(\vt) d \Tnub^\sharp \mu (\vt) = \int_\mathcal{X} \psi(\Tnub(\vt)) d\mu(\vt)\,.
\]
\end{defi}

\medskip

\noindent 
The next result provides a quantitative characterization of the effect of the $\Tnua$ and $\Tnub$ updates on the value of the objective $J(\cdot)$.

\begin{proposition}[Descent property]
\label{prop:incre}Assume that  \eqref{eqs:hyp_HP} holds. Then, for any $\nu\in\mathcal{M}_+(\mathcal{X})$ and $\CTV>0$ such that $\|\nu\|_{\mathrm{TV}} \leq \CTV$ and for any~$\alpha \ge 0$ and $\beta \ge 0$ such that:
\begin{equation}
\label{eq:small_learning_rates}
    \alpha < \frac{1}{10(1+ \CTV + \CC_{\mathcal{P}}+\kappa)(1\vee \CTV)}\,\quad \text{and} \quad 
    \beta \leq \frac{1}{2\CC_{\mathcal{P}}(\CC_{\mathcal{P}}+3\CTV)\,e^{1/5}}\, ,
\end{equation}
it holds that
        \[
        J(\nu^+)-J(\nu) \leq - \frac34 \left(\alpha \int_{\cX}  |J'_{\nu}|^2 d\nu\, 
        + \beta \int_{\cX}  \left\|\pi_{\cX}(\vt,\nabla J_\nu'(\vt),\beta)\right\|^2 d\nu \right).
        \]
where $\nu^+ := \Tnub^\sharp \Tnua\nu$.

\end{proposition}
\noindent The proof is deferred to Appendix~\ref{app:tec}. This result highlights both the strengths and limitations of the dynamics $\nu \longmapsto \nu^+$, thereby motivating the algorithmic modifications introduced below. First, the Weight \& Push-Forward update (Definition~\ref{def:PF}) enables the definition of an iterative sequence of measures $(\nuk)_{k \geq 1}$. Proposition~\ref{prop:incre} quantifies the associated ``descent'' property: it establishes a lower bound on the decrease of the objective, guaranteeing a minimal energy reduction for the transition $\nu \longmapsto \nu^+$.
    Secondly, while \cite{chizat2022sparse,de2021supermix} prove that the sequence $(\nu_k)_{k \ge 1}$ converges to a sparse measure $\nu_\infty$ satisfying $J'_{\nu_\infty} = 0$ on its support, they do not guarantee $J'_{\nu_\infty} \geq 0$ elsewhere. 
    Thus, standard conic particle gradient descent fails to satisfy the full optimality conditions of Proposition~\ref{prop:minimization} due to insufficient exploration of~$\cX$, and global convergence cannot be guaranteed in this setting.

\subsection{Birth and Death Stochastic conic particle gradient descent}
\label{s:BirthandDeath}

\begin{subequations}
\paragraph{The particle swarms and their non-convex program}
We consider a generic nonnegative measure composed of~$p$ Dirac masses, referred to as a \emph{particle swarm}. Let $\bm W := (\omega_1, \ldots, \omega_p)\in\bbR^p$, $\bm T := (\vt_1, \ldots, \vt_p)\in\bbR^{p\times d}$, and $\bm\kappa := (\kappa, \ldots, \kappa)\in\bbR^p$. We introduce
\begin{equation}
    \label{def:nu}
    \nu(\bm{W},\bm T) := \sum_{j=1}^p \omega_j \delta_{\vt_j},\,
\end{equation}
where, for any $j$, $\omega_j > 0$ denotes the weight of particle $j$. Then, the parametrization of $\nu$ in~\eqref{def:nu} yields:
\begin{align}
\label{def:F}
J(\nu(\bm{W},\bm T)) 
    & 
     =  \mathrm{F}(\bm{W},{\bm T})+\frac{1}{2}\| y \|_{\mathbb H}^2\quad\text{with}\quad 
     \mathrm{F}(\bm W,\bm T):=\langle \bm\kappa-k_{\bm T}, \bm W \rangle + \frac{1}{2}\bm W^\top K_{\bm T} \bm W\,,
\end{align}
where $\varphi_\vt\in\mathbb H$ is the feature map, $k_{\bm T} := (\langle y, \varphi_{\vt_1} \rangle_{\mathbb{H}}, \ldots, \langle y, \varphi_{\vt_p} \rangle_{\mathbb{H}}) \in \mathds{R}^p$, and $K_{\bm T}$ is a $(p \times p)$ kernel matrix with entries $K(\vs, \vt) := \langle \varphi_s, \varphi_t \rangle_{\mathbb{H}}$.

\medskip

While \eqref{eq:blasso} is convex over the space of measures, the parametric formulation \eqref{def:F} is non-convex due to the joint optimization of weights and positions. Nevertheless, if a solution to \eqref{eq:blasso} can be written as a particle swarm \eqref{def:nu} then the minimizer $(\bm{W}^\star, \bm{T}^\star)$ of $\mathrm{F}(\cdot,\cdot)$ is a global solution to \eqref{eq:blasso}. The existence of such sparse solutions is well-established under some conditions, see for instance \cite{duval2015exact,boyer2019representer}.
\end{subequations}

\paragraph{The stochastic version of the dual-certificate and its gradients}
In the following, we assume access to \emph{unbiased} stochastic estimators of both the dual certificate $J'_\nu$ and its spatial gradient $\nabla J'_\nu$. We will denote by $ \widehat{J'_{\nu}}(\cdot,Z)$ (resp. $\widehat{D_{\nu}}(\cdot,Z)$) the estimator of $J'_\nu$ (resp. $\nabla J'_\nu$) for some random variable $Z$ that captures the randomness of these approximations. We refer to \cite[Section~3]{FastPartv1} for examples and explicit constructions. In the following, we will require some properties on these estimators, as described in the following assumption. 

\medskip

\begin{subequations}
\label{eqs:hyp_gradient_sto}
\noindent
\textbf{Assumption $(\hat{\mathcal{H}}_{\mathrm{sto}})$: Stochastic unbiased gradients.}
    A random variable $Z$ exists such that:
    \begin{itemize}
        \item 
   $\forall \nu \in \mathcal{M}_+(\mathcal{X})$ and $\forall t \in \cX$,
\begin{equation}
\label{A1}
     \begin{cases} \widehat{J'_{\nu}}(\vt,Z) := J'_\nu(\vt) + \xi_{\nu}(\vt,Z)\\
    \widehat{D_{\nu}}(\vt,Z) := \nabla J'_\nu(\vt) + \zeta_{\nu}(\vt,Z)
    \end{cases}
    \text{with} \ \mathbb{E}_{Z}[\xi_{\nu}(\vt,Z)] = 0 \  \text{and} \  \mathbb{E}_{Z}[\zeta_{\nu}(\vt,Z)]=0_{\mathbb{R}^d}.
\end{equation}
\item 
There exist positive constants $\bm{H},\bm{G},\MM>0$ such that, almost surely,
\begin{equation}
\label{A2}
\forall \vt\in \mathcal{X}\,,\qquad
    |\xi_{\nu}(\vt,Z)| \vee \|\zeta_{\nu}(\vt,Z)\|
\leq \MM  \quad \mathrm{and} \quad \widehat{J'_{\nu}}(\vt,Z)  \geq  \bm{G}\,\|\nu\|_{\mathrm{TV}} - \bm{H} + \kappa\,.
\end{equation}
\item Almost surely, the \emph{stochastic dual-certificate} $\vt \longmapsto \widehat{J'_{\nu}}(\vt,Z)$ is uniformly $\Lip$-Lipschitz, regardless of the value of~$Z$ (the constant~$\Lip$ may be taken to be larger than the one appearing in~\eqref{def:lip_uniform_J'} if necessary).
 \end{itemize}
\end{subequations}
\begin{subequations}
\begin{rem}%
   In~\eqref{A2}, we assume a \emph{Hoeffding-type} condition on the centered random variables, along with a lower bound on the stochastic dual certificate. 
  The constant $\MM$ stands for the maximal size of the admissible noise level that perturbs the computation of $\widehat{J'_{\nu}}(\vt,Z)$ and $\widehat{D_{\nu}}(\vt,Z)$ at each iteration. Such a condition could be replaced by a sub-Gaussian assumption on the distribution of the noise.
   Regarding the lower bound (affine in $\|\nu\|_{\mathrm{TV}}$), note that it holds for the deterministic dual certificate under $(\cH_{\mathcal{P}})$; see Equation~\eqref{eq:lower_bound_DC}. This assumption will be used to show that the stochastic measure updates~$(\hat{\nu}_k)_{k \ge 0}$ remain bounded in total variation norm. 
\end{rem}

Interestingly, we can use these stochastic counterparts with a mini-batch strategy to reduce the variance of the approximations.  More precisely, at each step $k\geq1$, given $m_k\geq1$, a mini-batch sample size, we draw a $m_k$-sample of i.i.d.
random variables $\bm{Z}_{k+1}:= (Z_{1,k+1},\dots, Z_{m_k,k+1})$ satisfying \eqref{A1} and set 
\begin{equation}
\widehat{J'_{\hat\nu_k}}(\vt) := \frac{1}{m_k} \sum_{l=1}^{m_k} \widehat{J'_{\nuks}}(\vt,Z_{l,k+1}) \quad \mathrm{and} \quad \widehat{D_{k}}(\vt) :=  \frac{1}{m_k} \sum_{l=1}^{m_k} \widehat{D_{\nuks}}(\vt,Z_{l,k+1})\,.
\label{eq:minibatch}
\end{equation}

\noindent
Given a measure $\hat\nu_k=\sum_{j=1}^{p_k} {\omega}_j^{k} \delta_{{\vt}_j^{k}}$ (as in \eqref{def:nu}) composed of $p_k$ particles, the Push-Forward update (Definition~\ref{def:PF}) leads to the following  stochastic update $\sPF$ at step $k\geq1$.

\begin{defi}[Stochastic push-forward update $\sPF$]
Define $\hat{\nu}_{k^+} = \sPF(\hat{\nu}_k)$ as $\hat{\nu}_{k^+} = \sum_{j=1}^{p_k} {\omega}^{k^+}_j \delta_{{\vt}_j^{k^+}}$ with
\begin{equation}
\label{eq:up_w_pos_sto}
{\omega}^{k^+}_j = \omega^k_j e^{-\alpha \widehat{J'_{\hat\nu_k}}({\vt}^k_j)}\quad \text{and} \quad 
{\vt}_{j}^{k^+} ={\vt}_{j}^k-\beta\,\pi_\mathcal{X}\big({\vt}_{j}^k,\widehat{D_k}({\vt}_j^k),\beta \big)\,.
\end{equation}
where $\pi_\mathcal{X}(\cdot,\cdot,\cdot)$ denotes the projection operator over $\mathcal{X}$ introduced in Section \ref{sec:proj_gradient}.
\end{defi}
\end{subequations}

\paragraph{The mass tweaking (Birth and Death)}
Given the stochastic push-forward update 
$
\hat{\nu}_{k^+} = \sum_{j=1}^{p_k} {\omega}^{k^+}_j \delta_{{\vt}_j^{k^+}}
$, we now describe our stochastic update $\nukps \longmapsto \nukpps$, whose evolution involves both deletion $\nukps \longmapsto \nukkpps$ and creation $\nukkpps \longmapsto \nukpps$ of weighted particles. To this end, we define the \emph{pushed dual certificate} $\widehat{J'_{\hat\nu_{k^+}}}$ as in~\eqref{eq:minibatch} with a new independent $m_k$ mini-batch $\bm{Z}_{k+1}^+$ and the particle swarm $\hat{\nu}_{k^+}$ (see Step~\ref{step:pushed_DC} of Algorithm~\ref{algo:SCPGD_pro}).
We also need a \emph{decision rule} $\mathrm{\sf DR}\,:\,\bbR\times\bbR\to\bbR$ that takes as input the pushed dual certificate values and the push-forward weights $(\widehat{J'_{\hat\nu_{k^+}}}({\vt}_j^{k^+}),{\omega}^{k^+}_j)$ and outputs a deletion intensity. Finally, we are given some positivity (resp. negativity) schedules $(\poschedule)_k$ (resp. $(\negschedule)_k$).

Again, mass deletion only concerns regions where $\widehat{J'_{\hat\nu_{k^+}}} \geq 0$ (given by~\eqref{eq:minibatch} with a new independent $m_k$ mini-batch $\bm{Z}_{k+1}^+$ and the particle swarm $\hat{\nu}_{k^+}$) whereas our algorithm adds some mass in regions where~$\widehat{ J'_{\hat\nu_{k^+}}} \leq 0$.
For this purpose, we use the super-level (resp. sub-level) set of positivity (resp. negativity) of~$\widehat{J'_{\hat\nu_{k^+}}}$, \textit{i.e.,} we define $\Pkp\subseteq\mathrm{\sf Supp}(\nukps)$ and $\Nkp\subseteq\cX$ as:
\begin{subequations}
\label{eq:def_M_N_P_intro}
\begin{align}
    \Pkp
    &:= 
    \Big\{ 
    t\in\mathrm{\sf Supp}(\nukps)\,:\,
    \mathrm{\sf DR}\big(\widehat{J'_{\hat\nu_{k^+}}}(t),\nukps(\{t\})\big)
    \ge  \poschedule
    \Big\}\\
    \Nkp
    &:=
    \Big\{ 
        t\in\cX\,:\,
        \widehat{J'_{\hat\nu_{k^+}}}(t) \le \negschedule
    \Big\}\,,
\end{align}
\end{subequations}
where $\mathrm{\sf Supp}(\cdot)$ denotes the support of a measure and $(\negschedule)_k$ is a sequence of (small) positive numbers. We construct 
${\omega}^{k^+}_j \longrightarrow {\omega}^{k^{++}}_j$ by removing mass on $\Pkp$ and $\nukkpps \longrightarrow \nukpps$ adding mass on $\Nkp$. Define
\begin{subequations}
\label{eq:stoscheme_intro}
\begin{align}
    {\omega}^{k^{++}}_j
    &:= 
    \big(1-\bm{1}_{\Pkp}({\vt}_j^{k^+})\big)\,{\omega}^{k^+}_j
    \,,
    \label{eq:stoscheme_del_intro}\\
    \nukpps 
    &:= 
    \sum_{j=1}^{p_k} {\omega}^{k^{++}}_j \delta_{{\vt}_j^{k^+}} + \varepsilon_k \bm{1}_{\Nkp}(U_{k+1})\,\delta_{U_{k+1}} 
    \,,
    \label{eq:stoscheme_create_intro}
\end{align}
where the random variables $(U_l)_{l\in \mathbb{N}}$ are assumed independent from the other random variables sampled at step $k$ and uniformly sampled over $\mathcal{X}$, and, for any Borel set $\mathcal B\subseteq\cX$, $\bm{1}_{\mathcal{B}}(\cdot)$ denotes the indicator function of $\mathcal B$. The scheme \eqref{eq:stoscheme_create_intro} is implementable as it only requires the generation of a uniform random variable over the space $\mathcal{X}$ and a single evaluation of~$\widehat{\mJ'_{\hat\nu_{k^+}}}$. Also, we emphasize that the Positivity $(\poschedule)_k$ and Negativity $(\negschedule)_k$ schedules can be chosen adaptively depending on the stochastic dual certificate $\widehat{\mJ'_{\hat\nu_{k^+}}}$ and the weights $ \omega_j^{k^+}$. For instance, a valid strategy for the Death process is to target the particle with the largest ratio certificate value over weight, \textit{i.e.}, by setting a threshold related to $\max_j \big\{\widehat{\mJ'_{\hat\nu_{k^+}}}( t_j^{k^+})/ \omega_j^{k^+}\big\}$. Examples of explicit decision rules and specific tuning for constants are provided in Section~\ref{s:FP2}.
\end{subequations}

\paragraph{The Fast Spawn\&Prune Algorithm}
We have now all the ingredients to design an implementable procedure. The previous steps are gathered in Algorithm \ref{algo:SCPGD_pro}. 
\begin{center}
\begin{algorithm}[!th]
\newtext{\caption{\newtext{Birth and Death Stochastic Conic Particle Gradient Descent (\textsf{FS\&P}) \label{algo:SCPGD_pro}}}}
\begin{algorithmic}[1]
\Require{Learning rates $\alpha,\beta>0$; Mini-batch size schedule $(m_k)_{k \ge 1}$}; Exploration schedule $(\varepsilon_k)_{k \ge 1}$; Positivity (resp.~Negativity) schedules $(\poschedule)_k$ (resp.~$(\negschedule)_k$); Decision Rule $\mathrm{\sf DR}\,:\,\bbR\times\bbR\to\bbR$;
\State{Weights: $\weights^k$ and Positions: $\pos^k$}; 
\Comment{No specific initialization required}
    \For{$k =1,\ldots, K$ }\Comment{$K$ gradient steps}
        \State{Set $\hat \nu_k\longleftarrow\nu(\weights^{k},\pos^{k})\,;$}\Comment{Particle swarm}
        \State{Sample $\bm{Z}_{k+1}\longleftarrow (Z_{1,k+1},\dots, Z_{m_k,k+1})$ and compute stochastic values and gradients 
\[
        \widehat{\mJ'_{\hat \nu_k}}(\vt_j^k) := \frac{1}{m_k} \sum_{\ell=1}^{m_k} \widehat{\mJ'_{\hat\nu_k}}(\vt_j^k,Z_{\ell,{k+1}})\quad  \text{and} \quad 
        \widehat{\mD_{k}}(\vt_j^k) := \frac{1}{m_k} \sum_{\ell=1}^{m_k} \widehat{\mD_{\nu_k}}(\vt_j^k,Z_{\ell,{k+1}})
        \,;
\]
        \Comment{Stochastic mini-batch variables \eqref{eq:minibatch};}        
        }

        \State{Update weights and positions $\displaystyle\hat{\nu}_{k^+}  \longleftarrow \sum_{j=1}^{p_k} {\omega}^{k^+}_j \delta_{{\vt}_j^{k^+}}$ with
        \[
        \omega^{k^+}_j = \omega^{k}_j e^{-\alpha \widehat{\mJ'_{\hat\nu_k}}(\hat\vt_j^k)}\quad \text{and} \quad
        \vt_{j}^{k^+} =\vt_{j}^{k}-\beta\pi_\mathcal{X}\big(\vt_{j}^{k},\widehat{\mD_{k}}(\vt_j^k),\beta\big)\,;
        \]
        \Comment{Stochastic push-forward update \eqref{eq:up_w_pos_sto};}}

        \State{\label{step:pushed_DC}
        Sample $\bm{Z}_{k+1}^+\longleftarrow (Z_{1,k+1}^+,\dots, Z_{m_k,k+1}^+)$ and compute stochastic pushed dual certificate 
\[
        \widehat{\mJ'_{\hat\nu_{k^+}}}(\cdot) 
        \longleftarrow
        \frac{1}{m_k} \sum_{\ell=1}^{m_k} \widehat{\mJ'_{\hat\nu_{k^+}}}(\cdot,Z_{\ell,{k+1}}^+)
        \,;
\]
        \Comment{Stochastic mini-batch variables \eqref{eq:minibatch};}        
        }

        \State{Sample $U_{k+1} \sim \mathcal{U}_\mathcal{X}$  independent from the rest (Uniform measure on $\cX$) and compute 
        \begin{align*}
            \bm{1}_{\Pkp}({\vt}_j^{k^+})
            &\longleftarrow 1\  
                \text{if}\ \Big\{\mathrm{\sf DR}\big(\widehat{J'_{\hat\nu_{k^+}}}({\vt}_j^{k^+}),{\omega}^{k^+}_j\big)    \ge  \poschedule\Big\}
                \ \text{and}\ 0\ \text{otherwise}\,;
        \\
            \bm{1}_{\Nkp}(U_{k+1})
        &\longleftarrow 1\  
                \text{if}\ \Big\{\widehat{J'_{\hat\nu_{k^+}}}(U_{k+1}) \le \negschedule\Big\}
                \ \text{and}\ 0\ \text{otherwise}\,;
        \\
        \hat{\nu}_{k+1}  
        &\longleftarrow
        \sum_{j=1}^{p_k} \big(1-\bm{1}_{\Pkp}({\vt}_j^{k^+})\big)\,{\omega}^{k^+}_j\, \delta_{{\vt}_j^{k^+}} + \varepsilon_k \bm{1}_{\Nkp}(U_{k+1})\,\delta_{U_{k+1}}\,;
        \end{align*}
        \Comment{Mass tweaking \eqref{eq:stoscheme_intro};}
        }
    \EndFor
\end{algorithmic}
\end{algorithm}
\end{center}

\subsection{Global convergence results}
\label{sec:global_res_intro}
Our theoretical analysis proceeds in two steps. First, we analyze a deterministic version of the algorithm with continuous updates (Section~\ref{s:continuous}). Theorem~\ref{theo:convergence_deterministe} establishes that this method converges to the global optimum $\mu^\star$, escaping local minima thanks to the Birth process. We derive an explicit convergence rate of order $O(K^{-\frac{1}{2(2+d)}})$ for the minimum gap $\min_{k\le K} \{J(\nu_k) - J(\mu^\star)\}$, where $d$ is the dimension of the domain~$\mathcal{X}$. This dependence on $d$ reflects the computational cost of global exploration in a non-convex landscape. 

Second, we extend these results to the fully stochastic \textsf{FS\&P} algorithm (Section~\ref{s:algosto}). Theorem~\ref{thm:cvgce_sto_all} proves that, under suitable choices of learning rates, mini-batch sizes, and exploration schedules, the expected excess risk converges to zero, yielding a global minimization result. Specifically, we obtain a global convergence rate of order $O((\log K / K)^{\frac{1}{2(2+d)}})$, confirming that \textsf{FS\&P} achieves global optimization with computationally efficient stochastic updates. The mini-batch size scales as $m = K$, resulting in a total sample complexity of $\mathcal{O}(N^{-1/(4(2+d))})$ up to logarithmic factors (Corollary~\ref{cor:sample_complexity}). Finally, Theorem~\ref{thm:horizon_free} provides a horizon-free variant with iteration-dependent schedules ($m_k = k$, $\varepsilon_k = 1/\sqrt{k}$, $\beta_k = 1/k$) that achieves the same sample complexity rate without requiring prior knowledge of the total number of iterations.

\subsection{Related works}

\paragraph{Convex programming for sparse optimization on measures.}
Continuous sparse regression, often framed as the BLASSO problem, has been extensively studied through the lens of convex optimization. Early foundational works by \cite{candes2014towards}, \cite{azais2015spike}, and \cite{duval2015exact} established exact recovery guarantees using semidefinite programming or grid-free methods, focusing on the statistical properties of the minimizer. More recent contributions, such as \cite{poon2023geometry}, \cite{giard2025gaussian}, and \cite{de2025effective}, have further refined these statistical error bounds and extended the analysis to various geometries and metrics. However, these works generally analyze the static optimization problem rather than the algorithmic dynamics required to solve it efficiently in high dimensions.

\paragraph{Over-parameterized Gradient Descent and Global Convergence.}
The dynamic approach, which involves optimizing particle positions and weights via gradient descent, relies heavily on over-parametrization. The seminal works of \cite{chizat2018global} and \cite{chizat2022sparse} analyzed the mean-field limit of these particle systems. They established that in the regime where the number of particles tends to infinity, the gradient flow converges to the global optimum, provided that the initialization covers the entire domain. However, for a finite number of particles, \cite{chizat2022sparse} only guarantees local convergence to stationary points, or global convergence under restrictive assumptions, such as an exponential number of particles at initialization. The gap between the global convergence of the continuous flow and the local convergence of the discrete algorithm remains a significant theoretical hurdle.

\paragraph{Stochastic Algorithms and FastPart.}
To address the computational complexity of deterministic gradient descent, which scales quadratically with the number of particles, stochastic approximations were introduced in \cite{FastPartv1}. This method, referred to as FastPart, utilizes mini-batching and random features to achieve a time complexity of $O(1)$ per iteration with respect to the number of particles. While \cite{FastPartv1} proved the stability of the algorithm (boundedness of the total variation norm) and established convergence rates to stationary points of order $O(\log K / \sqrt{K})$, it did not guarantee global convergence from arbitrary sparse initializations. This work builds upon that foundation by integrating a mechanism to escape local minima.

\paragraph{Birth and Death Processes in Optimization.}
The idea of adding mass to ensure global optimality has antecedents in the Frank-Wolfe (conditional gradient) algorithm \cite{Bredis_Pikkarainen_13}, where particles are added iteratively to the support. However, Frank-Wolfe methods require locating the new atom at the global minimum of the Fréchet derivative (the dual certificate, see~\eqref{eq:Frechet_J}), which amounts to solving a non-convex optimization problem exactly. Our Birth process is considerably more flexible: it suffices to draw a random point from a sub-level set of the Fréchet derivative (defined by the threshold~$\negschedule$ in Algorithm~\ref{algo:SCPGD_pro}). This flexibility comes at a price—a sub-linear convergence rate—but one that is dimension-free up to the exponent $1/(2+d)$, thereby quantifying the cost of global exploration via the Birth process. The Death process is essentially harmless: it does not prevent the loss from decreasing, yet it reduces the per-iteration complexity of the algorithm. Our Death procedure is mathematically grounded in the Fréchet derivative and provides a principled criterion for safely removing particles from the support. To the best of our knowledge, this point of view is new.

\subsection{Notation}
Throughout the paper, $\CC>0$ denotes a generic constant used for upper bounds, while $\cc>0$ denotes a generic constant used for lower bounds; both $\CC$ and $\cc$ may change from line to line. Both constants are independent of $k$ and $d$. The Euclidean norm of a vector $x \in \bbR^d$ is denoted by $\|x\|$. For any set $A$, we denote by $\bm{1}_A$ its indicator function. The support of a measure $\mu$ is denoted by $\mathrm{\sf Supp}(\mu)$. A list of notation is provided in Table~\ref{tab:notation} in Appendix~\ref{app:list_notation}.

\section{Birth process for the deterministic CPGD}
\label{s:continuous}

Our objective in this section is to develop ideas that enable an effective exploration of the space~$\mathcal{X}$. We begin by focusing on a simplified setting in which the measures~$\nu_k$ remain continuous throughout the iterative process, and restrict to weight-only updates ($\beta = 0$ in Definition~\ref{def:PF}). This framework is convenient for developing and understanding the theoretical tools that lead to global convergence of the optimization procedure, which will be extended to the stochastic setting in Section~\ref{s:algosto}. The extension to position updates ($\beta > 0$) is carried out in Appendix~\ref{sec:beta_positive} for completeness.

\subsection{Update evolution\label{sec:update}}
The limitations of Proposition~\ref{prop:incre} lie in the fact that it provides no information about the sign of $J'_{\nu_k}$ on the entire domain~$\cX$. The core idea we pursue is to explicitly add some mass, at each iteration~$k$, on subsets of $\cX$ where $J'_{\nu_k}$ is negative. 
To this end, we design an iterative algorithm that generates, at each step $k$, a triple of positive measures $(\nuk, \nukp,\nukpp)$: the intermediate measure $\nukp$ is computed by performing a gradient descent step on $J$ starting from $\nuk$, while $\nukpp$
is obtained by modifying the mass of $\nuk^+$ in relevant regions.
Specifically, the transition $\nu_k \longmapsto \nu_k^+$ corresponds to a weight update only—i.e., setting $\beta = 0$ in Definition~\ref{def:PF}.
 In contrast, the transition $\nukp \longmapsto \nukpp$ involves modifying the mass in regions where $J'_{\nu_k^+} \leq 0$, and possibly removing mass from regions where $J'_{\nu_k^+} \geq 0$.
Such a construction ensures that the support of $\nuk$ remains included in $\mathcal{X}$ throughout the iterations, thereby avoiding the need for any projection or correction steps.
Formally, the iterative scheme alternates between two steps. Let $\nu_0$ be any positive measure supported on $\mathcal{X}$ with $\|\nu_0\|_{\mathrm{TV}} < \infty$.

\medskip

\noindent
\textbf{Weight update} ($k \longmapsto k^+$). The intermediate measure $\nukp$ is obtained through
\begin{equation}
\nukp = \Tnuka\nuk.
\label{eq:cgpd}
\end{equation}

\medskip

\noindent
\textbf{Birth-death step} ($k^+ \longmapsto k+1$). The measure $\nukpp$ is obtained by modifying the mass of $\nukp$. We require this step to satisfy Assumptions~\eqref{eqs:H_eps_determinist} and~\eqref{eq:HTVC} defined below.

\medskip

\noindent
\textbf{Assumption} $(\mathcal{H}_{\varepsilon})$.
Let $(\epk)_{k \ge 0}$ be a decreasing sequence with $\varepsilon_0 = 1$, and let $\lambda$ denote the Lebesgue measure. The transition $\nukp \longmapsto \nukpp$ satisfies $(\mathcal{H}_{\varepsilon})$ if the following three conditions hold for every $k \ge 1$:
\begin{subequations}
\renewcommand{\theequation}{\arabic{parentequation}\alph{equation}--$(\mathcal{H}_{\varepsilon})$}
\label{eqs:H_eps_determinist}
\begin{align}
        J(\nukpp)-J(\nukp) \leq \CC\, \varepsilon_k^2\,,
                \addtocounter{equation}{1}
	           \tag{\theparentequation\alph{equation}--$(\mathcal{H}^{\mathrm{smooth},1}_{\varepsilon})$}
    \label{eq:smooth_1} 
    \\%
        \|J'_{\nukpp}-J'_{\nukp}\|_{\infty} \leq \CC\, \varepsilon_k\,,
        \addtocounter{equation}{1}
        \tag{\theparentequation\alph{equation}--$(\mathcal{H}^{\mathrm{smooth},2}_{\varepsilon})$}
    \label{eq:smooth_2}
    \\
        \nukpp(\mathcal{B})\ge \epk\, \lambda\big(\mathcal{B}\cap \{J'_{\nukp}\le 0\}\big)\,,
        \addtocounter{equation}{1}
        \tag{\theparentequation\alph{equation}--$(\mathcal{H}^+_{\varepsilon})$}
    \label{eq:mass}
    \end{align}
for any Borel set $\mathcal{B}\subseteq\mathcal{X}$.
\end{subequations}

\medskip

\noindent
\textbf{Assumption} $\HTVC$.
A constant $\CTV > 0$ exists such that
\begin{equation}
\label{eq:HTVC}
    \addtocounter{equation}{1}
        \tag{\theequation--$\HTVC$}
    \| \nuk \|_{\mathrm{TV}} \leq \CTV \quad \forall k\in \mathbb{N}\,.
\end{equation}

\noindent
Note that Assumption~\eqref{eq:HTVC} implies that the sequence of values $J(\nuk)$ remains bounded throughout the iterations. Indeed, a rough computation yields:

\begin{equation}
J(\nuk) \leq \tfrac{1}{2}\bigl(\| {y}\|_\mathbb{H} + \|\Phi \nuk \|_\mathbb{H}\bigr)^2 + \kappa \| \nuk \|_{\mathrm{TV}} \leq \tfrac{1}{2}\bigl(\| {y}\|_\mathbb{H} + \CTV\bigr)^2 + \kappa \CTV.
\label{eq:boundJ}
\end{equation}

We emphasize that, at this stage, no specific update rule has been defined yet for constructing $\nukpp$ from~$\nukp$. In the next paragraph (Section~\ref{s:birth-death-proc}), we will present a strategy that is algorithmically well-motivated and satisfies both assumptions \eqref{eqs:H_eps_determinist} and \eqref{eq:HTVC}.
Such a strategy may not be unique, and other examples of transition rules could be proposed. Our objective is to show that, as long as properties \eqref{eqs:H_eps_determinist} and \eqref{eq:HTVC} hold, we obtain global convergence results with explicit convergence rates.
In Section~\ref{sec:conv_det}, we leverage Proposition~\ref{prop:incre} to derive the global convergence of our method, along with explicit rates.

\subsection{Transition $\nukp \longrightarrow \nukpp$}
\label{s:birth-death-proc}

In this section, we describe the key ingredients for designing a transition 
\[
\nukp \longrightarrow \nukkp \longrightarrow \nukpp
\]
that ensures our assumptions \eqref{eqs:H_eps_determinist} and \eqref{eq:HTVC} are satisfied. This transition is divided into two steps. The first step consists of removing mass from regions where $J'_{\nukp} \ge 0$, which typically leads to a decrease in the energy $J$. The second step involves adding mass in regions where $J'_{\nukp} \le 0$ in order to enhance the effectiveness of the subsequent weight update in the transition $\nukkp \longmapsto \nukpp$.

\medskip

\noindent
\textbf{Mass deletion $\nukp \longrightarrow \nukkp$.} The idea is to remove some mass on the set where $J'_{\nukp}$ is positive. 
More precisely, we can decide to cancel some subset of the domain $\mathcal{X}$ by considering
\begin{equation}
\nu_{k^{++}} = \nu_{k^+} (1- \bm{1}_{\mathcal{P}_{\nukp}}) \quad \mathrm{with} \quad \mathcal{P}_{\nukp} = \left\lbrace J_{\nukp}'> -2\alpha^{-1} \log \varepsilon_k + \CC_w \right\rbrace \cap \left\lbrace J_{\nukp}' >0 \right\rbrace,
\label{def:nukpp_alt1}
\end{equation}
where $\CC_w$ is a constant defined below in Remark~\ref{rem:weight_update_perturbation}.

\medskip

\begin{rem}[Weight-update perturbation bound]
\label{rem:weight_update_perturbation}
The deletion set $\mathcal{P}_{\nukp}$ is defined in terms of $J'_{\nukp}$, but the density identity $d\nukp(\vt)=e^{-\alpha J'_{\nuk}(\vt)}\,d\nuk(\vt)$ involves $J'_{\nuk}$. The constant $\CC_w$ in~\eqref{def:nukpp_alt1} is the precise gauge that bridges the two. We claim:
\begin{subequations}
\begin{align}
\label{eq:weight_update_perturbation}
&\|J'_{\nukp} - J'_{\nuk}\|_\infty \;\leq\; \CC_w, \qquad \CC_w \;:=\; \alpha\,(\CC_\mathcal{P} + \CTV + \kappa)\, e^{\alpha(\CC_\mathcal{P} + \CTV + \kappa)}\, \CTV,\\[2pt]
\label{eq:Pnukp_inclusion}
&\mathcal{P}_{\nukp} \;\subset\; \{J'_{\nuk} > -2\alpha^{-1}\log\varepsilon_k\}, \qquad\text{so that }\; e^{-\alpha J'_{\nuk}(\vt)} < \varepsilon_k^2 \text{ on } \mathcal{P}_{\nukp}.
\end{align}
\end{subequations}
Under condition~\eqref{eq:small_learning_rates}, $\alpha(\CC_\mathcal{P} + \CTV + \kappa) < 1/10$, hence $\CC_w < \CTV/5$. The proofs of~\eqref{eq:weight_update_perturbation} and~\eqref{eq:Pnukp_inclusion} are deferred to Section~\ref{s:proof_remark_perturbation}.
\end{rem}
\medskip

\noindent
\textbf{Mass creation $\nukkp \longrightarrow \nu_{k+1}$.} We introduce the set $\mathcal{N}_{\nukp}$, associated with the negative part of $J'_{\nukp}$, and defined as:
\begin{equation}
    \label{eq:negative}
\mathcal{N}_{\nukp} := \big\{ J'_{\nukp} \le 0 \big\}.
\end{equation}
We emphasize that the set $\mathcal{N}_{\nukp}$ is determined by the values of $J'_{\nukp}$, and this point will be carefully addressed below. We add some mass on $\mathcal{N}_{\nukp}$
defined in Equation \eqref{eq:negative} and define:
\begin{equation}
    \label{def:nukpp}
    \nukpp := \nukkp + \epk \bm{1}_{\mathcal{N}_{\nukp}} \lambda 
\end{equation}
Figure \ref{fig:evolution} provides a schematic representation to ease the understanding of the evolution from $\nukp$ to $\nukpp$. We stress that our assumptions are general enough to allow for alternative update strategies. For the sake of clarity, we have only focused on one specific scheme, but alternative approaches could be investigated. \\

\begin{figure}[h!]
\centering
\begin{tikzpicture}[scale=0.85]
  \draw[->] (-7.2,0) -- (7.2,0) node[below right] {$\mathcal{X}$};
  \draw[->] (0,-3) -- (0,5.5) node[left] {$J'_{\nukp}$};

  \draw[thick,blue,domain=-6.5:6.5,smooth,variable=\x,samples=200]
    plot ({\x}, {2*sin(1.2*\x r) - 0.5*\x + 1});

  \filldraw[black!70!black] (-4.1,5) circle (2pt) node[above left] {Maximum};

  \draw[dashed] (-6.8,2.5) -- (6.8,2.5); %

  \draw[thick, dash pattern=on 4pt off 2pt on 1pt off 3pt] (-5.9,0) -- (-5.9,2.5); %
  \draw[thick, dash pattern=on 4pt off 2pt on 1pt off 3pt] (-2.7,0) -- (-2.7,2.5); %
\draw[thick, dash pattern=on 4pt off 2pt on 1pt off 3pt] (-5.9,0.1) -- (-2.7,0.1); %
\node[right, font=\small] at (-5.95,0.6) {$\downarrow \nukp(1-\bm{1}_{\mathcal{P}_{\nukp}}\!) $};

\draw[thick, dash pattern=on 4pt off 2pt on 1pt off 3pt] (6.5,5) -- (7.5,5);
\node[right, font=\small] at (7.6,5) {$\mathcal{P}_{\nukp}$};

\draw[thick, dash pattern=on 2pt off 3pt on 6pt off 2pt] (-1.55,-0.1) -- (-0.65,-0.1); %
\node[right, font=\small] at (-1.85,0.4) {$\uparrow + \epk \lambda $};
\draw[thick, dash pattern=on 2pt off 3pt on 6pt off 2pt] (2.55,-0.1) -- (7,-0.1); %
\node[right, font=\small] at (5,0.4) {$\uparrow + \epk \lambda $};
\draw[thick, dash pattern=on 2pt off 3pt on 6pt off 2pt] (6.5,4.5) -- (7.5,4.5);
\node[right, font=\small] at (7.6,4.5) {$\mathcal{N}_{\nukp}$};

\end{tikzpicture}
\caption{Evolution from $\nukp$ to $\nukpp$: decrease on $\mathcal{P}_{\nukp}$, increase on $\mathcal{N}_{\nukp}$.\label{fig:evolution}}
\end{figure}
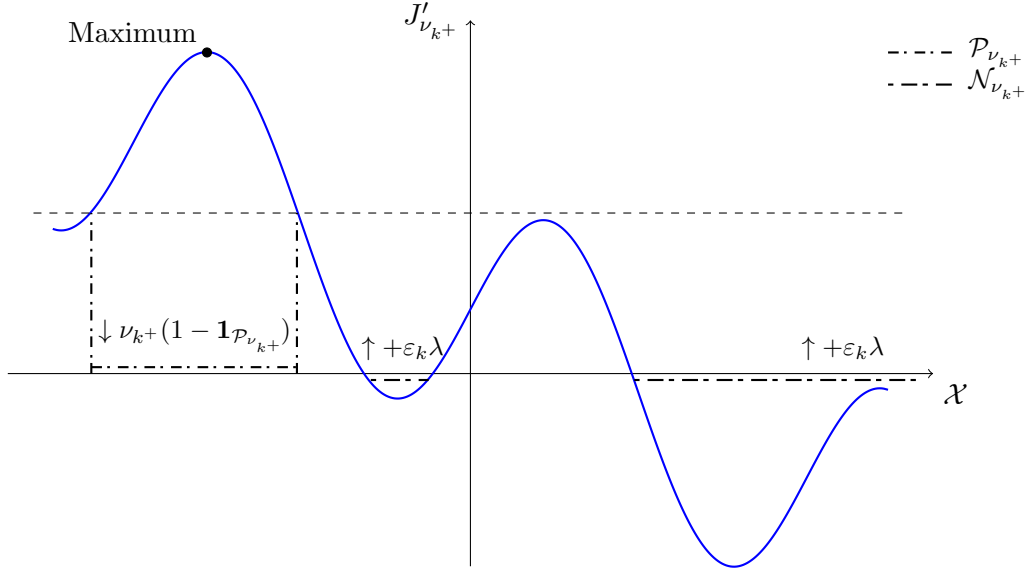

The next proposition ensures that the corresponding updated measure $\nukpp$ also satisfies the required assumptions: the boundedness of the total variation norm (Assumption~\eqref{eq:HTVC}) and the $\varepsilon$-smoothness (Assumptions~\eqref{eqs:H_eps_determinist}). The proof is deferred to Section~\ref{s:proof_hypdist}.%

\begin{proposition}
\label{prop:hypdist}
Assume~\eqref{eqs:hyp_HP}. Let $(\nuk)_{k \ge 1}$ defined according to~\eqref{def:nukpp_alt1} and~\eqref{def:nukpp}, and assume that the sequence $(\epk)_{k \ge 0}$ satisfies 
$\forall k \ge 0 \,: \epk \leq \UU  \alpha $.
Let us define 
$$
\RR := \frac{\|y\|_{\mathbb{H}}}{\cc_\mathcal{P}} e + \sqrt{\frac{e^3 \lambda(\cX)}{\cc_\mathcal{P}}}+\lambda(\cX)
$$
\begin{itemize}
    \item[$(\mathit{i})$] 
If we choose $\alpha \leq \frac{1}{1+\RR}$, then $\HTVC$ holds and:
$$
\forall k \ge 0 \qquad \|\nuk\|_{\mathrm{TV}} \leq  \|\nu_0\|_{\mathrm{TV}} \vee  2 \RR:= \CTV.
$$
    \item[$(\mathit{ii})$] If in addition $(\alpha,\beta)$ satisfies~\eqref{eq:small_learning_rates} (with the $\CTV$ of item i)), then Assumption~$\Heps$ holds true, \textit{i.e.} Equations~\eqref{eqs:H_eps_determinist} (\eqref{eq:smooth_1}, \eqref{eq:smooth_2} and \eqref{eq:mass}) are satisfied.
\end{itemize}
\end{proposition}

\subsection{Discussion: the role of \texorpdfstring{$\alpha$}{alpha}}
\label{sec:role_of_alpha}

Among the algorithmic parameters $(\alpha,\beta,\varepsilon_k,m_k)$, the weight learning rate $\alpha$ plays a distinguished role: it is fixed \emph{a priori} from problem data alone, while every other parameter is then calibrated as a function of $\alpha$ and of intrinsic problem constants. We make this hierarchy explicit, since two upper bounds on $\alpha$ appear in the analysis and a careless reading suggests a circular dependence.

\paragraph{An apparent circularity.} The descent property of Proposition~\ref{prop:incre} requires Inequality~\eqref{eq:small_learning_rates}, which involves the uniform total-variation bound $\CTV$. In turn, $\CTV$ is produced by Proposition~\ref{prop:hypdist}~$(\mathit{i})$, which itself requires $\alpha\leq (1+\RR)^{-1}$. Read in this order, the requirements appear to chain into a loop ($\alpha$ small with $\CTV$, then $\CTV$ from $\alpha$ small).

\paragraph{Resolution: $\alpha$ is fixed first, from intrinsic data.} The loop is only apparent. The constant
\[
\RR \;=\; \frac{\|y\|_{\mathbb{H}}}{\cc_{\mathcal{P}}}\,e \;+\; \sqrt{\frac{e^{3}\,\lambda(\cX)}{\cc_{\mathcal{P}}}} \;+\; \lambda(\cX)
\]
depends only on the problem data $(\|y\|_{\mathbb{H}}, \cc_{\mathcal{P}}, \lambda(\cX))$ and is independent of $\alpha$. As a consequence, so is the candidate TV bound
\begin{subequations}

\begin{equation}\label{eq:CTV_explicit}
\CTV \;:=\; \|\nu_0\|_{\mathrm{TV}} \,\vee\, 2\RR.
\end{equation}
Substituting~\eqref{eq:CTV_explicit} into~\eqref{eq:small_learning_rates} turns the descent constraint into a numerical one. Hence $\alpha$ may be fixed once and for all according to the explicit, self-contained rule
\begin{equation}\label{eq:alpha_explicit}
0 \;<\; \alpha \;\leq\; \min\!\left\{\frac{1}{1+\RR}\,,\;\frac{1}{10\,(1+\CTV+\CC_{\mathcal{P}}+\kappa)\,(1\vee \CTV)}\right\},
\end{equation}
whose right-hand side depends only on $(\|y\|_{\mathbb{H}}, \cc_{\mathcal{P}}, \kappa, \lambda(\cX), \|\nu_0\|_{\mathrm{TV}}, \CC_{\mathcal{P}})$. Under~\eqref{eq:alpha_explicit}, Proposition~\ref{prop:hypdist} delivers both $\HTVC$ and $\Heps$, and the descent inequality of Proposition~\ref{prop:incre} applies along the trajectory $(\nuk)_{k\geq 0}$.
\end{subequations}

\paragraph{Cascade of derived parameters.} Once $\alpha$ is fixed via~\eqref{eq:alpha_explicit}, the remaining parameters are calibrated downstream:
\begin{itemize}
\item the position learning rate must obey $\beta \leq \big(2\CC_{\mathcal{P}}(\CC_{\mathcal{P}}+3\CTV)\,e^{1/5}\big)^{-1}$ (the second part of~\eqref{eq:small_learning_rates}, with $\CTV$ now an explicit numerical constant);
\item the birth schedule satisfies $\varepsilon_k \leq \alpha$ for all $k$, with the horizon-dependent or horizon-free calibrations specified in Theorem~\ref{theo:convergence_deterministe}; the constraint $\max_k \varepsilon_k \leq \alpha$ reflects that small learning rates slow the exponential decay induced by $e^{-\alpha J'_{\nuk}}$, so the injected mass $\epk\lambda$ may accumulate before being absorbed.
\end{itemize}
The same hierarchy governs the stochastic setting of Section~\ref{s:algosto}: there, $\RR$ is replaced by $\widehat{\RR}=\tfrac{\bm{H}}{\bm{G}}\,e + \sqrt{e^{3}/\bm{G}} + 1$, intrinsic to Assumption~\eqref{A2}, and an additional Hoeffding cap $\alpha\leq \sqrt{8\log 8}\,/\,\MM$ enters via~\eqref{eq:hoeffding_alpha_condition}; the downstream calibration of $\beta$, $\varepsilon_k$ and $m_k$ proceeds identically.

\paragraph{Optimal magnitude of $\alpha$.} Within the admissible range~\eqref{eq:alpha_explicit}, the rates of Theorem~\ref{theo:convergence_deterministe} scale as $\alpha^{-1/(2+d)}$, and those of Theorem~\ref{thm:cvgce_sto_all} as $\alpha^{-1/2}$: it is therefore advantageous to take $\alpha$ as large as the constraints allow. This is the standard situation in deterministic optimization, where the descent inequality enforces a ceiling on the step size and saturating it yields the best worst-case rate.

\subsection{Global convergence result \label{sec:conv_det}}
We now establish a convergence result for the sequence $(J(\nu_k) - J(\nu^\star))_{k \in \mathbb{N}}$. In particular, we derive several distinct bounds on the sequence $( \min_{i \leq k} J(\nu_i) - J(\nu^\star))_{k \in \mathbb{N}}$ (assertions~i) and~ii) of Theorem~\ref{theo:convergence_deterministe}) and on the sequence $(J(\nu_k) - J(\nu^\star))_{k \in \mathbb{N}}$ itself (assertion~iii) of Theorem~\ref{theo:convergence_deterministe}). We also consider the cases where $\varepsilon$ depends on $k$ (horizon-dependent convergence) and where it does not (horizon-free convergence). Theorem~\ref{theo:convergence_deterministe} is stated with generic constants. The interested reader can refer to the proofs in Sections~\ref{sec:descent_det} and~\ref{s:proofconv1}, where the dependence of these constants with respect to our mathematical framework is made precise.

\begin{subequations}
\begin{theo}
\label{theo:convergence_deterministe}
 Assume that Assumption $\Heps$ stated in Equation \eqref{eqs:H_eps_determinist} holds, that $(\alpha,\beta)$ satisfies condition \eqref{eq:small_learning_rates} and that $\HTVC$ is satisfied. For any final horizon time $K \ge 2$, we have:
 \begin{itemize}
\item[$i)$]
 If $(\epk)_{k \geq 0}$ is \textit{non-adaptive} and
$ \epk= \varepsilon =  \sqrt{\frac{\CC}{ K}} \leq \alpha, \forall k \in \{1,\ldots,K\}$,
then we have:
\begin{equation}
\label{eq:ConvCont1}
\forall K \ge 2 \qquad
\min_{1 \leq k \leq K} \left\{ J(\nuk)-J(\nu^\star) \right\} \leq \CC\, \Lip^{\frac{2+2d}{2+d}}\alpha^{-\frac{1}{2+d}}K^{-\frac{1}{2(2+d)}}.
 \end{equation}
\item[$ii)$] If $(\epk)_{k \geq 0}$ is \textit{horizon-free} and  $\epk=
 \sqrt{\frac{\CC}{ (k+1)}} \leq \alpha$, then we have:
\begin{equation}
\min_{1 \leq k \leq K} \left\{ J(\nuk)-J(\nu^\star) \right\} \leq  \CC\, \Lip^{\frac{2+2d}{2+d}} \alpha^{-\frac{1}{2+d}}K^{-\frac{1}{2(2+d)}} \log(K)^{\frac{1}{(2+d)}}. %
 \end{equation}
\item[$iii)$] If $(\epk)_{k \geq 0}$ is
$ \epk = \varepsilon = \CC \left(\frac{\Lip^{2+2d}}{(d+1)\alpha } \right)^{\frac{1}{5+2d}} K^{-\frac{3+d}{5+2d}}$, then we have
$$J(\nuk)-J(\nu^\star)\leq \CC \left(\frac{\Lip^{2+2d}}{(d+1) \alpha} \right)^{\frac{2}{5+2d}} K^{-\frac{1}{5+2d}}.$$
 \end{itemize}
 In all three items, the generic constant $\CC$ may depend polynomially on $(\|\nu^\star\|_{\mathrm{TV}}/\Lip)$; in item~$iii)$, $\CC$ also depends polynomially on the initial excess $J(\nu_1)-J^\star$ (which is finite under~$\HTVC$).
\end{theo}
\end{subequations}

Items~$i)$ and~$ii)$ concern the minimal value of the functional $J$ along the first $K$ iterates of the algorithm, differing only in the choice of $\varepsilon$ at the end. Both results yield similar convergence rates, up to constants and a logarithmic term. In particular, using a horizon-free calibration for $\varepsilon$ introduces only an additional logarithmic factor in $k$ in the convergence rate.  
Item~$iii)$ provides a stronger result concerning the value of the last iterate, but the convergence rate obtained is slightly weaker than those in $i)$ and $ii)$. This difference is essentially technical, arising from a proof method based on a compensator/penalty strategy to construct a decreasing sequence—a method that results in a degraded convergence rate.
\begin{rem}
\begin{itemize}
    \item
    Although the setting considered in this section is specific — focusing on continuous measures and updates — we can still observe that the global convergence rates are slower than those reported in, e.g., \cite{chizat2022sparse} or \cite{FastPartv1}, which achieve rates of order $1/\sqrt{k}$. In our case, the convergence exhibits a dependence on the dimension: the volume of the region where $J'_{\nu_k} < 0$ at each iteration significantly influences the behavior of the objective $J$ (see, e.g., Proposition~\ref{prop:boundvk} and its proof). This is formalized by a geometric lemma on the volume of the target region, which ensures that $\lambda(\{J'_{\nu_k} \le 0\}) \ge C_d \mathfrak{L}^{-d} |\min_x J_{\nu_k}'(x)|^d$. This bounds the birth process probability from below, preventing it from vanishing arbitrarily fast compared to the loss gap. Nevertheless, unlike the aforementioned works, our method does not impose any \emph{local} specific constraints on the initialization measure $\nu_0$. The algorithm allows for a dynamic evolution of the support of $\nuk$ throughout the iterations, which enables \emph{global} convergence.  
However, this adjustment of the algorithm to redistribute the mass of the measure in regions where $J'_{\nuk}$ is negative affects the convergence speed in a dimension-dependent manner.

\item In the deterministic optimization literature, it is uncommon to observe convergence rates that depend on the problem’s dimension, particularly in convex optimization. The ellipsoid method in convex optimization, introduced by Shor, Yudin, and Nemirovsky (see \textit{e.g.} \cite{NemirovskiYudin}), achieves linear convergence rates, with the rate inversely proportional to the dimension. This degradation also stems from the curse of dimensionality, arising from the geometric optimization strategy—specifically, from the dependence of the ellipsoid’s volume on the dimension.  
Sometimes, particularly in Quasi-Newton and Newton methods, the dependence on the problem dimension is somewhat hidden within the overall computational cost: each iteration requires a number of operations that increases with $d$, while the total number of iterations remains essentially independent of $d$.

\item When comparing our method with that of Theorem~4.2 of \cite{chizat2022sparse}, we observe that their result is derived under the assumption that the initialization lies within a basin where a Polyak-Lojasiewicz inequality holds — an assumption that is both very strong and restrictive, and which does not hold in full generality.  
Moreover, Proposition~H.1 in \cite{chizat2022sparse} also exhibits a hidden dependence 
on the dimension in the setting where the density is uniformly lower bounded for the continuous-time gradient flow — an assumption that cannot be satisfied easily in a discretized counterpart. In particular, the proof of Proposition~H.1 in \cite{chizat2022sparse} is valid only in dimension $1$, and the rate also deteriorates as the dimension of the ambient space increases.
\end{itemize}
\end{rem}

\begin{rem}[Extension to $\beta > 0$]
The analysis of this section is carried out under the simplification $\beta = 0$ in order to develop and illustrate the theoretical tools—screening, birth-death dynamics, and descent inequalities—that will be central to the stochastic convergence theory of Section~\ref{s:algosto}.
Theorem~\ref{theo:convergence_beta} in Appendix~\ref{sec:beta_positive} extends Theorem~\ref{theo:convergence_deterministe} to $\beta > 0$ and establishes the same three convergence rates under the learning-rate condition~\eqref{eq:small_learning_rates}. The proof, given there for completeness, follows the same strategy with additional perturbation estimates controlling the effect of the position update~$T_{\nu,\beta}$.
\end{rem}

\section{Stochastic algorithm and stochastic convergence properties}
\label{s:algosto}

The previous section provides the main tools and ideas allowing the conic gradient descent to converge toward a global minimum.
The main ingredient consists of adding, at each iteration $k$, mass on some specific regions (namely where $J'_{\nu_k}$ is negative).
However, this principle is not feasible in practice since it involves continuous measures.
We investigate in this section the implementable Algorithm \ref{algo:SCPGD_pro} based on the birth and death process.  Similarly to the deterministic analysis, we first provide generic assumptions.
Then, we will exhibit some specific updates that will fit our requirements. We finally provide convergence results of the expected risk towards 0, leading to a global optimization result.

\subsection{Notation and assumptions on the stochastic update}
\label{s:assumption_sto}
We consider in this section the generic construction discussed in Section \ref{s:BirthandDeath}. Recall that the sequence $(\nuks)_{k\in \mathbb{N}}$ is built in two steps at each iteration $k$: first, a stochastic instance of a CPGD algorithm $\nuks \longrightarrow \nukps$, associated with a specific descent property; then, an additional update $\nukps \longrightarrow \nukpps$, which modifies the mass of the current measure at some strategic locations.
Both successive updates may involve stochastic computations. \\

\noindent
\textbf{Assumption} $(\hat{\mathcal{H}}_\mathfrak{F})$: There exist two increasing collections of $\sigma$-algebras $(\Fk)_{k\in\mathbb{N}}$ and $(\Fk^+)_{k\in\mathbb{N}}$ such that $\Fkm \subset \Fkm^+ \subset \Fk \subset \Fk^+$
for any $k\in\mathbb{N}$, and such that $\nuks$ is $\Fk$-measurable
and $\nukps$ is $\Fk^+$-measurable.
Equivalently, $(\Fk)_{k \ge 1}$ is adapted to $(\nuks)_{k \ge 1}$ and $(\Fk^+)_{k \ge 1}$ is adapted to $(\nukps)_{k \ge 1}$.\\

We impose the following requirements on the updates. Throughout, $\CC>0$ (resp.\ $\cc>0$) denotes a generic upper-bound (resp.\ lower-bound) constant that may change from line to line.

\begin{subequations}
\begin{itemize}
\item \textbf{Assumption} $(\hat{\mathcal{H}}^\infty_{\mathrm{TV}})$: There exists a constant $\CTV$ such that almost surely:
$$ \|
\hat \nu_k \|_{\mathrm{TV}} \leq \CTV\quad \forall k\in \mathbb{N}.$$

    \item Iteration $k \longmapsto k^+$: $\nukps$ satisfies the following descent property that is the stochastic counterpart of Proposition \ref{prop:incre}:\\
    \\
    \textbf{Assumption} $(\hat{\mathcal{H}}_D)$: For any $k \ge 1$:
    \[
    \mathbb{E}  \left[ J(\hat{\nu}_{k^+}) \big\vert \mathfrak{F}_{k} \right] - J(\hat\nu_k)  \leq  - \frac{\alpha}{2}  \|J'_{\hat\nu_k}\|^2_{\hat\nu_k} +
      \CC
\left(\frac{\alpha^2}{m_k}+\beta^2 +\frac{\beta}{m_k} \right),
    \]
    where $\alpha,\beta$ denote tunable parameters of the algorithm.
    
\item Iteration $k^+ \longmapsto k+1$: Let $a>0$ whose value will be made precise later on. The stochastic update $\hat \nu_{k+1}$ satisfies the following assumptions:
    \begin{itemize}  
    \item \textbf{Assumption}
$(\hat{\mathcal{H}}^+_{\varepsilon,a}):$
For any $k \ge 0$, 
for any Borel set $\mathcal{B}\subseteq \mathcal{X}$:
\begin{equation*}
\mathbb{E}\big[  \hat\nu_{k+1}(\mathcal{B})\, \vert \mathfrak{F}_{k}^+\big] \geq \cc \varepsilon_{k} \lambda\big(\mathcal{B}\cap \lbrace J'_{\hat\nu_{k^+}} <0 \rbrace\big) - 
\CC\varepsilon_k m_k^{-a}%
\label{eq:mass_sto}
\end{equation*}

where $\lambda$ stands for the Lebesgue measure.

\item \textbf{Assumption} $(\hat{\mathcal{H}}^{\text{smooth},1}_{\varepsilon}):$
For any $k \ge 1$:
    \begin{equation}
                \addtocounter{equation}{1}
        \tag{\theequation--$(\hat{\mathcal{H}}^{\text{smooth},1}_{\varepsilon})$}
\mathbb{E} \left[    J(\hat\nu_{k+1})-J(\hat\nu_{k^+}) \, \vert  \mathfrak{F}_{k}^+ \right] \leq \CC \left( \epk^2 + \epk \sqrt{\frac{\log m_k}{m_k}} \right) \quad
    \label{eq:smooth_1_as}
    \end{equation}
\item \textbf{Assumption} $(\hat{\mathcal{H}}^{\text{smooth},2}_{\varepsilon}):$
For any $k \ge 1$:
    \begin{equation}
                    \addtocounter{equation}{1}
        \tag{\theequation--$(\hat{\mathcal{H}}^{\text{smooth},2}_{\varepsilon})$}
    \|J'_{\hat\nu_{k+1}}-J'_{\hat\nu_{k^+}}\|_{\infty} \leq \CC \varepsilon_k \quad \text{a.s.}
    \label{eq:smooth_2_as}
    \end{equation}
            \end{itemize}
\end{itemize}
\end{subequations}
\noindent
The next section provides an example of an implementable stochastic update that satisfies all these assumptions.

\subsection{Fast Spawn\&Prune: a conic particle birth/death process}
\label{s:FP2}
Although the setting considered in Section \ref{s:algosto} is quite general, it encompasses Algorithm \ref{algo:SCPGD_pro}. The setting of Section \ref{s:algosto} may be seen as a minimal sufficient set of conditions for the global convergence of the stochastic dynamic. We provide here specific instances of the mass tweaking step introduced in \eqref{eq:stoscheme_intro}. Then, we prove that the resulting algorithm satisfies all the requirements introduced in  Section~\ref{s:assumption_sto}.

\subsubsection{Transition $\nukps \longrightarrow \nukpps$}
Inspired by Section \ref{s:birth-death-proc} for the deterministic side, we now describe our sequence of stochastic updates $\nuks \longrightarrow \nukps \longrightarrow \nukpps$, whose evolution still involves both deletion $\nukps \longrightarrow \nukkpps$ and creation $\nukkpps \longrightarrow \nukpps$ of weighted particles.
Again, mass deletion should only concern areas where $J'_{\hat\nu_k} \geq 0$ whereas our algorithm adds some mass in areas where $J'_{\hat\nu_k} \leq 0$. Since handling these objects might be time-consuming at each iteration, we deal instead with the sets of positivity and negativity of $J'_{\hat\nu_k}$, \textit{i.e.} we define $\Pkp$ and $\Nkp$ as:
\begin{align}
\label{eq:def_M_N_P}
\Pkp
    &:= 
    \left\{  t_j^{k^+}: \widehat{J'_{\nukps}}(t_j^{k^+})  \ge  0 \ \mathrm{and} \  \omega_j^{k+} \leq \sqrt{2}\epk\right\}\\ 
\Nkp
    &:=
    \Big\{ t \in \mathcal{X}: \widehat{J'_{\hat\nu_{k^+}}}(t)\le c_a \sqrt{\frac{\log m_k}{m_k}} \Big\}\,,\notag
\end{align}
where $c_a$ is a positive constant depending on $a>0$ {involved in $(\hat{\mathcal{H}}^+_{\varepsilon,a})$}.

\begin{subequations}
\paragraph{Mass deletion $\nukps \longrightarrow \nukkpps$}
\label{s:delet_sto}
We build $\nukkpps$ from $\nukps$ while removing some mass on $\Pkp$ and define:
\begin{equation}
\nukkpps:=  \nukps (1- \bm{1}_{\Pkp}(V_{k+1}) \delta_{V_{k+1}})  \quad \mathrm{with} \quad V_{k+1} \sim \mathcal{U}_{\mathrm{supp}(\nukps)}
\label{eq:stoscheme_del}
\end{equation}
\paragraph{Mass creation $\nukkpps \longrightarrow \nukpps$}
For any $k\in \mathbb{N}$, given $\hat\nu_{k^+}$, define
 \begin{equation}
    \nukpps = \nukkpps + \varepsilon_k \bm{1}_{\Nkp}(U_{k+1})
    \delta_{U_{k+1}} \quad \mathrm{with} \quad U_{k+1} \sim \mathcal{U}_{\mathcal{X}}\,.
    \label{eq:stoscheme}
\end{equation}
\end{subequations}
Contrary to the deterministic update described in the previous section, the scheme \eqref{eq:stoscheme} is implementable.
Indeed, it only requires the generation of a pair of uniform random variables over the spaces $\mathrm{supp}(\nukps)$ and $\mathcal{X}$ and pointwise evaluations of $\widehat{J}'_{\nukps}$.\\

\begin{rem}
We stress that the constant $c_a$ involved in~\eqref{eq:def_M_N_P} is positive. Hence, the set~$\Nkp$ is not exactly defined as the negative part of $\widehat{J}'_{\hat\nu_k}$. Recall that, according to the KKT conditions (see Proposition \ref{prop:minimization})  our initial target is the negative part of $J'_{\hat\nu_k}$. Since we use a stochastic approximation of $J'_{\hat\nu_k}$, we have to ensure a complete exploration of this latter set. In this context, choosing $c_a>0$ allows to control some kind of Type II error.
\end{rem}

Thanks to the previous definitions, we can now state that the sequence of measures $(\nuks,\nukps)_{k \ge 1}$ satisfies our previous assumptions. The proof of the next proposition is deferred to Section \ref{sec:appendix_sto}.

\begin{subequations}
\begin{proposition}
    \label{prop:hypsto}
    Define 
    \[ \widehat{\RR}= \frac{\mathbf{H}}{\mathbf{G}} e + \sqrt{\frac{e^3}{\mathbf{G}}} +1.\] 
    Then provided $\max_k \varepsilon_k \leq \alpha $ and $\alpha \leq (1+\widehat{\RR})^{-1}$, then $\|\hat\nu_k\|_{TV} \leq \widehat{\RR}\wedge \|\nu_0\|_{TV} := \CTV$ for all $k\geq 1$.  
    Assume moreover that $(\alpha,\beta)$ satisfies~\eqref{eq:small_learning_rates} and 
    \begin{equation}
        \label{eq:hoeffding_alpha_condition}
        \alpha \leq \frac{1}{\MM}\sqrt{8\log 8},
    \end{equation}
    and that the mini-batch schedule satisfies $p_k \leq p_0 + m_k$ (equivalently $k \leq m_k$, which holds in particular for $m_k = K$ and $m_k = k$).
    Then, for any $a>0$, picking the threshold $c_a=\MM\sqrt{2a}$ in \eqref{eq:def_M_N_P} ensures that the sequence $(\nuks,\nukps)_{k \ge 1}$ of Algorithm~\ref{algo:SCPGD_pro} enjoys
    $(\hat{\mathcal{H}}_D)$, $(\hat{\mathcal{H}}^+_{\varepsilon,a})$, $(\hat{\mathcal{H}}^{\text{smooth},1}_{\varepsilon})$, $(\hat{\mathcal{H}}^{\text{smooth},2}_{\varepsilon})$ and $(\hat{\mathcal{H}}^\infty_{\mathrm{TV}})$.
\end{proposition}
\end{subequations}

\subsection{Global convergence results}

We have now all the ingredients to establish convergence rates for Algorithm \ref{algo:SCPGD_pro}. The following theorem provides a control on the minimal value of the sub-optimality value sequence $(J(\hat\nu_k) - J(\nu^\star))_{k\in \{1,\ldots,K\}}$ in a finite horizon scenario.

\begin{theo}
\label{thm:cvgce_sto_all}
Assume that the stochastic sequences  $(\nuks,\nukps)_{k \ge 1}$ satisfy
 $\hat{\mathcal{H}}_D$, $(\hat{\mathcal{H}}^+_{\varepsilon,a})$ with $a \ge \frac{d}{2(2+d)}$, $\hat{\mathcal{H}}^{\text{smooth},1}_{\varepsilon}$, $\hat{\mathcal{H}}^{\text{smooth},2}_{\varepsilon}$ and $(\hat{\mathcal{H}}^\infty_{\mathrm{TV}})$.
 For any final horizon iterate $K$, we define $\hat{\rho}_K$ as the lowest value of the excess loss along the $K$ iterations:
\[
    \hat{\rho}_K := \min_{1 \leq k \leq K}
        \big\{J(\nuks)-J(\nu^\star)\big\}\,.
\]
If we choose:
$
\alpha\leq \frac{1}{\MM}\sqrt{8\log 8},
\quad
\beta\leq\frac{1}{\alpha^{d/4}\sqrt{K}},
\quad
\epk=\frac{1}{\sqrt{K}} \quad \text{and} \quad m_k=K,
$
then the sequence $(\nuks,\nukps)_{k \ge 1}$ verifies the \emph{global} convergence rate:
 \begin{equation}
     \label{eq:global_conv_rate_sto}
     \mathbb{E}[\hat{\rho}_K] \leq \CC \alpha^{-1/2} \left(\frac{\log K}{K}\right)^{\frac{1}{2(2+d)}}.
 \end{equation}
\end{theo}
The schedule above satisfies $\epk \leq \alpha$ together with~\eqref{eq:small_learning_rates} (instantiated with $\widehat{\RR}$ as defined in Proposition~\ref{prop:hypsto}), so the hypotheses of Proposition~\ref{prop:hypsto} are met. Note that this requires $K\geq 1/\alpha^2$.

\begin{coro}[Sample complexity]
\label{cor:sample_complexity}
Under the assumptions of Theorem~\ref{thm:cvgce_sto_all}, the total number of stochastic oracle calls over $K$ iterations is $N = K \times m = K^2$. Expressed in terms of~$N$, the convergence rate reads:
$$
\mathbb{E}[\hat{\rho}_K] \leq \CC \alpha^{-1/2} \left(\frac{\log N}{\sqrt{N}}\right)^{\frac{1}{2(2+d)}}.
$$
In particular, $\mathbb{E}[\hat{\rho}_K] = \mathcal{O}\left(N^{-\frac{1}{4(2+d)}} (\log N)^{\frac{1}{2(2+d)}} \right)$.
\end{coro}
\begin{proof}
Since $m = K$, the total number of oracle evaluations is $N = K^2$, hence $K = \sqrt{N}$. Substituting into~\eqref{eq:global_conv_rate_sto}:
\[
    \mathbb{E}[\hat{\rho}_K] 
        \leq \CC \alpha^{-1/2} \Big(\frac{\log \sqrt{N}}{\sqrt{N}}\Big)^{\frac{1}{2(2+d)}} 
        \leq \CC \alpha^{-1/2} \Big(\frac{\log N}{\sqrt{N}}\Big)^{\frac{1}{2(2+d)}}\,,
\]
and the result is proven.
\end{proof}

The horizon-dependent tuning of Theorem~\ref{thm:cvgce_sto_all} requires the knowledge of $K$ in advance to set $\beta$ and $m$.
The following result removes this requirement by using iteration-dependent schedules.

\begin{theo}[Horizon-free variant]
\label{thm:horizon_free}
Under the same assumptions as Theorem~\ref{thm:cvgce_sto_all}, choose the iteration-dependent schedules:
$$
m_k = k\vee 1, \quad \varepsilon_k = \min\!\Big(\alpha,\, \frac{1}{\sqrt{k\vee 1}}\Big), \quad \beta_k = \frac{1}{k\vee 1},
$$
with $\alpha$ a fixed constant satisfying \eqref{eq:small_learning_rates} (independent of $K$). 
Then, for any horizon $K \ge 4/\alpha^2$,
\begin{equation}
    \label{eq:horizon_free_rate}
    \mathbb{E}[\hat{\rho}_K] \leq \CC \alpha^{-1/2} \left(\frac{(\log K)^{3}}{K}\right)^{\frac{1}{2(2+d)}}.
\end{equation}
In particular, no prior knowledge of $K$ is needed: the algorithm is ``any-time''. The total number of oracle calls is $N = \sum_{k=1}^K k = \frac{K(K+1)}{2}$, yielding the sample complexity
$$\mathbb{E}[\hat{\rho}_K] = \mathcal{O}\left( N^{-\frac{1}{4(2+d)}} (\log N)^{\frac{3}{2(2+d)}} \right).$$
\end{theo}

\medskip

\noindent The proofs of Theorems~\ref{thm:cvgce_sto_all} and~\ref{thm:horizon_free} are displayed in Section \ref{s:cvgce_sto_all}. They follow essentially the same lines as the deterministic scheme discussed in Section \ref{s:continuous}. However, the stochastic approximation of the Fréchet derivative of $J$ along the trajectory of the algorithm requires careful control. We stress that the final convergence rate does not significantly differ from~\eqref{eq:ConvCont1}. The main difference with the latter is a logarithmic term in $K$, which is a consequence of the stochastic approximation. Furthermore, note that the cap $\varepsilon_k \le \alpha$ ensures the hypothesis $\epk \le \alpha$ of Proposition~\ref{prop:hypsto} for every $k \ge 0$, and the monotonicity $\beta_k \le \beta_0$ guarantees that~\eqref{eq:small_learning_rates} (instantiated with $\widehat{\RR}$) holds at every step.

\newpage

\section{Numerical experiments}
\label{sec:exp}

In this section, we illustrate the capabilities of the \emph{FS\&P} algorithm on two distinct tasks: density estimation using Gaussian Mixture Models (GMM) on synthetic data, and a regression task using two-layer neural networks (NN) on the California Housing dataset.
In both cases, we compare the standard Full-Batch and Stochastic Conic Particle Gradient Descent (CPGD) against our proposed versions augmented with the Birth and Death (BD) processes.
A companion Jupyter notebook reproducing every experiment, figure and table of this section --- with explicit cross-references to the paper --- is available at \url{https://github.com/ydecastro/Fast-Spawn-Prune} (see \texttt{section4\_companion.ipynb}).

\subsection{Gaussian Mixture Models (GMM) with Fixed Covariance}
In this first experiment, we illustrate the ability of our dynamic particle system to \emph{escape local minima thanks to the birth process}.
We consider a simple density estimation problem where we observe $n$ i.i.d.\ points $X_1,\dots,X_n \in \mathbb{R}^2$ from a 2D Gaussian mixture ($n=24,000$).
The goal is to recover the means and weights of this mixture by minimizing a regularized $L^2$ fitting criterion over the space of discrete positive measures as in~\cite{de2021supermix}.
The birth process acts as an \emph{exploration mechanism}. Following our theoretical birth criterion, the algorithm evaluates the Fréchet derivative of the objective (or the dual certificate)~$J'_\nu(t)$ at proposed locations.
When the first-order optimality condition is violated---that is, when~$J'_\nu(t)$ is negative in some region of the state space---new particles are spawned in these promising areas.
Empirically, the continuous infusion of mass in regions of negative gradient allows the interacting particle system to efficiently escape poor local configurations and quickly recover all the true modes of the synthetic distribution (see Figure~\ref{fig:gmm_positions}).

\paragraph{Loss function and Kernel}
The goal of the optimization is to minimize the regularized~$L^2$ distance between the smoothed empirical measure $\hat{f}_\tau^n = \frac{1}{n}\sum_{i=1}^n \phi_\tau(\cdot - X_i)$ and the model.
For a discrete positive measure $\nu = \sum_j \omega_j \delta_{t_j}$, the objective is given by:
\begin{equation}
    J(\nu) = \frac{1}{2}\|\phi_\tau * \nu - \hat{f}_\tau^n\|_{L^2(\mathbb{R}^2)}^2 + \kappa\|\nu\|_{\mathrm{TV}},
\end{equation}
where $\phi_\tau = \mathcal{N}(0, \tau^2 I_2)$ is a Gaussian smoothing kernel, see~\cite{de2021supermix} for further details on smoothing applied to GMM for BLASSO.
The associated feature map is $\varphi_t(x) = \mathcal{N}(x; t, (1+\tau^2)I_2)$, which yields the smooth Gram matrix $K(t_i, t_j) = \mathcal{N}(t_i; t_j, 2(1+\tau^2)I_2)$.

\paragraph{Verification of the assumptions}
Because the Gaussian kernel $K(t_i, t_j)$ is bounded, strictly positive, and twice continuously differentiable, this formulation perfectly satisfies the mathematical assumptions of the paper.
Most notably, it respects Assumption~$(\mathcal{H}_{\mathcal{P}})$ which requires $\mathcal{C}^2$ spatial smoothness for the theoretical descent lemma.

\subsection{Training two-layer neural networks on the California Housing dataset}
Next, we turn to a real-world dataset to assess the impact of the death process in typical machine learning applications.
We consider the regression task on the California housing dataset using a two-layer neural network with ReLU activations.
As the ReLU function is positively homogeneous, we can equivalently rewrite the network by normalizing the locations (the incoming weights of the neurons) to live on the unit Euclidean sphere $\cX = \mathbb{S}^{d-1}$.
The optimization is performed over this connected space by minimizing the regularized empirical risk with respect to the weights $\bm{W}$ and locations $\bm{T}$.
We use a stochastic gradient descent (SGD) scheme to update the weights and locations over a sequence of mini-batches.
In this context, while the birth process can introduce new neurons to increase the capacity of the network, the death process ensures that the size of the network remains controlled.
In practice, the death rule removes particles whose weight falls below a given threshold or whose contribution becomes negligible.
Our numerical experiments demonstrate that the death process harms neither the convergence nor the quality of the final predictive results.
Indeed, \emph{pruning the inactive neurons significantly reduces the computational burden and prevents over-parameterization without deteriorating the generalization error}.
This simple experiment perfectly illustrates the theoretical insights of the paper, indicating that the dynamic interactions and systematic particle suppressions do not prevent the system from reaching an optimal configuration.

\paragraph{Loss function and Kernel}
We apply the algorithm to a regression task on the California Housing dataset ($n=18,576$ samples, $d=8$ features).
The objective is to minimize the regularized Mean Squared Error (MSE):
\begin{equation}
    J(\nu) = \frac{1}{2} \|y - \Phi(\nu)\|_{\mathbb{H}}^2 + \kappa \|\nu\|_{\mathrm{TV}},
\end{equation}
with the prediction function being a two-layer ReLU network $f_\nu(x) = \Phi(\nu)(x)$.
Here, the feature map is given by affine ReLU neurons $\varphi_t(x) = \mathrm{ReLU}(\langle v, x\rangle + b)$, where the parameters $t = (v, b)$ are constrained to the unit ball $\bar{B}(0,1) \subset \mathbb{R}^{d+1}$. Concretely, the incoming weight vector $v$ is normalized to lie on the unit sphere $\mathbb{S}^{d-1}\subset\mathbb{R}^d$ (as described above) and the bias $b\in[-1,1]$; the pair $(v,b)$ therefore belongs to $\mathbb{S}^{d-1}\times[-1,1]\subset \bar{B}(0,1)\subset\mathbb{R}^{d+1}$, reconciling the two descriptions.
The empirical kernel relies on the inner product over the $n$ observations: $K(t_i, t_j) = \frac{1}{n}\sum_{k=1}^n \mathrm{ReLU}(\langle v_i, x_k\rangle + b_i)\mathrm{ReLU}(\langle v_j, x_k\rangle + b_j)$.

\paragraph{Verification of the assumptions}
Unlike the GMM experiment, the ReLU activation function is continuous but not everywhere differentiable.
Strictly speaking, it does not fully satisfy the bounded $\mathcal{C}^2$ spatial smoothness assumption~$(\mathcal{H}_{\mathcal{P}})$ established in the theoretical sections.
In practice, the gradients are computed using the ReLU subgradient, and the empirical method remains highly effective and robust despite this theoretical mismatch.

\subsection{Experimental Results and Discussion}

In this section, we present the empirical results of our experiments, highlighting the benefits of integrating the Birth and Death (BD) process into both Full-Batch and Stochastic Conic Particle Gradient Descent~(CPGD).

\paragraph{Birth and Death process}
Particles are removed from the support if their contribution becomes negligible or falls into regions where the dual certificate is strongly positive.
Specifically, a particle at $t_i$ with weight~$\omega_i$ is pruned if $J'_\nu(t_i) / \omega_i > \tau_{\text{death}}$.
Particles are added by evaluating $J'_\nu(t)$ on a set of randomly sampled candidate positions to find regions~$\mathcal{N}_\nu$ where the first-order optimality condition is violated.
For the stochastic setting, the targeted birth level is set to spawn particles where $J'_\nu(t) < \tau_{\text{birth}} \sqrt{\log(m_k) / m_k}$, where $m_k$ is the mini-batch size.

\subsubsection{Dynamics of the Birth and Death Process}

To understand how the network capacity adapts during training, we plot the number of active particles over the optimization iterations in Figure~\ref{fig:bd_events}.

\begin{figure}[htpb]
    \centering
    \includegraphics[width=0.48\linewidth]{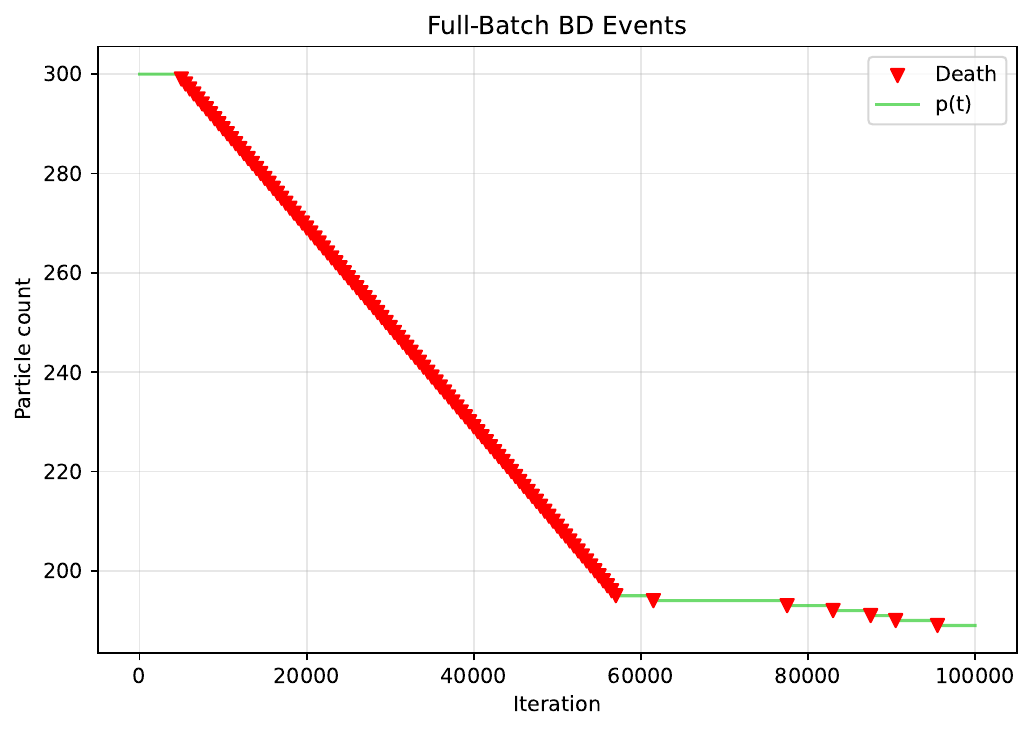}
    \includegraphics[width=0.48\linewidth]{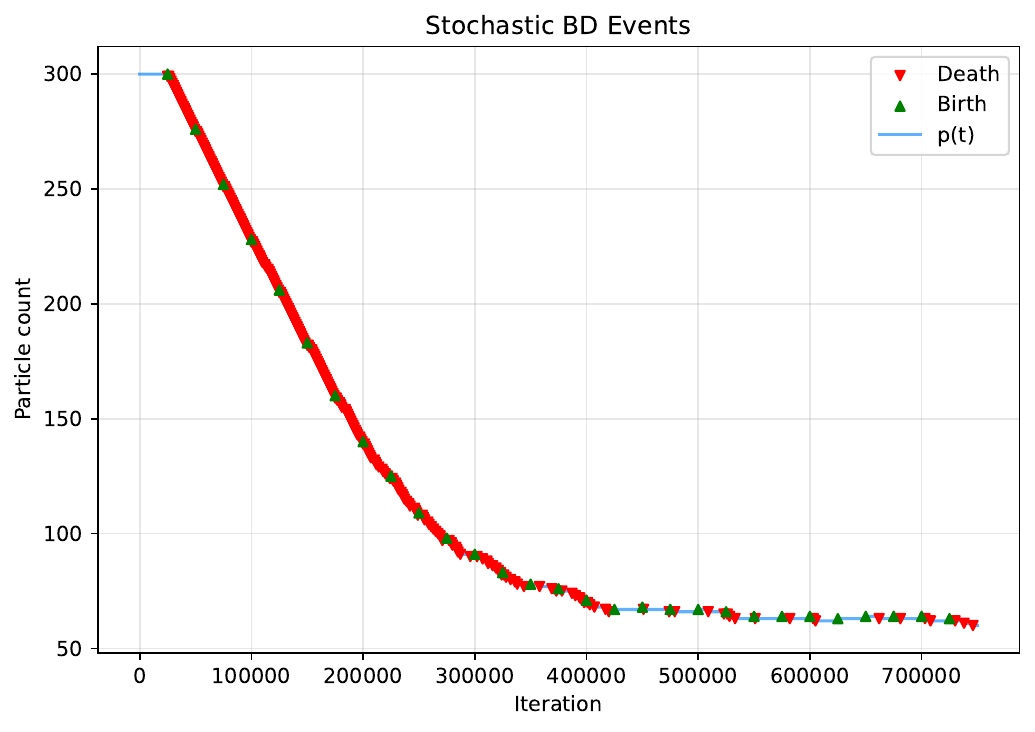}
    \includegraphics[width=0.48\linewidth]{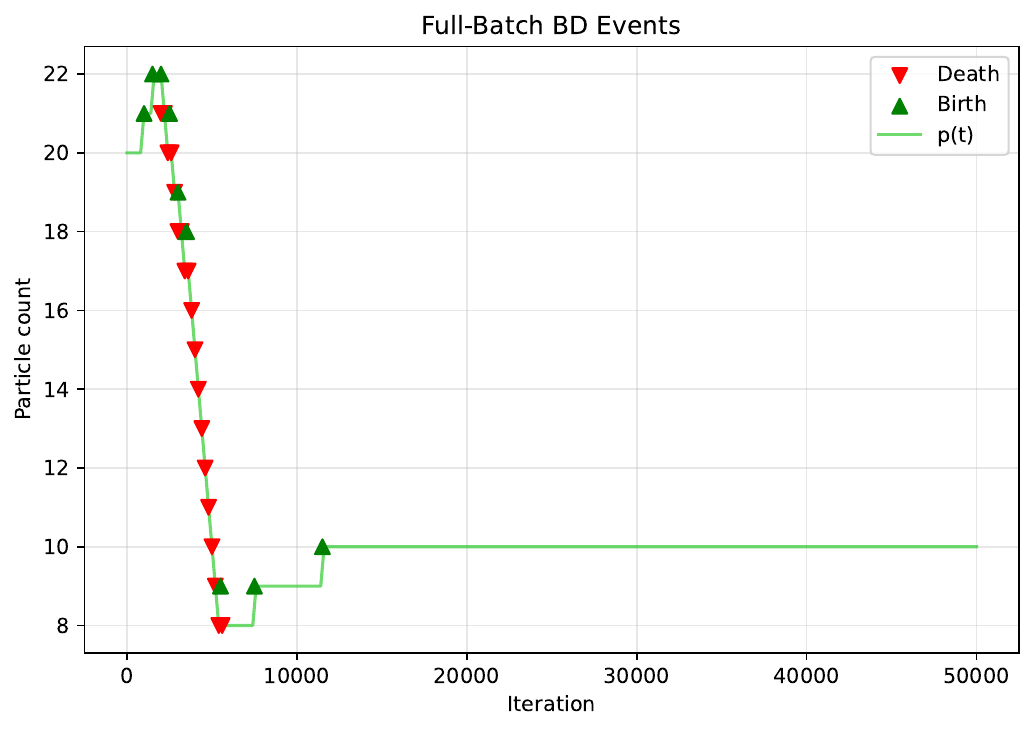}
    \includegraphics[width=0.48\linewidth]{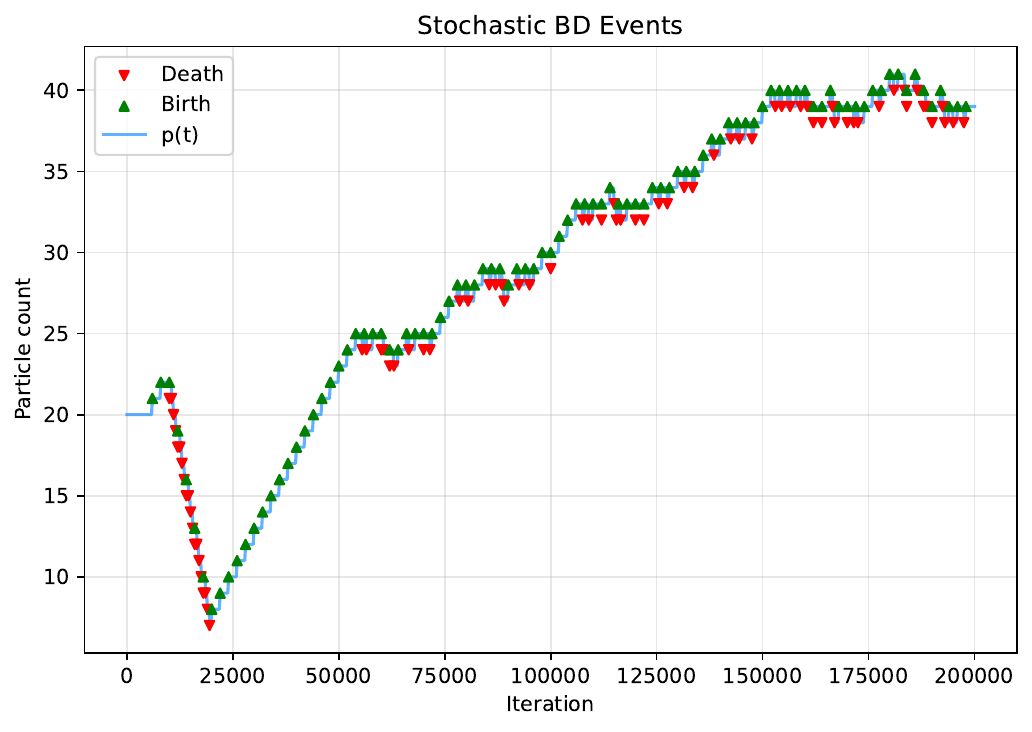}
    \caption{Particle count over iterations for GMM (bottom) and Neural Network (top) regression tasks.
    }
    \label{fig:bd_events}
\end{figure}

The dynamic capacity of the model is visible through the sharp drops (Death events pruning inactive particles) and spikes (Birth events) in Figure~\ref{fig:bd_events}.
In the Neural Network example (top row) we see that an over-parametrized layer is reduced by $50$--$70\%$ by the death process.
In the bottom row (GMM toy example), the right panel shows a pruning of spurious particles and an exploration phase (as in the left panel) which ends with a stabilization around $p=39$ (and $p=10$ in the left panel).

\newpage

\subsubsection{Convergence and Generalization}

Figure~\ref{fig:loss_vs_time} illustrates the temporal evolution of the optimization metrics.
We compare the standard CPGD methods against our proposed BD-augmented variants in terms of Wall-Clock time.

\begin{figure}[htpb]
    \centering
    \includegraphics[width=0.49\linewidth]{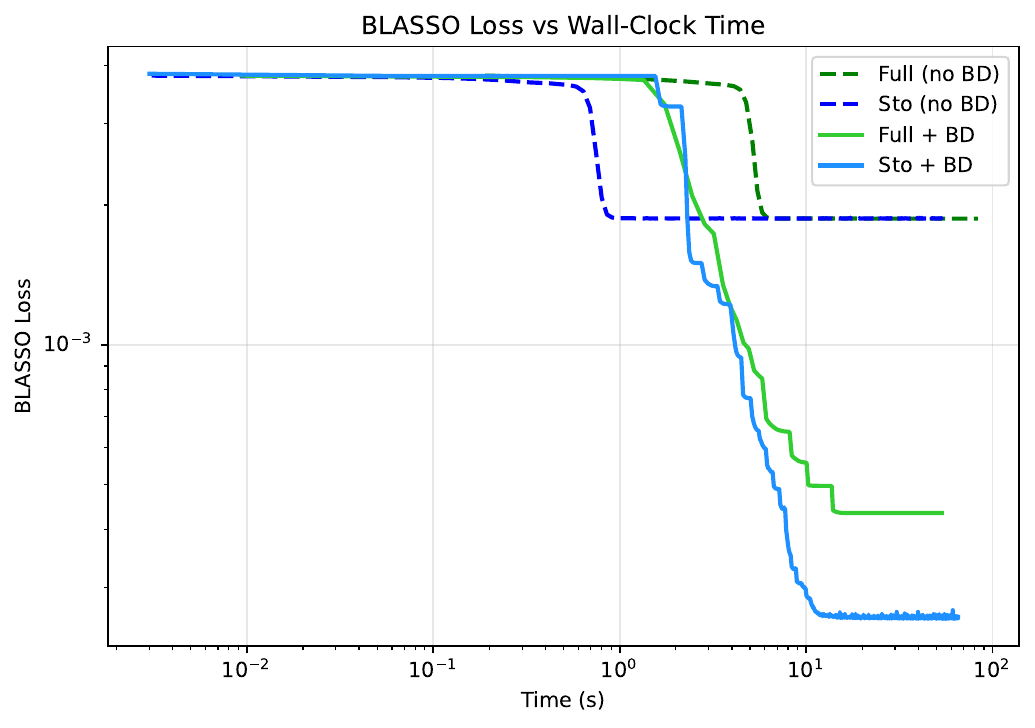}
    \includegraphics[width=0.49\linewidth]{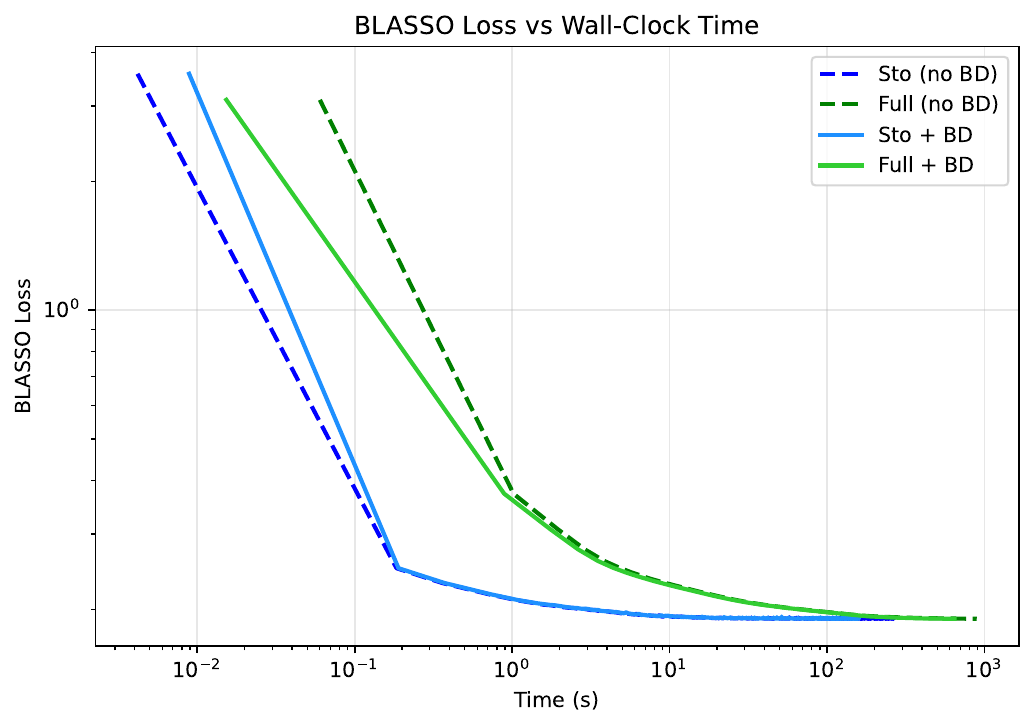}
    \caption{BLASSO loss evaluated on the GMM (left) and Neural Network (right) regression tasks.}
    \label{fig:loss_vs_time}
\end{figure}

As observed in Figure~\ref{fig:loss_vs_time} (left panel--GMM), the standard algorithms quickly plateau into local minima.
The BD mechanisms, however, periodically inject new particles in regions where the dual certificate is highly negative, allowing the loss to drop further. Also, the convergence is not altered by the death process (right panel--NN): the solution gets sparser and sparser (with fewer and fewer neurons) and achieves the same performance as the large neural network with $300$ hidden neurons.

\subsubsection{Spatial Distribution of Particles in GMM}

Finally, we visualize the end-state positions of the particles for the 2D Gaussian Mixture Model experiment in Figure~\ref{fig:gmm_positions}.
The baseline models (top row) suffer from the presence of spurious particles that fail to align with the true means, keeping the meaningful particle count significantly below $p=20$ (the number of initial particles).

\begin{figure}[htpb]
    \centering
    \includegraphics[width=0.49\linewidth]{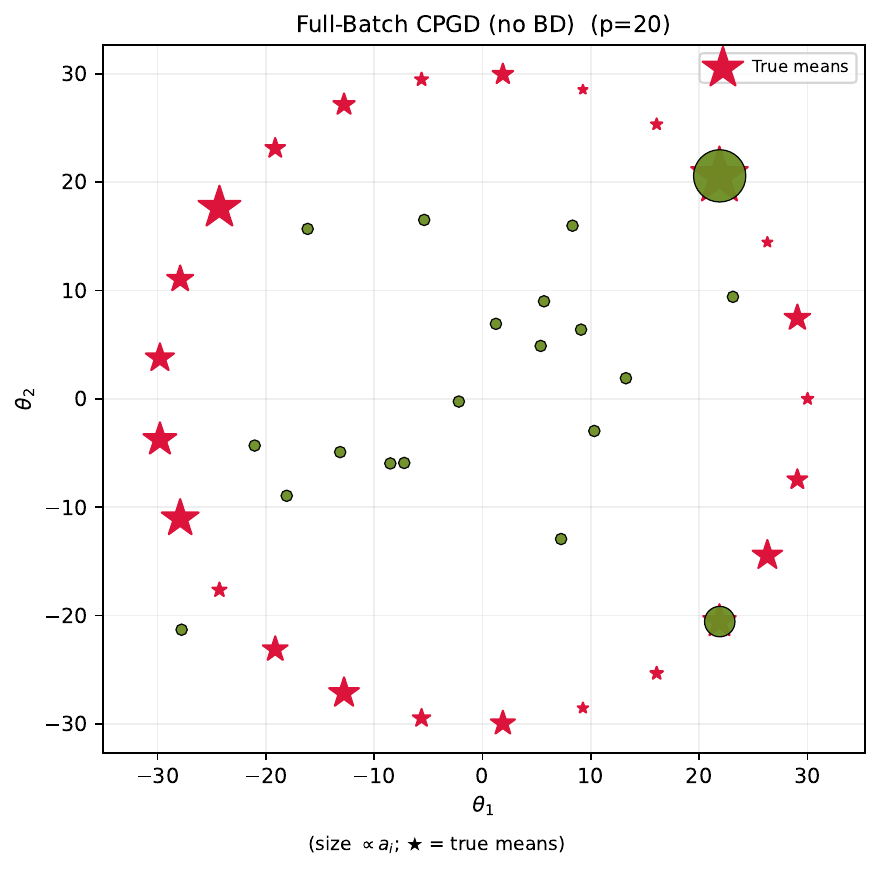}
    \includegraphics[width=0.49\linewidth]{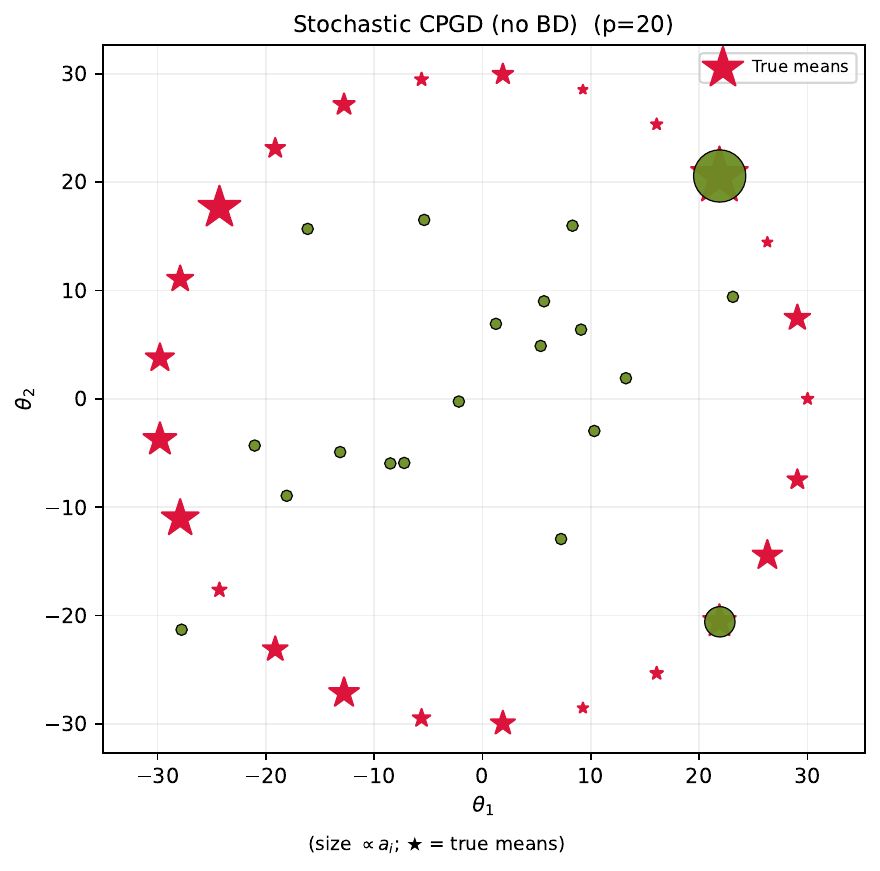}
    \includegraphics[width=0.49\linewidth]{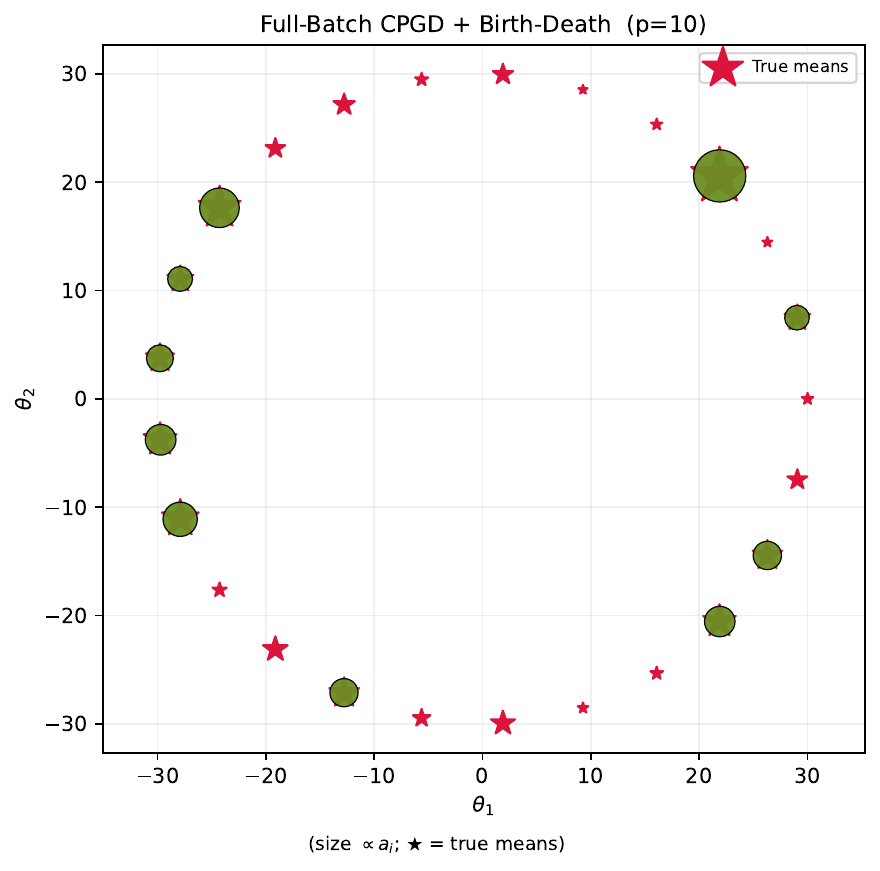}
    \includegraphics[width=0.49\linewidth]{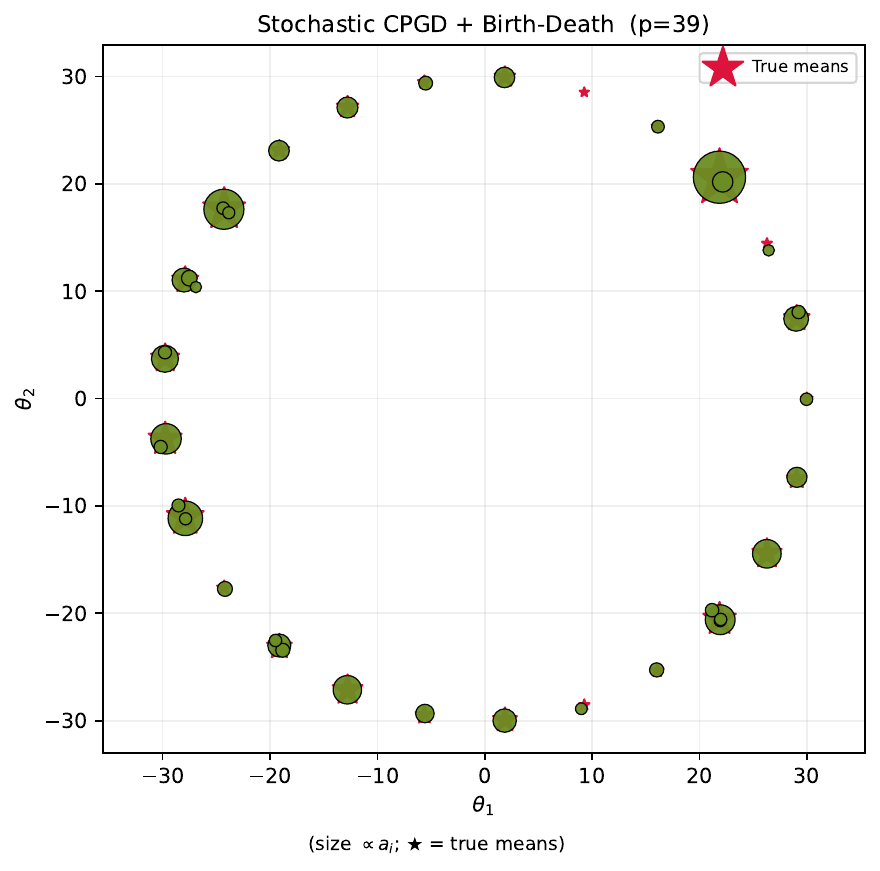}
    \caption{Final positions of the particles $t=(t_1, t_2)$ relative to the true cluster means (represented by cross marks) in the GMM experiment.
    Top row: baseline CPGD without BD. Bottom row: CPGD with BD.
    The size and color of the markers correspond to the weights $\omega_i$ of the particles.
    A common random seed has been fixed for shared initialization and batch sampling.}
    \label{fig:gmm_positions}
\end{figure}

\medskip

In contrast, the Full-Batch model equipped with the BD process (bottom-left) recovers the true means and cleanly prunes all unnecessary components, albeit yielding an under-representation with $p=10$ particles.
The stochastic counterpart (bottom-right) successfully maps out the target distribution and maintains a matching set of particles (all targets are identified except the smallest one, which has a weight of $0.0005$).

\newpage

\begin{table}[htpb]
    \centering
    \caption{Summary of shared hyperparameters across the two main experiments.}
    \label{tab:hyperparameters}
    \begin{tabular}{lcc}
        \toprule
        \textbf{Parameter} & \textbf{GMM (Fixed Covariance)} & \textbf{Neural Net. (California)} \\
        \midrule
        Train Samples ($n$) & $24,000$ & $18,576$ \\
        Input Dimension ($d$) & 2 & 8 \\
        Initial Particles ($p$) & 20 & 300 \\
        TV Regularization ($\kappa$) & $0.0001$ & $0.0005$ \\
        Full-Batch Size & $24,000$ & $18,576$ \\
        Full-Batch Iterations & $50,000$ & $100,000$ \\
        Stochastic Batch Size & 256 & 256 \\
        Stochastic Iterations & $200,000$ & $750,000$ \\
        Death Threshold ($\tau_{\text{death}}$) & 5.0 & 5.0 \\
        \bottomrule
    \end{tabular}
\end{table}

\subsubsection{Performance Analysis and Experimental Parameters}

To formalize the empirical observations, we report the exact hyper-parameters used for both the GMM and NN regression tasks in Table~\ref{tab:hyperparameters}, and the quantitative final results in Tables~\ref{tab:gmm_results} and \ref{tab:nn_results}.

\begin{table}[htpb]
    \centering
    \caption{Final results for the 2D GMM Experiment. The BD variations significantly lower the final loss and yield a better TV norm (true is $1$) while maintaining or reducing computation time.}
    \label{tab:gmm_results}
    \begin{tabular}{lcccccc}
        \toprule
        \textbf{Method} & \textbf{Loss} &  \textbf{TV} & $\mathbf{p_{\text{final}}}$ & \textbf{Time (s)} & \textbf{Deaths} & \textbf{Births} \\
        \midrule
        Full-Batch & 0.001869 & 0.2760 & 20 & 83.67 & 0 & 0 \\
        Stochastic & 0.001872 & 0.2863 & 20 & 55.02 & 0 & 0 \\
        Full-Batch + BD & 0.000434 & 0.7788 & 10 & 61.88 & 19 & 9 \\
        Stochastic + BD & \textbf{0.000259} &  \textbf{0.9786} & 39 & 65.23 & 78 & 97 \\
        \bottomrule
    \end{tabular}
\end{table}

In the GMM experiment, the dataset consists of $n=24,000$ training samples generated from $K=25$ true components.
The optimization regularizes the total variation (TV) norm with $\kappa = 0.0001$, starting from an initial set of $p=20$ particles.
As shown in Table~\ref{tab:gmm_results}, integrating the Birth-Death (BD) process yields a drastic improvement in the final BLASSO loss.
For example, the Full-Batch variant sees its loss drop from $0.001869$ to $0.000434$ while trimming the final number of particles down to $p=10$.
The Stochastic+BD method provides the lowest overall loss ($0.000259$) and proxy test error, managing a highly dynamic capacity (78 deaths and 97 births) to settle at $p=39$ particles.

\begin{table}[htpb]
    \centering
    \caption{Final results for the Two-Layer Neural Network regression on the California Housing dataset. The BD process yields highly parsimonious networks and faster training times without sacrificing Test MSE.}
    \label{tab:nn_results}
    {
    \begin{tabular}{lccccccc}
        \toprule
        \textbf{Method} %
        & \textbf{MSE} & \textbf{MSE (test)} & $\mathbf{p_{\text{final}}}$ & \textbf{Time (s)} & \textbf{Deaths} & \textbf{Births} \\
        \midrule
        Full-Batch %
        & 0.365476 & 0.393433 & 300 & 892.2 & 0 & 0 \\
        Stochastic %
        & 0.366231 & 0.393154 & 300 & 265.1 & 0 & 0 \\
        Full-Batch + BD %
        & \textbf{0.365506} & 0.393495 & 189 & 644.0 & 111 & 0 \\
        Stochastic + BD %
        & 0.366369 & \textbf{0.392429} & \textbf{60} & \textbf{158.9} & 269 & 29 \\
        \bottomrule
    \end{tabular}
    }
\end{table}

\newpage

For the California Housing NN regression, the model starts over-parameterized with $p=300$ initial particles (neurons) to learn from $n=18,576$ training samples with $d=8$ input features.
The regularization is set to $\kappa = 0.0005$. Table~\ref{tab:nn_results} highlights that the Stochastic+BD algorithm prunes the network down to just $p=60$ neurons (via 269 death events and 29 births). This significant reduction in parameters translates to a $40\%$ decrease in wall-clock training time (from $265.1$s to $158.9$s) compared to the standard Stochastic CPGD, all while achieving a slightly better test MSE of $0.392429$.
The Full-Batch+BD method similarly reduces both the particle count ($p=189$) and the training time ($644.0$s compared to $892.2$s).
In both experimental settings, the death threshold was consistently set to $\tau_{\text{death}} = 5.0$.

\newpage

\vskip 0.2in
\bibliography{biblio}

\newpage

\appendix

\section{Technical Lemmas}

\subsection{Existence of the Feature Map}
\begin{lem}
    \label{lem:dual_Phi}
    Let $\Phi\,:\,\mathcal{M}(\mathcal{X})\to\mathbb{H}$ be a \emph{bounded linear} and weak-* continuous operator, then its dual operator $\Phi^\star\,:\,\mathbb{H}\to\mathcal{C}(\mathcal{X})$ reads
    \[
    \Phi^\star\,:\,h\in\mathbb{H}\longmapsto \big(t\in\mathcal{X}\longmapsto\langle\varphi_t,h\rangle_{\mathcal H}\big)\in\mathcal{C}(\mathcal{X})\,,
    \]
    where we identified the pre-dual space $\mathcal{C}(\mathcal{X})$ as a subspace of the dual $\mathcal{M}(\mathcal{X})^\star$, and where $\varphi_t:=\Phi\, \delta_t$ with~$\delta_t$ the Dirac measure at $t\in\mathcal{X}$. 
    
    Moreover, for every $\mu\in\mathcal M(\mathcal X)$, one has the Bochner integral representation in $\mathbb{H}$:
    \[
        \Phi\mu \,=\, \int_{\mathcal X} \varphi_t\,\mathrm d\mu(t).
    \]
\end{lem}

\begin{proof}
    Fix $h\in\mathbb{H}$ and consider the linear functional $L_h:\mathcal{M}(\mathcal{X})\to\mathbb R$ defined by $L_h(\mu):=\bracket{h,\Phi\mu}_\mathbb{H}$. Since~$\Phi$ is weak-* continuous and $h\longmapsto \bracket{h,\cdot}_\mathbb{H}$ is continuous on $\mathbb{H}$, the map $L_h$ is linear and weak-* continuous on $\mathcal{M}(\mathcal{X})$. By definition of the weak-* topology on a dual space $E^\star$ (here $E^\star = \mathcal{M}(\mathcal{X})$), the continuous linear functionals on $(E^\star, \text{weak-*})$ are precisely the evaluations by elements of the pre-dual space $E$ (here $E = \mathcal{C}(\mathcal{X})$). Formally, $(E^\star, \sigma(E^\star, E))^\star \cong E$. Hence, there exists a unique $f_h\in\mathcal{C}(\mathcal{X})$ such that
    \[
        \bracket{h,\Phi\mu}_\mathbb{H}\,=\,\bracket{f_h,\mu}_{\mathcal{C}(\mathcal{X}),\mathcal{M}(\mathcal{X})}\quad\text{for all }\mu\in\mathcal{M}(\mathcal{X}).
    \]
    Define $\Phi^\star h:=f_h\in\mathcal{C}(\mathcal{X})$. Then $\Phi^\star:\mathbb{H}\to\mathcal{C}(\mathcal{X})$ is linear and bounded, with
    \[
        \|\Phi^\star h\|_\infty\,=\,\sup_{\|\mu\|_{\mathrm{TV}}\le 1}\big|\bracket{\Phi^\star h,\mu}\big|\,=\,\sup_{\|\mu\|_{\mathrm{TV}}\le 1}\big|\bracket{h,\Phi\mu}_\mathbb{H}\big|\,\le\,\|h\|_\mathbb{H}\,\|\Phi\|.
    \]
    For $t\in\mathcal{X}$, let $\delta_t$ be the Dirac measure at $t$ and set $\varphi_t:=\Phi\delta_t\in\mathbb{H}$. Evaluating the identity at $\mu=\delta_t$ yields
    \begin{equation}
\notag
                (\Phi^\star h)(t)\,=\,\bracket{\Phi^\star h,\delta_t}\,=\,\bracket{h,\Phi\delta_t}_\mathbb{H}\,=\,\bracket{\varphi_t,h}_\mathbb{H}.
    \end{equation}
    Hence $\Phi^\star h$ is precisely the continuous function $t\longmapsto \bracket{\varphi_t,h}_\mathbb{H}$, as claimed.

    We now prove the integral representation. First, note that 
    \[
    \|\varphi_t\|_\mathbb{H}=\|\Phi\delta_t\|_\mathbb{H}\le \|\Phi\|\,\|\delta_t\|_{\mathrm{TV}}=\|\Phi\|
    \]
    for all $t\in\mathcal{X}$, so $t\longmapsto\varphi_t$ is bounded. Next, for each fixed $h\in\mathbb{H}$, we already identified $\Phi^\star h\in\mathcal{C}(\mathcal{X})$ and established $(\Phi^\star h)(t)=\bracket{\varphi_t,h}_\mathbb{H}$. Since $\Phi^\star h\in\mathcal{C}(\mathcal{X})$, the scalar map $t\longmapsto(\Phi^\star h)(t)=\bracket{\varphi_t,h}_\mathbb{H}$ is continuous on~$\mathcal{X}$. Hence $t\longmapsto\varphi_t$ is weakly continuous (i.e., all scalar evaluations $\bracket{\varphi_t,h}_\mathbb{H}$ are continuous in~$t$). Because $\mathbb{H}$ is separable, weak measurability/continuity implies strong (Bochner) measurability by Pettis' theorem, see~\cite[Theorem 1.1.6]{hytonen2016analysis}. Consequently, for any finite signed measure $\mu\in\mathcal{M}(\mathcal{X})$, the Bochner integral $\int_{\mathcal{X}}\varphi_t\,\mathrm d\mu(t)$ is well-defined in $\mathbb{H}$ and satisfies, for all $h\in\mathbb{H}$,
    \begin{equation}
        \label{eq:Fubini_Hilbert}
                \Big\langle h, \int_{\mathcal{X}}\varphi_t\,\mathrm d\mu(t) \Big\rangle_\mathbb{H} \,=\, \int_{\mathcal{X}} \bracket{h,\varphi_t}_\mathbb{H}\,\mathrm d\mu(t).
    \end{equation}
Using the identity $\bracket{h,\Phi\mu}_\mathbb{H}=\bracket{\Phi^\star h,\mu}=\int_{\mathcal{X}} (\Phi^\star h)(t)\,\mathrm d\mu(t)=\int_{\mathcal{X}}\bracket{h,\varphi_t}_\mathbb{H}\,\mathrm d\mu(t)$, we deduce that 
    \[
        \bracket{h,\Phi\mu}_\mathbb{H}=\big\langle h, \int_{\mathcal{X}}\varphi_t\,\mathrm d\mu(t) \big\rangle_\mathbb{H}
    \] 
for all $h\in\mathbb{H}$. By uniqueness of the Riesz representation in $\mathbb{H}$, this implies
    \[
        \Phi\mu \,=\, \int_{\mathcal{X}} \varphi_t\,\mathrm d\mu(t),
    \]
    which is the desired representation.
\end{proof}

\subsection{Lipschitz Continuity of the Feature Map}

\begin{lem}
\label{lem:lipschitz_feature_map}
Assume that the kernel $K$ satisfies Assumption ($\mathcal{H}_{\mathcal{P}}$), specifically that $K(t,t)=1$ for all $t \in \cX$ and that its second derivatives are bounded by $\CC_{\mathcal{P}}$. Then, the kernel metric $d_K$ satisfies the following Lipschitz inequality:
\begin{equation}
    d_K(s,t) := \|\varphi_t - \varphi_s\|_{\bbH} \leq \sqrt{\CC_{\mathcal{P}}} \|t-s\|\,,
\end{equation}
where $\|\cdot\|_{\bbH}$ denotes the norm in the Hilbert space $\bbH$ and $\|\cdot\|$ is the Euclidean norm.
\end{lem}

\begin{proof}
By the definition of the kernel metric $d_K$ and the relation $K(s,t) = \langle \varphi_s, \varphi_t \rangle_{\bbH}$, we have:
\begin{align*}
    d_K(s,t)^2 &= \|\varphi_t - \varphi_s\|_{\bbH}^2 \\
    &= \langle \varphi_t - \varphi_s, \varphi_t - \varphi_s \rangle_{\bbH} \\
    &= \langle \varphi_t, \varphi_t \rangle_{\bbH} - 2\langle \varphi_t, \varphi_s \rangle_{\bbH} + \langle \varphi_s, \varphi_s \rangle_{\bbH} \\
    &= K(t,t) - 2K(s,t) + K(s,s)\,.
\end{align*}
Using the normalization property $K(u,u)=1$ for all $u \in \cX$ from Assumption ($\mathcal{H}_{\mathcal{P}}$), this simplifies to:
\begin{equation}
\label{eq:dist_kernel_identity}
    d_K(s,t)^2 = 2(1 - K(s,t))\,.
\end{equation}
Consider the function $g(s) := K(t,s)$. From the normalization and the Cauchy-Schwarz inequality, we know that $K(s,t) \leq \sqrt{K(s,s)K(t,t)} = 1$. Thus, $g(s)$ achieves its global maximum at $s=t$, implying that the gradient vanishes at this point: $\nabla_s K(t,s)|_{s=t} = 0$.

Applying a second-order Taylor expansion of $K(t,s)$ with respect to $s$ around $t$, there exists $\xi$ on the segment $[s,t]$ such that:
\[
    K(t,s) = K(t,t) + \langle \nabla_s K(t,t), s-t \rangle + \frac{1}{2} (s-t)^\top \nabla^2_s K(t,\xi) (s-t)\,.
\]
Substituting $K(t,t)=1$ and $\nabla_s K(t,t)=0$:
\[
    1 - K(t,s) = - \frac{1}{2} (s-t)^\top \nabla^2_s K(t,\xi) (s-t)\,.
\]
Substituting this back into \eqref{eq:dist_kernel_identity}:
\[
    d_K(s,t)^2 = - (s-t)^\top \nabla^2_s K(t,\xi) (s-t) = |(s-t)^\top \nabla^2_s K(t,\xi) (s-t)|\,.
\]
By Assumption ($\mathcal{H}_{\mathcal{P}}$), the Hessian is bounded, i.e., $\|\nabla^2_s K(\cdot,\cdot)\|_\infty \leq \CC_{\mathcal{P}}$. Therefore:
\[
    d_K(s,t)^2 \leq \CC_{\mathcal{P}} \|s-t\|^2 \implies d_K(s,t) \leq \sqrt{\CC_{\mathcal{P}}} \|s-t\|\,.
\]
\end{proof}

\paragraph{Comment on the Kernel Metric $d_K$:} The kernel metric $d_K(s,t) := \|\varphi_t - \varphi_s\|_{\bbH}$ measures the Hilbertian distance between data points after they have been mapped into the high-dimensional feature space $\bbH$. It defines the ``pullback'' geometry of the feature space onto the input space $\cX$. The lemma shows that for smooth kernels (specifically $C^2$ kernels like the Gaussian kernel), this map is Lipschitz continuous with respect to the Euclidean distance on $\cX$. This ensures that points close in the input space $\cX$ remain close in the feature space $\bbH$, a critical property for the complexity (with respect to the dimension) of the particle gradient descent algorithms discussed in the paper.

\subsection{Fréchet derivatives}

We consider the objective function $J(\nu)$ defined as the sum of a data-fitting term $R(\nu)$ and the total variation norm, i.e., $J(\nu) = R(\nu) + \kappa\|\nu\|_{\mathrm{TV}}$. The following lemmas establish the Fréchet derivatives of these components.

\begin{lem}[Derivative of the Risk Term]
\label{lem:1.2a}
Let $R(\nu) = \frac{1}{2} \| \int_{\cX} \varphi_{x} \mathrm{d}\nu(x) - y \|_{\mathbb{H}}^2$ be the risk functional defined on the space of measures $\mathcal{M}(\cX)$, where $x \longmapsto \varphi_x$ is the feature map into a Hilbert space $\mathbb{H}$ and $y \in \mathbb{H}$ is the target. Then, the Fréchet derivative of $R$ at $\nu$, denoted by $R'(\nu)$, is the function on $\cX$ given by:
\begin{equation}
\label{eq:1.2a}
\forall t \in \cX\,,\qquad
R'(\nu)(t) = \left\langle \int_{\cX} \varphi_{x} \mathrm{d}\nu(x) - y, \varphi_t \right\rangle_{\mathbb{H}}.
\end{equation}
\end{lem}

\begin{proof}
Let $\nu \in \mathcal{M}(\cX)$ and consider a perturbation $\sigma \in \mathcal{M}(\cX)$. We expand the term $R(\nu + \sigma)$:
\begin{align*}
R(\nu + \sigma) &= \frac{1}{2} \left\| \int_{\cX} \varphi_{x} d(\nu + \sigma)(x) - y \right\|_{\mathbb{H}}^2 \\
&= \frac{1}{2} \left\| \left(\int_{\cX} \varphi_{x} \mathrm{d}\nu(x) - y\right) + \int_{\cX} \varphi_{x} d\sigma(x) \right\|_{\mathbb{H}}^2 \\
&= \frac{1}{2} \left\| \int_{\cX} \varphi_{x} \mathrm{d}\nu(x) - y \right\|_{\mathbb{H}}^2 + \left\langle \int_{\cX} \varphi_{x} \mathrm{d}\nu(x) - y, \int_{\cX} \varphi_{x} d\sigma(x) \right\rangle_{\mathbb{H}} 
+ \frac{1}{2} \left\| \int_{\cX} \varphi_{x} d\sigma(x) \right\|_{\mathbb{H}}^2.
\end{align*}
The first term is $R(\nu)$. The third term is of order $O(\|\sigma\|_{\mathrm{TV}}^2)$. The second term is the linear part in $\sigma$. Using~\eqref{eq:Fubini_Hilbert}, we can rewrite the inner product as:
$$
\left\langle \int_{\cX} \varphi_{x} \mathrm{d}\nu(x) - y, \int_{\cX} \varphi_{x} d\sigma(x) \right\rangle_{\mathbb{H}} = \int_{\cX} \left\langle \int_{\cX} \varphi_z \mathrm{d}\nu(z) - y, \varphi_{x} \right\rangle_{\mathbb{H}} d\sigma(x).
$$
This identifies the Fréchet derivative $R'(\nu)$ as the function $t \longmapsto \langle \int_{\cX} \varphi_{x} \mathrm{d}\nu(x) - y, \varphi_t \rangle_{\mathbb{H}}$, proving \eqref{eq:1.2a}.
\end{proof}

\begin{lem}[Derivative of the Regularization Term]
\label{lem:1.2b}
Consider the regularization term $H(\nu) = \kappa\nu(\cX)$ for non-negative measures $\nu \in \mathcal{M}_+(\cX)$ (which corresponds to the TV norm for non-negative measures). Its Fréchet derivative is constant:
\begin{equation}
\label{eq:1.2b}
\forall t \in \cX\,,\qquad
H'(\nu)(t) = \kappa\,.
\end{equation}
\end{lem}

\begin{proof}
Let $\nu \in \mathcal{M}_+(\cX)$ and let $\sigma$ be a perturbation such that $\nu + \sigma \in \mathcal{M}_+(\cX)$. The functional $H$ is linear:
$$
H(\nu + \sigma) = \kappa(\nu(\cX) + \sigma(\cX)) = \kappa\nu(\cX) + \kappa\sigma(\cX).
$$
We can write $\kappa\sigma(\cX)$ as the integral against the constant function $\kappa$:
$$
\kappa\sigma(\cX) = \int_{\cX} \kappa\, d\sigma(x).
$$
Thus, the linear variation is represented by the constant function $t \longmapsto \kappa$. Therefore, the Fréchet derivative is $H'(\nu)(t) = \kappa$ for all $t \in \cX$.
\end{proof}

\subsection{Symmetrization trick}
\label{sec:symmetrization}

Following \citet{chizat2022sparse}, we address the optimization problem over the space of signed measures $\mathcal{M}(\mathcal{X})$ by lifting it to the space of non-negative measures on an augmented domain. We introduce the extended space $\tilde{\mathcal{X}} := \mathcal{X} \times \{-1, +1\}$ and associate to any signed measure $\mu \in \mathcal{M}(\mathcal{X})$ a non-negative measure $\nu \in \mathcal{M}_+(\tilde{\mathcal{X}})$. The correspondence is established through the linear map $P: \mathcal{M}_+(\tilde{\mathcal{X}}) \to \mathcal{M}(\mathcal{X})$ defined by:
\begin{equation}
\notag
    \mu = P(\nu) := \nu(\cdot, +1) - \nu(\cdot, -1).
\end{equation}
In the context of the linear model where observations are given by $\int_{\mathcal{X}} \varphi_x \mathrm{d}\mu(x)$, we define the augmented feature map $\tilde{\varphi}: \tilde{\mathcal{X}} \to \mathbb{H}$ as $\tilde{\varphi}(x, s) := s \varphi(x)$ for any $(x,s) \in \tilde{\mathcal{X}}$. Consequently, the linear measurements satisfy:
\begin{equation}
\notag
    \int_{\mathcal{X}} \varphi(x) d\mu(x) = \int_{\tilde{\mathcal{X}}} \tilde{\varphi}(\tilde{x}) d\nu(\tilde{x}).
\end{equation}
Furthermore, the total variation norm satisfies $\|\mu\|_{\text{TV}} \leq \nu(\tilde{\mathcal{X}})$, with equality holding if and only if the positive and negative parts of $\mu$ have disjoint supports (which is verified for optimal solutions of sparse problems). This symmetrization allows us to solve the signed problem by applying the conic particle gradient descent algorithm to the non-negative measure $\nu$ on the space $\tilde{\mathcal{X}}$.

\subsection{Lipschitz Continuity of the Fréchet derivative}

By smoothness of the kernel~\eqref{eq:upper_bound_kernel_gradient_hessian}, the Fréchet derivative $J'_\nu$ is twice continuously differentiable for any $\nu$. Define for any twice continuously differentiable $\psi$:
\begin{equation}
   \|\psi\|_{\mathcal C^2(\cX)}:=\max\{\|\psi\|_\infty,\|\nabla\psi\|_\infty,\|\nabla^2\psi\|_\infty\}\,.
   \label{eq:normC2X}
\end{equation}
The next lemma gives an upper bound on $\|J'_\nu\|_{\mathcal C^2(\cX)}$. We already know from~\eqref{eq:lower_bound_DC} that $\|\nu\|_{\mathrm{TV}}+\CC_{\mathcal{P}}+\kappa\ge\|J_\nu'({t}) \|_\infty$.
\begin{lem}
\label{lem:lipschitz_J_prime}
Assume that Assumption $(\mathcal{H}_{\mathcal{P}})$ holds. For any measure $\nu \in \mathcal{M}(\mathcal{X})$, the function $t \longmapsto J'_\nu(t)$ is Lipschitz continuous with constant:
$$
\mathfrak{L}(\nu) =\|\nabla J'_\nu\|_\infty\leq {\sqrt{\CC_{\mathcal P}}} \,(\CC_{\mathcal P}+\|\nu\|_{\mathrm{TV}}).
$$
and the dual certificate is gradient-Lipschitz with constant:
\[
\|\nabla^2 J'_\nu\|_\infty\leq \CC_{\mathcal P}\bigl(\|\nu\|_{\mathrm{TV}}+\|y\|_{\mathbb{H}}\bigr),
\]
so that the full $\mathcal{C}^2$ norm satisfies:
\[
\|J'_\nu\|_{\mathcal C^2(\cX)}\leq \CC_{\mathcal P}\bigl(\|\nu\|_{\mathrm{TV}}+\CC_{\mathcal P}\bigr)+\kappa.
\]
\end{lem}

\begin{proof}
Throughout the proof we use the expression of the Fréchet derivative~\eqref{eq:Frechet_J}:
$$
J'_\nu(t) = \langle \Phi \nu - y,\, \varphi_t \rangle_{\mathbb{H}} + \kappa,
$$
and the common bound:%
\begin{equation}\label{eq:Phinu_minus_y_bound}
\|\Phi\nu - y\|_\mathbb{H} \leq \|\nu\|_{\mathrm{TV}} + \|y\|_\mathbb{H}.
\end{equation}

\medskip

\noindent
$\bullet$
For any $s,t\in\mathcal{X}$:
\begin{align*}
|J'_\nu(t)-J'_\nu(s)| &= |\langle \Phi\nu-y,\,\varphi_t-\varphi_s\rangle_\mathbb{H}|
\leq \|\Phi\nu-y\|_\mathbb{H}\,\|\varphi_t-\varphi_s\|_\mathbb{H}.
\end{align*}
By Lemma~\ref{lem:lipschitz_feature_map}, $\|\varphi_t-\varphi_s\|_\mathbb{H}\leq\sqrt{\CC_{\mathcal P}}\|t-s\|$.
Combined with \eqref{eq:Phinu_minus_y_bound}:
$$
|J'_\nu(t)-J'_\nu(s)| \leq \sqrt{\CC_{\mathcal P}}\,(\|\nu\|_{\mathrm{TV}}+\|y\|_\mathbb{H})\,\|t-s\|,
$$
establishing $\mathfrak{L}(\nu)=\|\nabla J'_\nu\|_\infty\leq\sqrt{\CC_{\mathcal P}}\,(\|\nu\|_{\mathrm{TV}}+\|y\|_\mathbb{H})$.

\medskip

\noindent
$\bullet$
Differentiating $J'_\nu(t)=\langle\Phi\nu-y,\varphi_t\rangle_\mathbb{H}+\kappa$ twice under the inner product (justified by the $\mathcal{C}^2$ smoothness of $t\mapsto\varphi_t$ under Assumption \eqref{eqs:hyp_HP}) gives, for any $t\in\mathcal{X}$ and unit vectors $u,v\in\mathbb{R}^d$:
$$
u^\top \nabla^2_t J'_\nu(t)\, v
= \langle \Phi\nu - y,\; D^2_t\varphi_t[u,v]\rangle_\mathbb{H}.
$$
By Cauchy-Schwarz and \eqref{eq:Phinu_minus_y_bound}:
$$
|u^\top \nabla^2_t J'_\nu(t)\,v|
\;\leq\; \|\Phi\nu-y\|_\mathbb{H}\;\|D^2_t\varphi_t[u,v]\|_\mathbb{H}
\;\leq\; (\|\nu\|_{\mathrm{TV}}+\|y\|_\mathbb{H})\;\|D^2_t\varphi_t[u,v]\|_\mathbb{H}.
$$
It remains to bound $\|D^2_t\varphi_t[u,v]\|_\mathbb{H}$.
Since $K(s,t)=\langle\varphi_s,\varphi_t\rangle_\mathbb{H}$, differentiating twice with respect to $t$ gives:
$$
u^\top\nabla^2_t K(s,t)\,v = \langle\varphi_s,\,D^2_t\varphi_t[u,v]\rangle_\mathbb{H}.
$$
Since $\varphi_s=\Phi\delta_s$, the closed linear span $\overline{\mathrm{span}}\{\varphi_s:s\in\mathcal{X}\}\subset\mathbb{H}$ coincides with the closure of $\Phi(\mathcal{M}(\mathcal{X}))$. We assume throughout that this closure equals the whole of $\mathbb{H}$ (the standard non-degeneracy condition for the RKHS associated with $K$). Then any unit vector $h\in\mathbb{H}$ is the $\mathbb{H}$-limit of finite linear combinations $\sum_i\alpha_i\varphi_{s_i}$, and by the Hahn--Banach theorem the operator norm of a continuous linear functional on $\mathbb{H}$ equals its supremum over $\overline{\mathrm{span}}\{\varphi_s\}$. Combined with $\|\varphi_s\|_\mathbb{H}=1$, this yields
$$
\|D^2_t\varphi_t[u,v]\|_\mathbb{H}
= \sup_{\|h\|_\mathbb{H}=1}\langle h,D^2_t\varphi_t[u,v]\rangle_\mathbb{H}
\;=\; \sup_{s\in\mathcal{X}}\big|\langle\varphi_s,D^2_t\varphi_t[u,v]\rangle_\mathbb{H}\big|
\;=\; \sup_{s\in\mathcal{X}}|u^\top\nabla^2_t K(s,t)\,v|
\;\leq\; \|\nabla^2 K\|_\infty \;\leq\; \CC_{\mathcal P}.
$$
Hence $\|\nabla^2 J'_\nu\|_\infty\leq\CC_{\mathcal P}(\|\nu\|_{\mathrm{TV}}+\|y\|_\mathbb{H})$.

\medskip

\noindent
$\bullet$
Combining with the $\|\cdot\|_\infty$ bound from~\eqref{eq:lower_bound_DC} and using $\|y\|_\mathbb{H}\leq\CC_\mathcal{P}$:
$$
\|J'_\nu\|_{\mathcal C^2(\cX)}
= \max\bigl(\|J'_\nu\|_\infty,\|\nabla J'_\nu\|_\infty,\|\nabla^2 J'_\nu\|_\infty\bigr)
\leq \CC_{\mathcal P}\bigl(\|\nu\|_{\mathrm{TV}}+\CC_{\mathcal P}\bigr)+\kappa.%
$$
\end{proof}

\subsection{Generalized descent}

Some useful properties of this generalized projected gradient can be found for instance in \cite{Ghadimi-Lan-Zhang}. In particular Lemma 1 of \cite{Ghadimi-Lan-Zhang} may be stated as follows.
\begin{lem}[\cite{Ghadimi-Lan-Zhang}\label{lem:Ghadimi-Lan-Zhang}]
The projection operator $\pi_{\cX}(\vt,{v},\beta)$ satisfies the two properties:
\begin{subequations}
\begin{itemize}
    \item  Correlation of the projected gradient and gradient lower bound:   for any $\vt \in \cX$, ${v} \in \bbR^d$ and $\beta>0$:
\begin{equation}
\langle  {v}, \pi_{\cX}(\vt,{v},\beta) \rangle \ge \|\pi_{\cX}(\vt,{v},\beta)\|^2\,.
\label{eq:Ghadimi1}
\end{equation}
\item $1$-Lipschitz inequality for projection:
 For any $\vt \in \cX$, $(v_1,v_2) \in \bbR^d$ and $\beta>0$:
\begin{equation}
\|\pi_{\cX}(\vt,v_1,\beta) -\pi_{\cX}(\vt,v_2,\beta) \|  \le \|v_1-v_2\|\,.
\label{eq:Ghadimi2}
\end{equation}
\end{itemize}
\end{subequations}
\end{lem} 

\section{Descent properties: technical results}
\label{app:tec}
The goal of this Appendix section is to provide some technical proofs of the key descent properties stated in Proposition \ref{prop:incre}.%

\begin{subequations}
\subsection{Proof of Proposition \ref{prop:incre}}
Our argument follows the same lines as in the proof of \cite[Lemma~2.5]{chizat2022sparse}.
We decompose $\nu^+$ as $\nu^+=\nu+(\tnu-\nu)+(\nu^+-\tnu)$, where $\tnu=\Tnua \nu$ and
$\nu^+=T^{\sharp}_{\nu,\beta} \tnu$.
Invoke~\eqref{eq:Frechet_Jprime} with $\sigma=\nu^+-\nu $ to get that
\begin{align*}
    \quad J(\nu^+)-J(\nu)&=\langle J'_\nu,\sigma\rangle + \frac12\|\Phi(\sigma)\|_\mathbb{H}^2\\
    &= \int_{\cX}J'_\nu(\mathrm{d}\nu^+-\mathrm{d}\tnu) + \int_{\cX}J'_\nu(\mathrm{d}\tnu-\mathrm{d}\nu) +\frac{1}{2}\left\|\int_{\cX}\varphi_\vt(\mathrm{d}\nu^+(\vt)-\mathrm{d}\nu(\vt))\right\|_\mathbb{H}^2 \\ 
    &\leq \underbrace{\int_{\cX}J'_\nu(\mathrm{d}\nu^+-\mathrm{d}\tnu)}_{:=A_1}+ \underbrace{\int_{\cX}J'_\nu(\mathrm{d}\tnu-\mathrm{d}\nu)}_{:=A_2}\\
    &\qquad + \underbrace{\left\|\int_{\cX}\varphi_\vt(\mathrm{d}\nu^+(\vt)-\mathrm{d}\tnu(\vt))\right\|_\mathbb{H}^2}_{:=B_1}
    + \underbrace{\left\|\int_{\cX}\varphi_\vt(\mathrm{d}\tnu(\vt)-\mathrm{d}\nu(\vt))\right\|_\mathbb{H}^2}_{:=B_2}
\end{align*}
In order to study the previous terms, we 
define $\delta=\alpha ( \|J'_\nu\|_{\mathcal C^2(\cX)} \vee 1)$.\\

\noindent
\underline{Study of $A_1$.} We use the proximal update defined in Equation \eqref{eq:gradient_generalise}:
\begin{align*}
    A_1 &= \int [J'_{\nu}(\vt-\beta \pi_{\cX}(\vt,\nabla J'_{\nu}(\vt),\beta)) - J'_{\nu}(\vt)] \text{d}\tnu(\vt)\\
    & \leq \int \left( - \beta \langle \pi_{\cX}(\vt,\nabla J'_{\nu}(\vt),\beta)),\nabla J'_{\nu}(\vt)\rangle + \frac{\beta^2 \|\pi_{\cX}(\vt,\nabla J'_{\nu}(\vt),\beta)\|^2 }{2} \|\nabla^2 J'_{\nu}\|_{\infty}\right)\text{d}\tnu(\vt).\\
\end{align*}
By Lemma~\ref{lem:lipschitz_J_prime},
$\|\nabla^2 J'_{\nu}\|_{\infty} \leq \CC_{\mathcal P}(\CTV+\CC_{\mathcal P})$.
By Lemma~\ref{lem:Ghadimi-Lan-Zhang} $ii)$ and Lemma~\ref{lem:lipschitz_J_prime},
$$\|\pi_{\cX}(\vt,\nabla J'_{\nu}(\vt),\beta)\|^2 \leq \|\nabla J'_{\nu}\|^2_{\infty} \leq \CC_{\mathcal P}(\CTV+\CC_{\mathcal P})^2.$$
Using $\int\|\pi_{\cX}\|^2\,\mathrm{d}\tnu\leq e^{\delta}\int\|\pi_{\cX}\|^2\,\mathrm{d}\nu$ (same ratio argument as \eqref{eq:descente-proj}), the Taylor remainder of $A_1$ satisfies:
\begin{equation}
\frac{\beta^2}{2}\|\nabla^2 J'_{\nu}\|_{\infty}\int_{\cX}\|\pi_{\cX}(\vt,\nabla J'_{\nu}(\vt),\beta)\|^2\,\mathrm{d}\tnu
\leq \frac{\beta}{2}\,\CC_{\mathcal P}(\CTV+\CC_{\mathcal P})\,e^{\delta}\,\|g_{\nu}^{\beta}\|^2_{L^2(\nu)}.
\label{eq:A1_remainder}
\end{equation}
We are led to study the first order term. Starting with \eqref{eq:Ghadimi1}, we get
\begin{align}
\int &- \beta \langle \pi_{\cX}(\vt,\nabla J'_{\nu}(\vt),\beta)),\nabla J'_{\nu}(\vt)\rangle \text{d}\tnu(\vt)\nonumber \\
& \leq - \beta \int \| \pi_{\cX}(\vt,\nabla J'_{\nu}(\vt),\beta))\|^2 \text{d}\tnu(\vt) \nonumber \\
&   \leq - \beta \int \| \pi_{\cX}(\vt,\nabla J'_{\nu}(\vt),\beta))\|^2 \text{d}\nu(\vt)
  + \beta \int \| \pi_{\cX}(\vt,\nabla J'_{\nu}(\vt),\beta))\|^2 \left|e^{-\alpha J'_{\nu}(\vt)}-1 \right|\text{d}\nu(\vt)\nonumber \\
  & \leq - \beta \int \| \pi_{\cX}(\vt,\nabla J'_{\nu}(\vt),\beta))\|^2 \text{d}\nu(\vt) + \alpha \beta 
  \int \| \pi_{\cX}(\vt,\nabla J'_{\nu}(\vt),\beta))\|^2  |J'_{\nu}(\vt)| e^{\alpha |J'_{\nu}(\vt)|} \text{d}\nu(\vt)\nonumber \\
  & \leq - \beta (1- \alpha (\kappa+\|\nu\|_{TV}+\|{y}\|_{\mathbb{H}})e^{\delta})\int \| \pi_{\cX}(\vt,\nabla J'_{\nu}(\vt),\beta))\|^2 \text{d}\nu(\vt),
  \label{eq:descente-proj}
\end{align}
where we have used \eqref{eq:Frechet_J} and rough upper bounds for the last inequality. \\

\noindent
\underline{Study of $A_2$.}
Observe that for any $\psi\in\mathcal C^2(\cX)$, one has
\begin{align}
\label{eq:decomposition-order-1}
    \int_{\cX}\psi(\mathrm{d}\tnu-\mathrm{d}\nu)
    &=\int_{\cX}\big(e^{-\alpha J'_\nu(\vt)}-1\big)\psi(\vt)\,\mathrm{d}\nu(\vt)
    \end{align}
We will use the standard inequality:
\begin{equation}
    \label{eq:expo}
    \forall u \in \mathbb{R} \qquad |e^{-u}-1+u| \leq \frac{u^2}{2} e^{|u|}
\end{equation}
and a first and second Taylor expansion on $\psi$. For any $(\vt,\vt+\bm{h})\in \mathcal{X}$:
\begin{equation}\label{eq:Taylor}
|\psi(\vt+\bm{h}) - \psi(\vt)| \leq \|\bm{h}\| \|\nabla \psi\|_{\infty} \quad \text{and} \quad
|\psi(\vt+\bm{h}) - \psi(\vt)-\langle \bm{h},\nabla \psi(\vt)\rangle| \leq \frac{\|\bm{h}\|^2}{2} \|\nabla^2\psi\|_{\infty}.
\end{equation}
We then use Equation \eqref{eq:expo} and Equation \eqref{eq:Taylor} and observe that
the first term of the right hand side of Equation \eqref{eq:decomposition-order-1} may be upper-bounded as follows:

\begin{eqnarray*}
    \big(e^{-\alpha J'_\nu(\vt)}-1\big)\psi(\vt)
    &=& - \alpha J'_\nu(\vt)\psi(\vt)+ \big(e^{-\alpha J'_\nu(\vt)}-1+\alpha J'_\nu(\vt) \big)\psi(\vt)\\
    & \leq &  - \alpha J'_\nu(\vt)\psi(\vt)  + \frac{\alpha^2 J'_\nu(\vt)^2}{2} e^{\alpha |J'_\nu(\vt)|} \|\psi\|_{\infty}.
\end{eqnarray*}
We then integrate with respect to $\text{d} \nu$ and get (while omitting the variable $\vt$ for the sake of readability):
\begin{align}
     &\int_{\cX}\psi(\mathrm{d}\tnu-\mathrm{d}\nu)  \leq - \int_{\cX}  \alpha J'_\nu\psi  
     \text{d} \nu 
     +  \underbrace{ \frac{\alpha^2 \|\psi\|_{\infty}}{2}  \int_{\cX} |J'_\nu|^2 e^{\alpha |J'_\nu|} \text{d} \nu }_{\text{Remainder}(\psi)} 
 \label{eq:decente-intermediaire}
\end{align}

\noindent
We are led to study the remainder term. Replacing $\psi$ by $J'_\nu$ in \eqref{eq:decente-intermediaire}, we get:
\begin{align*}
    \text{Remainder}(J'_\nu)   & \leq \frac{\alpha \|J'_\nu\|_{\infty} e^{\alpha \|J'_\nu\|_{\infty}}  }{2}
    \|g_{\nu}^{\alpha}\|^2_{L^2(\nu)}.
\end{align*}
Using $\delta=\alpha ( \|J'_\nu\|_{\mathcal C^2(\cX)} \vee 1)$, we then deduce the following bound: 
\begin{equation}
A_2= \int_{\cX}J'_\nu(\mathrm{d}\tnu-\mathrm{d}\nu) \leq - \|g_{\nu}^{\alpha}\|^2_{L^2(\nu)}
\left(1-\frac{\delta e^{\delta}}{2}\right).
\label{eq:appA}
\end{equation}
\noindent
\underline{Study of $B_1$.} To upper bound the second order terms, we use the $\sqrt{\CC_{\mathcal P}}$-Lipschitz continuity of $t\mapsto\varphi_t$ from Lemma~\ref{lem:lipschitz_feature_map}, so that $\|\varphi_t(a)-\varphi_t(b)\|_\mathbb{H}^2\leq\CC_{\mathcal P}\|a-b\|^2$:
\begin{align}
    B_1 & = \left\|\int_{\cX}\varphi_\vt(\mathrm{d}\nu^+(\vt)-\mathrm{d}\tnu(\vt))\right\|_\mathbb{H}^2 \nonumber\\
    & = \left\|\int_{\cX}\varphi_\vt(\vt-\beta \pi_{\cX}(\vt,\nabla J'_{\nu}(\vt),\beta)) -\varphi_\vt(\vt) \mathrm{d}\tnu(\vt)\right\|_\mathbb{H}^2 \nonumber\\
    & \leq \beta^2 \CC_{\mathcal P} \|\tnu\|_{TV}^2 \left(\int \| \pi_{\cX}(\vt,\nabla J'_{\nu}(\vt),\beta))\| \frac{\text{d}\tnu(\vt)}{\|\tnu\|_{TV}}\right)^2 \nonumber\\
    & \leq \beta^2 \CC_{\mathcal P} \|\tnu\|_{TV} \int \| \pi_{\cX}(\vt,\nabla J'_{\nu}(\vt),\beta))\|^2 \text{d}\tnu(\vt) \nonumber\\
    & \leq \beta^2 \CC_{\mathcal P} \CTV e^{2 \delta}\int \| \pi_{\cX}(\vt,\nabla J'_{\nu}(\vt),\beta))\|^2 \text{d}\nu(\vt) \label{eq:descente-proj2}
\end{align}
where we used Lemma \ref{lem:Ghadimi-Lan-Zhang} $ii)$ and the Jensen inequality on the normalized measure $\frac{\text{d}\tnu(\vt)}{\|\tnu\|_{TV}}$.%

\noindent
\underline{Study of $B_2$.} Since $K\leq 1$, for any signed measure $\mu=\mu^+-\mu^-$ one has $\|\Phi\mu\|_\mathbb{H}^2\leq(\mu^+(\cX)+\mu^-(\cX))^2=\|\mu\|_{\mathrm{TV}}^2$; hence:
\begin{align*}
\left\|\int_{\cX}\varphi_\vt(\mathrm{d}\tnu(\vt)-\mathrm{d}\nu(\vt))\right\|_\mathbb{H}^2
&=\int_\cX
\Big(\int_\cX
K(\vs,\vt)
\,(\mathrm{d}\tnu(\vs)-\mathrm{d}\nu(\vs))
\Big)
\,(\mathrm{d}\tnu(\vt)-\mathrm{d}\nu(\vt)) \\
&\leq \|\tnu-\nu\|_{\mathrm{TV}}^2\\
&=\Big(\int_{\cX}\big|e^{-\alpha J'_\nu(\vt)}-1\big|\,\mathrm{d}\nu(\vt)\Big)^2\\
&\leq \Big(\int_{\cX}
    \alpha |J'_\nu(\vt)| + \frac{\alpha^2 |J'_\nu(\vt)|^2}{2} e^{\alpha |J'_\nu(\vt)|}\,\mathrm{d}\nu(\vt)\Big)^2\\
&\leq 2 \| \nu\|_{\mathrm{TV}} \int_{\cX} \Big(\alpha^2 |J'_\nu(\vt)|^2 + \frac{\alpha^4 |J'_\nu(\vt)|^4}{4}e^{2 \alpha |J'_\nu(\vt)|} \Big)\mathrm{d}\nu(\vt),
\end{align*}
where we used in the last line the Cauchy-Schwarz inequality and $(u+v)^2 \leq 2 u^2+2v^2$.
Using again $\delta$ defined above, we deduce that:
\begin{equation}
B_2=\left\|\int_{\cX}\varphi_\vt(\mathrm{d}\tnu(\vt)-\mathrm{d}\nu(\vt))\right\|_\mathbb{H}^2 \leq 
2\alpha \CTV \|g_{\nu}^\alpha\|^2_{L^2(\nu)} \left( 1 +\frac{ \delta^2 e^{2 \delta}}{4}\right).
\label{eq:appB}
\end{equation}

Gathering \eqref{eq:descente-proj} with \eqref{eq:A1_remainder} and \eqref{eq:descente-proj2} on one side, and
\eqref{eq:appA} with \eqref{eq:appB} on the other side, we deduce the upper bound:
\begin{align*}
J(\nu^+)-J(\nu)& \leq  - \|g_{\nu}^\alpha\|^2_{L^2(\nu)} \left(1 -\frac{\delta e^{\delta}}{2} - 2\alpha \CTV \left( 1 +\frac{ \delta^2 e^{2 \delta}}{4}\right)\right) \\
& - \|g_{\nu}^\beta\|^2_{L^2(\nu)} \left( 1-\alpha (\kappa+\|\nu\|_{TV}+\|y\|_{\mathbb{H}})e^{\delta} - \beta \CC_{\mathcal P}\!\left(\CTV e^{2\delta}+\frac{(\CTV+\CC_{\mathcal P})\,e^{\delta}}{2}\right)\right).
\end{align*}
We are then led to choose
$$
\alpha < \frac{1}{10(1+ \nutv + \|{y}\|_{\mathbb{H}}+\kappa)(1\vee \CTV)},
$$
Such a constraint on $\alpha$ ensures that $ \delta e^{\delta} < \frac{1}{7}$, which occurs as soon as $\alpha$ is such that $\delta < \frac{1}{10}$. By Lemma~\ref{lem:lipschitz_J_prime}, since $\|\nu\|_\mathrm{TV}\leq\CTV$ and $\|y\|_\mathbb{H}\leq\CC_\mathcal{P}$:
\[
\|J'_\nu\|_{\mathcal C^2(\cX)}\leq \CC_{\mathcal P}(\CTV+\CC_{\mathcal P})+\kappa,
\]
so $\delta=\alpha(\|J'_\nu\|_{\mathcal C^2(\cX)}\vee 1)$ is small under the constraint on $\alpha$.
It is then straightforward to verify that:
$$
\delta e^{\delta} > \max\left(\alpha, \alpha \frac{\delta^2 e^{2 \delta}}{4}\right).$$
In the meantime, using $e^{\delta}\leq e^{1/5}$ and $e^{2\delta}\leq e^{1/5}$, the choice $\beta \leq \frac{1}{2\CC_{\mathcal P}(\CC_{\mathcal P}+3\CTV)\,e^{1/5}}$ yields
$$\beta\CC_{\mathcal P}\!\left(\CTV e^{2\delta}+\frac{(\CTV+\CC_{\mathcal P})\,e^{\delta}}{2}\right) \leq \frac{\beta\CC_{\mathcal P}\,e^{1/5}(3\CTV+\CC_{\mathcal P})}{2} \leq \frac{1}{4},$$
which finally entails:
\begin{equation}\label{eq:descente_finale}
J(\nu^+)-J(\nu) \leq  - \frac{3}{4} \left( \|g_{\nu}^\alpha\|^2_{L^2(\nu)}  +\|g_{\nu}^\beta\|^2_{L^2(\nu)} \right)
\end{equation}

\end{subequations}

\subsection{Linearization of the mirror descent strategy}

This section concerns conditional expectations on the mirror descent strategy.
Below, we state a general result on $\widehat{J'_{\nu}}(\vt) = \frac{1}{m} \sum_{l=1}^m
\widehat{J'_{\nu}}(\vt,Z_l)$ where $(Z_1,\ldots,Z_m)$ stands for a mini-batch sample of size $m$.   
\begin{proposition}\label{prop:hoeffding}
Assume that $\alpha \leq \sqrt{8\log 8}\, \MM^{-1}$, then:
\[
\forall \vt \in \cX: \qquad 
\left|\mathbb{E}\left[ e^{-\alpha \widehat{J'_{\nu}}(\vt,Z)} -(1-\alpha J'_{\nu}(\vt))\ \big|\ \nu  \right]\right| \leq \frac{\alpha^2 \MM^2}{m} e^{-\alpha J'_{\nu}(\vt)} + \frac{\alpha^2J'_{\nu}(\vt)^2}{2} e^{\alpha |J'_{\nu}(\vt)|}.
\]
    \end{proposition}

\begin{proof} Below, we compute the expectation with respect to the randomness brought by the mini-batch sample.
Let $\vt \in \mathcal{X}$. We begin with the following decomposition:
    \begin{align*}
    \lefteqn{\mathbb{E} \left[   \alpha  J'_{\nu}(\vt) + e^{-\alpha \widehat{J'_{\nu}}(\vt)} -  1   \right]}\\
    & = 
    \alpha  J'_{\nu}(\vt) - 1 + \mathbb{E}\left[e^{-\alpha \widehat{J'_{\nu}}(\vt)}  \right]  \\
    & = 
     \left[ \alpha  J'_{\nu}(\vt) - 1 + e^{-\alpha  J'_{\nu}(\vt)} \mathbb{E}\left[e^{-\alpha [\widehat{J'_{\nu}}(\vt)- J'_{\nu}(\vt)]}  \right]\right] \\
    & =   \left[ \alpha  J'_{\nu}(\vt ) - 1 + e^{-\alpha  J'_{\nu}(\vt)}\right]
     +   e^{-\alpha  J'_{\nu}(\vt)}
     \mathbb{E}\left[e^{-\alpha [\widehat{J'_{\nu}}(\vt)- J'_{\nu}(\vt)]} - 1  \right] \\
     & =  \left[ \alpha  J'_{\nu}(\vt ) - 1 + e^{-\alpha  J'_{\nu}(\vt)}\right] 
     +   e^{-\alpha  J'_{\nu}(\vt)}
     \mathbb{E}\left[e^{-\frac{\alpha}{m} \sum_{l=1}^m [\widehat{J'_{\nu}}(\vt,Z_l)- J'_{\nu}(\vt)]} - 1 
     \right]\\
     & =  \left[ \alpha  J'_{\nu}(\vt ) - 1 + e^{-\alpha  J'_{\nu}(\vt)}\right] 
     +   e^{-\alpha  J'_{\nu}(\vt)}\left(
     \left(\mathbb{E}\left[e^{-\frac{\alpha}{m} [\widehat{J'_{\nu}}(\vt,Z_l)- J'_{\nu}(\vt)]}\right]\right)^m - 1
     \right).
\end{align*}
To derive an upper bound, we use the inequality $|e^{-u}-1+u| \leq \frac{u^2}{2} e^{|u|}$ which holds for any $u\in \mathbb{R}$ for the first term and 
we apply the Hoeffding Lemma to the random variable $\widehat{J'_{\nu}}(\vt,Z)- J'_{\nu}(\vt)$, which is a centered random variable almost surely bounded by $\MM$. According to Assumption \eqref{A1}, we obtain that:
\[
\left|
\mathbb{E}\left[e^{-\alpha [\widehat{J'_{\nu}}(\vt,Z)- J'_{\nu}(\vt)]} - 1 \right] \right|\leq e^{\frac{\MM^2 \alpha^2}{8 m}} - 1 \leq \frac{\alpha^2 \MM^2}{8 m} e^{\frac{\alpha^2 \MM^2}{8 m}} \leq \frac{\alpha^2 \MM^2}{m} \\
\]
where thanks to our assumption on $\alpha$, we observe that $e^{\frac{\alpha^2 \MM^2}{8 m}} \leq 8$.
We finally obtain the desired upper bound.
\end{proof}

\section{Proofs of the deterministic results}
In this paragraph, we present all the proofs related to the deterministic results introduced at the beginning of our work. In particular, we establish the proof of the key results of Section \ref{s:birth-death-proc} 
and Section \ref{sec:conv_det}. 

\subsection{Total variation boundedness}
\label{s:proof_hypdist}

\begin{proof}[Proof of Proposition \ref{prop:hypdist}, $i)$]
We address \eqref{eq:HTVC}. %
Let $k\in \mathbb{N}$ be fixed. Our starting point is the relationship
\[
\nukpp  = \nukkp + \epk \bm{1}_{\mathcal{N}_{\nukp}} \lambda  
 = \left( 1 - \bm{1}_{\mathcal{P}_{\nukp}} \right) \nukp
+ \epk \bm{1}_{\mathcal{N}_{\nukp}} \lambda.
\]
Computing the total variation norm, we obtain:
\[
\| \nu_{k+1} \|_{\mathrm{TV}} \leq \|\nukp\|_{\mathrm{TV}} + \epk \lambda(\cX) \leq   \int_{\mathcal{X}} e^{-\alpha J_{\nu_k}'(\vt)} d\nu_k(\vt) + \max_{j\ge 0}\varepsilon_j\, \lambda(\mathcal{X}).
\]
According to Assumption \eqref{eqs:hyp_HP}, observe that for any $\vt\in \mathcal{X}$
\[
J_{\nu_{k}}'(\vt) = \int_\mathcal{X} \langle \varphi_\vt ,\varphi_\vs \rangle d\nu_k(\vs) - \langle {y},\varphi_\vt\rangle + \kappa \geq \cc_{\mathcal{P}} \| \nu_k\|_{\mathrm{TV}} - \| {y} \|_\mathbb{H} + \kappa \ge  - \| {y} \|_\mathbb{H} + \kappa  .
\]
We then deduce that thanks to the condition $\epk \leq \UU  \alpha$:
$$
\| \nu_{k+1} \|_{\mathrm{TV}}  \leq e^{-\alpha (\cc_{\mathcal{P}} \| \nuk \|_{\mathrm{TV}}- \| {y} \|_\mathbb{H}) }\| \nuk \|_{\mathrm{TV}} +\UU   \lambda(\mathcal{X}) \alpha
$$
We define $\mathfrak{M}$ as $\mathfrak{M} = \frac{2 }{\cc_{\mathcal{P}}} \| {y} \|_\mathbb{H} +  \sqrt{\frac{2 e \UU  \lambda(\cX)}{\cc_{\mathcal{P}}}}$ and we verify that
$$
u \ge \mathfrak{M} \Longrightarrow 
\cc_{\mathcal{P}} u - \| {y} \|_\mathbb{H} \ge \frac{\cc_{\mathcal{P}} u}{2}.
$$
We now consider the two different cases:
\begin{itemize}
    \item   If $\|\nuk\|_{\mathrm{TV}} \ge \mathfrak{M}$, then 
    \begin{align*}
\| \nu_{k+1} \|_{\mathrm{TV}} & \leq e^{-\alpha 
\frac{\cc_{\mathcal{P}}}{2} \|\nuk\|_{\mathrm{TV}}}   \| \nu_{k} \|_{\mathrm{TV}}   + \UU   \lambda(\mathcal{X}) \alpha \\
& = \| \nu_{k} \|_{\mathrm{TV}} - (1-e^{-\alpha 
\frac{\cc_{\mathcal{P}}}{2} \|\nuk\|_{\mathrm{TV}}} )  \| \nu_{k} \|_{\mathrm{TV}}   + \UU   \lambda(\mathcal{X}) \alpha \\
& = \| \nu_{k} \|_{\mathrm{TV}} -\frac{2}{\alpha \cc_{\mathcal{P}}} \varphi\left(\alpha 
\frac{\cc_{\mathcal{P}}}{2} \|\nuk\|_{\mathrm{TV}}\right)   + \UU   \lambda(\mathcal{X}) \alpha,
    \end{align*}
    where $\varphi$ is defined by $\varphi(t)=  t(1-e^{-t})$. We check that $\varphi$ is an increasing function
 so that when $\|\nuk\|_{\mathrm{TV}} \ge \mathfrak{M}$, then
    $$
    \| \nu_{k+1} \|_{\mathrm{TV}} \leq \|\nuk\|_{\mathrm{TV}} 
    - \frac{2}{\alpha \cc_{\mathcal{P}}} \varphi\left( \alpha 
\frac{\cc_{\mathcal{P}}}{2} \mathfrak{M} \right)
    + \UU   \lambda(\mathcal{X}) \alpha = \|\nuk\|_{\mathrm{TV}} 
     - \mathfrak{M} \left(1- e^{- \alpha 
\frac{\cc_{\mathcal{P}}}{2} \mathfrak{M}}\right) 
    + \UU   \lambda(\mathcal{X}) \alpha.
    $$
    Thanks to our condition on $\alpha$, we know that
    $\alpha \frac{\cc_{\mathcal{P}}}{2} \mathfrak{M} \leq 1,$ and the convex inequality
 when $t \in [0,1] \,: e ^{-t} \leq 1-t/e$ yields:   
      $$
    \| \nu_{k+1} \|_{\mathrm{TV}} \leq \|\nuk\|_{\mathrm{TV}} - \mathfrak{M} \times \frac{\alpha \frac{\cc_{\mathcal{P}}}{2} \mathfrak{M}}{e} + \UU   \lambda(\mathcal{X}) \alpha.
    $$
    Thanks to our definition of $\mathfrak{M}$, we then observe that in this case $
    \| \nu_{k+1} \|_{\mathrm{TV}} \leq \|\nuk\|_{\mathrm{TV}}$.
    \item   If $\|\nuk\|_{\mathrm{TV}} \le \mathfrak{M}$, then we use the straightforward upper bound:
    $$
     \| \nu_{k+1} \|_{\mathrm{TV}} \leq  \mathfrak{M} e^{\alpha \| {y} \|_\mathbb{H}} + \UU  \lambda(\cX) \alpha \leq  e \mathfrak{M} + \UU  \lambda(\cX) 
    $$
\end{itemize}
A direct induction argument then shows that
$$
\forall k \ge 0 \qquad \| \nu_{k} \|_{\mathrm{TV}} \leq \| \nu_{0} \|_{\mathrm{TV}} \vee\left( e \mathfrak{M}  +\UU  \lambda(\cX) \right).
$$
Recalling $\mathfrak{M} = \frac{2\|y\|_\mathbb{H}}{\cc_\mathcal{P}} + \sqrt{\frac{2e\lambda(\cX)}{\cc_\mathcal{P}}}$ and $\RR = \frac{e\|y\|_\mathbb{H}}{\cc_\mathcal{P}} + \sqrt{\frac{e^3\lambda(\cX)}{\cc_\mathcal{P}}} + \lambda(\cX)$, we have $e\mathfrak{M} + \lambda(\cX) = \frac{2e\|y\|_\mathbb{H}}{\cc_\mathcal{P}} + \sqrt{\frac{2e^3\lambda(\cX)}{\cc_\mathcal{P}}} + \lambda(\cX) \leq 2\RR$, since $\sqrt{2}\le 2$ and $\lambda(\cX)\le 2\lambda(\cX)$. The proof bound is thus at most the proposition's statement bound $\|\nu_0\|_\mathrm{TV}\vee 2\RR=:\CTV$, concluding the proof of the boundedness of the sequence $(\nuk)_{k \ge 0}$.
\end{proof}

\subsection{$\varepsilon$-smoothness evolution}
\label{s:proof_remark_perturbation}
In this paragraph, we now establish that our sequence $(\nuk)_{k \ge 0}$ verifies Assumption $\Heps$ given by Equation \eqref{eqs:H_eps_determinist}, \textit{i.e.} we establish the second part of  Proposition \ref{prop:hypdist}. We first prove the perturbation bound and the deletion-set inclusion stated without proof in Remark~\ref{rem:weight_update_perturbation}, on which the proof of Proposition~\ref{prop:hypdist}~$ii)$ relies.

\medskip

\begin{proof}[Proof of Remark~\ref{rem:weight_update_perturbation}]

\noindent\emph{Proof of~\eqref{eq:weight_update_perturbation}.} The Fréchet identity~\eqref{eq:Frechet_J}, with $\|\varphi_\vt\|_\mathbb{H}=1$, gives
\[
|J'_{\nukp}(\vt) - J'_{\nuk}(\vt)| \;\leq\; \int_\cX \big|e^{-\alpha J'_{\nuk}(s)} - 1\big|\, d\nuk(s).
\]
The uniform bound $\|J'_{\nuk}\|_\infty \leq \CC_\mathcal{P} + \CTV + \kappa$ (from~\eqref{eq:lower_bound_DC} and~\eqref{eq:HTVC}), the elementary inequality $|e^{-x}-1| \leq |x|\,e^{|x|}$, and $\|\nuk\|_{\mathrm{TV}}\leq\CTV$ now yield~\eqref{eq:weight_update_perturbation}.

\smallskip
\noindent\emph{Proof of~\eqref{eq:Pnukp_inclusion}.} For $\vt\in\mathcal{P}_{\nukp}$, definition~\eqref{def:nukpp_alt1} gives $J'_{\nukp}(\vt) > -2\alpha^{-1}\log\varepsilon_k + \CC_w$, and~\eqref{eq:weight_update_perturbation} then yields $J'_{\nuk}(\vt) \geq J'_{\nukp}(\vt) - \CC_w > -2\alpha^{-1}\log\varepsilon_k$.
\end{proof}

\medskip

\begin{proof}[Proof of Proposition \ref{prop:hypdist}, $ii)$]
Recall that the update defined in Section~\ref{sec:update} yields:    

\begin{equation}
    \nukpp := \nukkp + \epk \bm{1}_{\mathcal{N}_{\nukp}} \lambda = \nu_{k^+} (1 - \bm{1}_{\mathcal{P}_{\nukp}}) + \epk \bm{1}_{\mathcal{N}_{\nukp}} \lambda  ,
\end{equation}
where $\nukkp$ is a positive measure. 

\paragraph{Assumption $(\mathcal{H}_\varepsilon^+)$} The definition of $\mathcal{N}_{\nukp} = \{J'_{\nukp} \leq 0\}$  implies Assumption $\mathcal{H}_\varepsilon^+$. %

\paragraph{Assumption $(\mathcal{H}_\varepsilon^{\mathrm{smooth},1})$}
\label{s:proof_delet}

First, using \eqref{eq:Frechet_Jprime},
\begin{eqnarray*}
J(\nu_{k+1}) - J(\nu_{k^+})
& = & \int J_{\nu_{k^+}}' d(\nu_{k+1}- \nu_{k^+}) + \frac{1}{2} \| \Phi(\nu_{k+1}- \nu_{k^+}) \|_\mathbb{H}^2, \\
& \leq & - \int_{\mathcal{P}_{\nukp}} J_{\nu_{k^+}}' d\nu_{k^+} +\varepsilon_k \int_{\mathcal{N}_{\nukp}} J_{\nu_{k^+}}' d\lambda \\
&&\quad +  \| \Phi(\nu_{k^+} \bm{1}_{\mathcal{P}_{\nukp}} )\|_\mathbb{H}^2 + \| \Phi ( \epk \bm{1}_{\mathcal{N}_{\nukp}} \lambda) \|_\mathbb{H}^2 ,\\
& \leq &  \| \Phi(\nu_{k^+} \bm{1}_{\mathcal{P}_{\nukp}} )\|_\mathbb{H}^2 + \| \Phi ( \epk \bm{1}_{\mathcal{N}_{\nukp}} \lambda) \|_\mathbb{H}^2,
\end{eqnarray*}
since $\mathcal{P}_{\nukp} \subset \lbrace J_{\nu_k^+}' >0 \rbrace $ and $\mathcal{N}_{\nukp} \subset \lbrace J_{\nu_k^+}' <0 \rbrace$. Now, by Remark~\ref{rem:weight_update_perturbation}, it holds the following inclusion $\mathcal{P}_{\nukp} \subset \{J'_{\nuk} > -2\alpha^{-1}\log\varepsilon_k\}$, so that $e^{-\alpha J'_{\nuk}(\vt)} < \varepsilon_k^2$ on $\mathcal{P}_{\nukp}$. Now, using that $d\nukp(\vt) = e^{-\alpha J'_{\nuk}(\vt)} d\nuk(\vt)$ and $\|\varphi_\vt\|_\mathbb{H} = 1$:
\begin{eqnarray*}
\| \Phi(\nu_{k^+} \bm{1}_{\mathcal{P}_{\nukp}}) \|_\mathbb{H}^2
& = & \left\| \int_{\mathcal{P}_{\nukp}} \varphi_\vt d\nu_{k^+}(\vt) \right\|_\mathbb{H}^2, \\
& \leq & \nu_{k^+}(\mathcal{P}_{\nukp}) \times \int_{\mathcal{P}_{\nukp}} \| \varphi_\vt \|_\mathbb{H}^2 d\nu_{k^+}(\vt), \\
& = & \left[\int_{\mathcal{P}_{\nukp}} e^{-\alpha J_{\nu_k}'(\vt)} d\nu_k(\vt)\right]^2 \leq \| \nu_k\|_{\mathrm{TV}}^2 \varepsilon_k^4 \leq \| \nu_k\|_{\mathrm{TV}}^2 \varepsilon_k^2,
\end{eqnarray*}
where the last inequality uses $\varepsilon_k \leq 1$; the tighter $\varepsilon_k^4$ bound shows that the prune-set contribution is asymptotically dominated by the birth-set $\varepsilon_k^2$ term derived below. Similarly, by Cauchy--Schwarz,
\begin{eqnarray*}
\| \Phi ( \epk \bm{1}_{\mathcal{N}_{\nukp}} \lambda) \|_\mathbb{H}^2
& = & \varepsilon_k^2 \left\| \int_{\mathcal{N}_{\nukp}} \varphi_s \, d\lambda(s) \right\|_\mathbb{H}^2, \\
& \leq & \varepsilon_k^2 \, \lambda(\mathcal{N}_{\nukp}) \int_{\mathcal{N}_{\nukp}} \|\varphi_s\|_\mathbb{H}^2 \, d\lambda(s) \;=\; \varepsilon_k^2 \, \lambda(\mathcal{N}_{\nukp})^2 \;\leq\; \lambda(\mathcal{X})^2 \varepsilon_k^2,
\end{eqnarray*}
using $\|\varphi_s\|_\mathbb{H} = 1$.
This entails that Assumption $(\mathcal{H}_\varepsilon^{\mathrm{smooth},1})$ is satisfied.

\paragraph{Assumption $(\mathcal{H}_\varepsilon^{\mathrm{smooth},2})$} Concerning $(\mathcal{H}_\varepsilon^{\mathrm{smooth},2})$, we use again the inclusion $\mathcal{P}_{\nukp} \subset \{J'_{\nuk} > -2\alpha^{-1}\log\varepsilon_k\}$ (Remark~\ref{rem:weight_update_perturbation}). For any $\vt\in \mathcal{X}$:
\begin{eqnarray*}
\left| J_{\nu_{k+1}}'(\vt) - J_{\nu_{k^+}}'(\vt) \right|
& = & \left| \langle \varphi_\vt , \Phi(\nu_{k+1} - \nu_{k^+})\rangle \right|, \\
& = & \left| \int \langle \varphi_u, \varphi_\vt \rangle_\mathbb{H} d(\nu_{k+1} - \nu_{k^+})(u) \right|, \\
& \leq & \left| \int_{\mathcal{P}_{\nukp}} \langle \varphi_u, \varphi_\vt \rangle_\mathbb{H} d\nu_{k^+}(u) \right| + \left| \varepsilon_k \int_{\mathcal{N}_{\nukp}} \langle \varphi_u, \varphi_\vt \rangle_\mathbb{H} d\lambda(u) \right|, \\
& \leq & \int_{\mathcal{P}_{\nukp}}  e^{-\alpha J_{\nu_k}'(u)} d\nu_k(u) + \varepsilon_k \lambda(\mathcal{X}), \\
& \leq & \|\nu_k\|_{\mathrm{TV}} \varepsilon_k^2  + \varepsilon_k \lambda(\mathcal{X}). %
\end{eqnarray*}
This inequality is uniform in $\vt\in \mathcal{X}$, which proves the desired result. 
\end{proof}

\subsection{One-step analysis  \label{sec:descent_det}}
We introduce:
\begin{equation}\label{eq:min}
\vkd = \min (J'_{\nukm}) \wedge 0 \quad \mathrm{and} \quad \Thetakm := \left\{ J'_{\nukm} \leq \frac{\vkd}{2}\right\}.\end{equation}
We now present a quantitative result that explicitly connects the weight update $T_{\nu, \alpha}$ to Assumption~\eqref{eqs:H_eps_determinist}, by relating the quantity $\vkd$ to the evolution of the energy $J$ under the update $T_{\nu, \alpha}$.
Proposition \ref{prop:boundvk} below is largely inspired by Proposition~H.1 in \cite{chizat2022sparse}. Here, to relate the increments $J(\nukp) - J(\nuk)$ to the minimum of $J'_{\nukm}$, we rely on two main ingredients: the descent property (see Proposition \ref{prop:incre}) and Assumption $\Heps$, which ensures that sufficient mass is present where $J'_{\nu_k}$ is negative.

\begin{proposition}
\label{prop:boundvk}
Assume $\Heps$ and $\HTVC$ hold so that Equations \eqref{eqs:H_eps_determinist} and \eqref{eq:HTVC} are satisfied, if $\alpha$ is chosen such that \eqref{eq:small_learning_rates} holds and assume that $\vkd^2 \geq 24\epkm^2 \CC^2$, then:
\[
J(\nukp)-J(\nuk) \leq - \frac{3\alpha}{2} C_d \Lip^{-d} |\vkd|^{2+d} \epkm.
\]
where $\Lip$ is introduced in \eqref{def:lip_uniform_J'} and $C_d$ in Lemma \ref{lem:geometric_sublevel}.
\end{proposition}

\begin{subequations}
 \begin{proof}
Consider $k\in \mathbb{N}^\star$. Our starting point is Proposition \ref{prop:incre}: we keep only the negative descent contribution produced by $|J'_{\nuk}|^2$ and observe that, with $\alpha$ small enough so that \eqref{eq:small_learning_rates} holds and with $\beta=0$, we have:
\begin{align}
J(\nukp)-J(\nuk) 
& \leq - \frac{3}{4} \|g_{\nuk}^\alpha\|^2_{L^2(\nuk)} \nonumber\\
& = - \frac{3}{4} \alpha \int_{\mathcal{X}} |J'_{\nuk}|^2 d\nuk \label{eq:inter-inter1} \\
& \leq - \frac{3}{4}
\alpha \int_{\Thetakm} |J'_{\nuk}|^2 d\nuk.
\label{eq:inter1}
\end{align}
Now, according to Assumption $(\mathcal{H}^{\text{smooth},2}_{\varepsilon})$, we know that 
$|J'_{\nuk}(t)-J'_{\nukm}(t)| \leq \CC \epkm$ uniformly for all values of $t$.
In particular, for any $t\in \mathcal{X}$, we deduce from the inequality
$(a+b)^2 \geq \frac{a^2}{2}-b^2$ that:
\begin{equation}
|J'_{\nu_k}(t)|^2 \geq \frac{1}{2} |J'_{\nukm}(t)|^2 -  \CC^2 \epkm^2 
\label{eq:inter2}
\end{equation}
Using \eqref{eq:inter1} together with \eqref{eq:inter2} leads to:
\begin{align*}
J(\nukp)-J(\nuk) 
& \leq - \frac{3}{8}\alpha \int_{\Thetakm} |J'_{\nukm}|^2 d\nu_k + \frac{3}{4}\alpha {\CC}^2 \epkm^2\int_{\Thetakm} d\nu_k, \\
& \leq -\frac{3}{32}\alpha \vkd^2 \int_{\Thetakm} d\nu_k + \frac{3}{4}\alpha {\CC}^2 \epkm^2   \int_{\Thetakm} d\nu_k, \\
& = - \frac{3}{8} \alpha \left[  \frac{\vkd^2}{4} - 2 \epkm^2 \CC^2 \right]\int_{\Thetakm} d\nu_k,
\end{align*}
where we used the definition of $\Thetakm$ to lower bound $|J'_{\nukm}|^2$ in the first line. In particular, in the specific regime considered here, namely when $24\epkm^2  \CC^2 \leq \vkd^2$, we deduce that:
$$
J(\nukp)-J(\nuk) \leq - \frac{3}{2} \alpha  \vkd^2 \int_{\Thetakm} d\nu_k.
$$
Finally, Assumption $(\mathcal{H}^+_{\varepsilon})$ yields:
\begin{equation} 
J(\nukp)-J(\nuk) \leq -\frac{3}{2}\alpha \epkm \vkd^2 \lambda(\Thetakm),
\label{eq:inter3}
\end{equation}

\noindent We are led to lower bound $\lambda(\Thetakm)$. To this end, we introduce the following geometric lemma:
\begin{lemma}[Geometric lower bound on the sub-level set]
\label{lem:geometric_sublevel}
    Let $g: \mathcal{X} \subset \mathbb{R}^d \to \mathbb{R}$ be a $\Lip$-Lipschitz function achieving a minimum value $v^\star = \min_{x \in \mathcal{X}} g(x) \le 0$. There exists a purely dimensional constant $C_d > 0$ such that the Lebesgue measure of its sub-level sets satisfies:
    $$
    \lambda(\{x \in \mathcal{X} : g(x) \le 0\}) \ge \lambda\left(\left\{x \in \mathcal{X} : g(x) \le \frac{v^\star}{2}\right\}\right) \ge C_d \Lip^{-d} |v^\star|^d.
    $$
\end{lemma}
Recall from Lemma \ref{lem:lipschitz_J_prime} that for any $\nu\in \mathcal{M}(\mathcal{X})$,  $J'_{\nu}$ is a $\mathfrak{L}(\nu)$-Lipschitz function with $\mathfrak{L}(\nu) \leq \Lip$ as soon as $\| \nu\|_{TV} \leq \CTV$ (see Remark \ref{rem:lipschitz_J_prime}). In the meantime, we know from Assumption~\eqref{eq:HTVC} that $(\nuk)_{k \ge 1}$ and $(\nukp)_{k \ge 1}$ are two bounded sequences in terms of their TV-norm.   Therefore, $\mathfrak{L}(\nu_k)\leq \Lip$.
By applying Lemma~\ref{lem:geometric_sublevel} to $g = J'_{\nukm}$ and observing that $\vkd = \min_{x} J'_{\nukm}(x) \le 0$, we have:
\begin{align*}
\forall x\in \mathcal{X}: \qquad 
    |x-\argmin J'_{\nukm}| \leq \frac{|\vkd|}{2\Lip} & 
         \Longrightarrow |J'_{\nukm}(x)-\min(J'_{\nukm})| \leq \Lip \times \frac{|\vkd|}{2\Lip}, \\
         &  \Longrightarrow J'_{\nukm}(x) \leq \frac{\vkd}{2} \le 0,\\
& \Longrightarrow x \in \Thetakm.         
\end{align*}
We then deduce that 
$$
\left\{ x \in \mathcal{X} : |x-\argmin J'_{\nukm}| \leq \frac{|\vkd|}{2\Lip} \right\} \subset 
\Thetakm.
$$
This inclusion entails that the volume is bounded by that of a Euclidean ball of radius $\frac{|\vkd|}{2\Lip}$, so for some constant $C_d > 0$:
$$
\lambda(\Thetakm) \ge C_d \Lip^{-d} |\vkd|^{d}.
$$
We finally obtain from the previous lower bound and from \eqref{eq:inter3} that:
$$
J(\nukp)-J(\nuk) \leq - \frac{3\alpha}{2} C_d \Lip^{-d} |\vkd|^{2+d} \epkm.
$$
This concludes the proof. 
\end{proof}
\end{subequations}
 
The following proposition describes the evolution of the cost function itself along the iterations. For this purpose, we introduce the non-negative sequence $(\Delta_{k^+})_{k \geq 1}$ and the auxiliary sequence $(\Delta_k)_{k \geq 1}$ defined by
\begin{equation}\label{def:Delta}
\forall k \geq 1, \quad \Delta_{k^+} := J(\nuk) - J(\nukp) \quad \text{and} \quad \Delta_k := J(\nu_{k-1}) - J(\nu_k).
\end{equation}
These two sequences are then associated with our longitudinal evolution as follows.

\[
\forall k \ge 1 \qquad
\underbrace{\overbrace{\nu_{k-1}\longrightarrow\nukm}^{\Delta_{{k-1}^+}}\longrightarrow
\lefteqn{\overbrace{\phantom{\nuk\longrightarrow\nukp}}^{\Delta_{k^+}}}\nuk}_{\Delta_k}\longrightarrow\nukp.
\]
We emphasize that, from Proposition \ref{prop:incre}, the sequence $(\Delta_{k^+})_{k \geq 1}$ is non-negative. However, we cannot draw the same conclusion for the sequence $(\Delta_k)_{k \geq 1}$.

\begin{proposition}
Assuming \eqref{eqs:H_eps_determinist} and \eqref{eq:HTVC}, and that $\alpha$ is chosen so that \eqref{eq:small_learning_rates} holds,
then 
$$
\forall k \geq 1 \qquad 
J(\nu_k) - J^\star \leq \CC \left(\|\nu^\star\|_{\mathrm{TV}} \vee \Lip\right) \max\left( \left[  \frac{\Lip^d \Delta_{k^+}}{\alpha \epkm} \right]^{\frac{1}{2+d}} ; \  \epkm \right).$$
\label{prop:upper-bound-increment}
\end{proposition}
\noindent  
The above result relates the evolution of the sequence $J(\nu_k)$ to the increments $J(\nu_k^+) - J(\nu_k)$, and is primarily based on Proposition~\ref{prop:boundvk}.
This result involves a trade-off in the choice of the parameter $\varepsilon_k$ (the amount of mass available in regions where $J_{\nu_k}'$ is negative), which must be selected carefully. On the one hand, $\varepsilon_k$ needs to be sufficiently large to control the increment of the objective function (see Proposition~\ref{prop:boundvk}); on the other hand, they contribute an additional term that affects the value of $J(\nu_k)$.

\medskip

\begin{subequations}
\begin{proof}
Using \eqref{eq:Frechet_Jprime}, we have in a first time
$$ J^\star - J(\nu_k) = \int_\mathcal{X} J_{\nu_k}' d(\nu^\star - \nu_k) + \frac{1}{2} \| \Phi(\nu_k-\nu^\star)\|^2.$$
This implies 
\begin{equation}
J(\nuk) - J^\star \leq  \int_\mathcal{X} J_{\nu_k}' d (\nuk - \nu^\star) = \underbrace{\int_\mathcal{X} J_{\nu_k}' d \nuk}_{:=A} -  \underbrace{\int_\mathcal{X} J_{\nu_k}' d \nu^\star}_{:=B}.
\label{eq:cvg1}
\end{equation}

\noindent \underline{Study of $A$:}
We can first observe that \eqref{eq:inter-inter1} can be written as
\begin{equation}
\int_{\mathcal{X}} |J'_{\nuk}|^2 d\nuk \leq 4 \frac{J(\nuk)-J(\nukp)}{3 \alpha} = \frac{4\Delta_{k^+}}{3 \alpha}.
 \label{eq:inter-inter1_a}
\end{equation}
Then, the Cauchy-Schwarz inequality associated with Equation \eqref{eq:inter-inter1_a} yields:

\begin{equation} 
\label{eq:cvg2}
|A| = \left| \int_\mathcal{X} J_{\nu_k}' d\nu_k\right| \leq \left[ \|\nu_k\|_{\mathrm{TV}} \int_\mathcal{X} |J_{\nu_k}'|^2 d\nu_k \right]^{1/2} \leq \left[ 4 \frac{\|\nu_k\|_{\mathrm{TV}}}{3\alpha} \Delta_{k^+} \right]^{1/2}.
\end{equation}

\noindent \underline{Study of $B$:}
We use the smoothness of $\nu \longmapsto J'_\nu$ induced by $(\mathcal{H}^{\text{smooth},2}_{\varepsilon})$ and obtain that:
\begin{align*}
B=\int_\mathcal{X} J_{\nu_k}'\text{d}\nu^\star & = \int_\mathcal{X} J_{\nukm}'\text{d}\nu^\star + 
\int_\mathcal{X} (J_{\nu_k}'-J_{\nukm}')\text{d}\nu^\star \\
& \geq \vkd \|\nu^\star\|_{\mathrm{TV}} - \|\nu^\star\|_{\mathrm{TV}} \|J_{\nu_k}'-J_{\nukm}'\|_{\infty}\\
& \geq \vkd  \|\nu^\star\|_{\mathrm{TV}}  - \CC \epkm \|\nu^\star\|_{\mathrm{TV}}.
\end{align*}

\noindent
At this stage, two distinct cases may arise depending on the value of $\vkd$.

\begin{itemize}
\item \textbf{$1^\mathrm{st}$ case}: $\vkd \leq - 2 \sqrt{6}\CC \epkm$. Then, we get from the previous bound that
\begin{align*}
B = \int_\mathcal{X} J_{\nu_k}'\text{d}\nu^\star &\geq 
\vkd  \|\nu^\star\|_{\mathrm{TV}}  - \CC \epkm \|\nu^\star\|_{\mathrm{TV}}\\
& \geq \left( 1+\frac{1}{2\sqrt{6}}\right) \vkd  \|\nu^\star\|_{\mathrm{TV}} \\
& \geq 
-\left( 1+\frac{1}{2\sqrt{6}}\right) \|\nu^\star\|_{\mathrm{TV}} \left[ \frac{2\Delta_{k^+} \Lip^d}{3\alpha \epkm C_d} \right]^{\frac{1}{2+d}},
\end{align*}
 where the last line is obtained using Proposition \ref{prop:boundvk}.
We deduce that:
 $$
 \int_\mathcal{X} J_{\nu_k}'\text{d}\nu^\star \geq - \left( 1+\frac{1}{2\sqrt{6}}\right) \|\nu^\star\|_{\mathrm{TV}} \left[ \frac{2 \Lip^d \Delta_{k^+}}{3\alpha \epkm C_d} \right]^{\frac{1}{2+d}}.
 $$
 
\item \textbf{$2^\mathrm{nd}$ case:} $\vkd \geq - 2\sqrt{6}\CC  \varepsilon_{k-1} $. In such a situation we immediately have
$$
\int_\mathcal{X} J_{\nu_k}'\text{d}\nu^\star \geq \vkd \|\nu^\star\|_{\mathrm{TV}} - \CC \epkm \| \nu^\star\|_{\mathrm{TV}} \geq -(1+2\sqrt{6}) \CC \| \nu^\star\|_{\mathrm{TV}} \epkm.
$$
\end{itemize}

\noindent
Regardless of the value of $\vkd$, we then get 
\begin{equation} 
\int_\mathcal{X} J_{\nu_k}'\text{d}\nu^\star \geq - \|\nu^\star\|_{\mathrm{TV}} \min \left( (1+2\sqrt{6})\CC \epkm; \left( 1+ \frac{1}{2\sqrt{6}}\right) \left[ \frac{2 \Lip^d \Delta_{k^+}}{3\alpha \epkm C_d} \right]^{\frac{1}{2+d}}\right).
\label{eq:cvg3a}
\end{equation}
Using Equations \eqref{eq:cvg1}, \eqref{eq:cvg2} and \eqref{eq:cvg3a}, we deduce that 
\begin{eqnarray*} 
J(\nu_k)-J^\star \leq \CC \left(\|\nu^\star\|_{\mathrm{TV}} \vee \Lip \right)
\max\left( \left[  \frac{\Lip^d \Delta_{k^+}}{\alpha \epkm} \right]^{\frac{1}{2+d}} ; \left[\frac{\Delta_{k^+}}{\alpha } \right]^{\frac{1}{2}} ; \ \epkm \right).
\end{eqnarray*}
We conclude while observing that $(\Delta_{k^+})_{k \ge 1}$ and $(\epk)_{k \ge 1}$ are two bounded sequences, which implies that:
$$
J(\nu_k)-J^\star \leq \CC  \left(\|\nu^\star\|_{\mathrm{TV}} \vee \Lip \right) \max\left( \alpha^{-1/2} \left[  \frac{\Lip^d \Delta_{k^+}}{\epkm} \right]^{\frac{1}{2+d}} ; \  \epkm \right).
$$
\end{proof}
\end{subequations}

\subsection{Proof of the deterministic global convergence\label{s:proofconv1}}
Below, we finally provide the proof of our global convergence result in the deterministic situation, stated in Theorem \ref{theo:convergence_deterministe}. We will use the key property obtained in Section \ref{sec:descent_det}.

\medskip

\begin{subequations}
\begin{proof}[Proof of Theorem \ref{theo:convergence_deterministe}]
We introduce the function $f: \mathbb{N} \longrightarrow \mathbb{R}^+$ defined by $f(k)=J(\nu_k)-J^\star$ for all $k\in \mathbb{N}$.
The triangle inequality and Assumption $(\mathcal{H}^{\text{smooth},1}_{\varepsilon})$ yields
\begin{equation}
    \Deltakpp=J(\nuk)-J(\nukpp) 
    =J(\nuk)-J(\nukp)+J(\nukp)-J(\nukpp)
     \geq \Deltakp - \CC \epk^2. \label{eq:Key-deter}
\end{equation} 
Using Equation \eqref{eq:Key-deter} and the definition of $f$, we then obtain:
\begin{equation}\label{eq:tec-case0_bis}
f(k)-f(k+1)=  \Delta_{k+1} \ge  \Delta_{k^+} - \CC \varepsilon_k^2.
\end{equation}
Simultaneously, if $A=\CC \left(\|\nu^\star\|_{\mathrm{TV}} \vee \Lip \right)$, then we can apply Proposition \ref{prop:upper-bound-increment} to get
\[
[A^{-1} f(k)]^{2+d} \leq  \max \left( \frac{\Lip^d \Delta_{k^+}}{\alpha  \varepsilon_{k-1}},\varepsilon_{k-1}^{2+d}\right) \leq  \frac{\Lip^d \Delta_{k^+}}{\alpha  \varepsilon_{k-1}}+ \varepsilon_{k-1}^{2+d}.
\]
It leads to
\begin{equation}
\label{eq:tec-case1_bis}
\Delta_{k^+} \ge  \alpha \Lip^{-d} \varepsilon_{k-1} [A^{-1} f(k)]^{2+d} -  \alpha \Lip^{-d}  \varepsilon_{k-1}^{3+d} \ge 
  \alpha \Lip^{-d} \varepsilon_{k-1} [A^{-1} f(k)]^{2+d} -  \alpha \Lip^{-d} \varepsilon_{k-1}^{2} ,
\end{equation}
where the last inequality comes from the bound $\varepsilon_k \leq 1$. Hence, using together \eqref{eq:tec-case0_bis} and \eqref{eq:tec-case1_bis}, we get
\[
f(k)-f(k+1) \geq  \alpha \Lip^{-d} \varepsilon_{k-1} [A^{-1} f(k)]^{2+d} - (1+\alpha \Lip^{-d}) \varepsilon_{k-1}^{2}.
\]
We use a telescopic sum argument (summing for $k=1,\dots,K-1$), leading to the inequality
$$
 f(1) - f(K) \geq  \alpha \Lip^{-d}A^{-2-d} \sum_{k=1}^{K-1} \varepsilon_{k-1} f(k)^{2+d} -  (1+\alpha \Lip^{-d} ) \sum_{k=1}^{K-1} \varepsilon_{k-1}^{2},$$
 which in turn implies
 \begin{equation}\label{eq:inegalite_importante}
f(K) + \alpha \Lip^{-d}A^{-2-d} \sum_{k=1}^{K-1} \varepsilon_{k-1} f(k)^{2+d} \leq f(1) +  (1+\alpha \Lip^{-d} ) \sum_{k=1}^{K-1} \varepsilon_{k-1}^{2}.
\end{equation}

\medskip

\noindent
\underline{Proof of $i)$ and $ii)$:}
We introduce in the following the quantity $\rho_K$, defined as the minimum value of the sequence $(J(\nu_k) - J(\nu^\star))_{1\le k\le K}$ over the first $K$ iterations:
$$
\forall K\in \mathbb{N} \qquad \rho_K = \min_{0 \leq k \leq K} J(\nu_k)-J(\nu^\star) = \min_{0\leq k \leq K} f(k).
$$
Thanks to the definition of $\rho_K$, \eqref{eq:inegalite_importante} implies that 
\[
\rho_K^{2+d} \alpha \Lip^{-d}A^{-2-d} \sum_{k=1}^{K-1} \varepsilon_{k-1}  \leq  \left( f(1) + (1+\alpha \Lip^{-d}) \sum_{k=1}^{K-1} \varepsilon_{k-1}^{2}\right),
\]
which can be rewritten as
\begin{equation}
\rho_K \leq  \CC \left[\frac{\displaystyle1+(1 + \alpha \Lip^{-d}) \sum_{k=1}^{K-1} \varepsilon_{k-1}^{2}}{\displaystyle
\alpha \Lip^{-d}A^{-2-d} \sum_{k=1}^{K-1} \varepsilon_{k-1} }\right]^{\frac{1}{2+d}}.
\label{eq:inter_cvgce}
\end{equation}
Starting from \eqref{eq:inter_cvgce}, we now consider two different cases.

\medskip

\noindent
$\bullet$ \textit{Horizon dependent step-size sequence $i)$}
 Considering the case of a constant step-size sequence $(\epk)_{k \ge 0}$
 with $\epk=\varepsilon$ for all $k\in \mathbb{N}$, we deduce from \eqref{eq:inter_cvgce} that in this specific case
 $$
 \forall K \ge 0 \qquad \rho_K \leq \CC \frac{A \Lip^{\frac{d}{2+d}}}{\alpha^{\frac{1}{2+d}}} \left( \frac{1}{ (K-1) \varepsilon } + \varepsilon  (1+ \alpha \Lip^{-d} )\right)^{\frac{1}{2+d}}.
 $$
 It remains to optimize the previous upper bound in terms of $\varepsilon$. The trade-off between the two terms appearing in the r.h.s. of the previous bound is attained for 
 $$\varepsilon = \CC \sqrt{\frac{1}{ (K-1)(1+\alpha \Lip^{-d} )}},$$ 
 which yields
 \begin{equation}
\rho_K \leq \CC \frac{A \Lip^{\frac{d}{2+d}}}{\alpha^{\frac{1}{2+d}}}  \left(\frac{(1+ \alpha \Lip^{-d})}{(K-1)}\right)^{\frac{1}{2(2+d)}}.
 \end{equation}

\medskip

\noindent
$\bullet$ \textit{Horizon-free step-size sequence $ii)$}
It is also possible to derive a convergence rate with a horizon-free step-size sequence that does not depend on the horizon of the simulation. For this purpose, we simply consider the sequence:
$$\forall k \ge 0 \qquad \varepsilon_k = \frac{\CC}{\sqrt{ (k+1)}}.$$
In this case, we verify that:
$$
\sum_{k=1}^{K-1} \epk^2 \leq   \CC \sum_{k=1}^{K-1} \frac{1}{k} \leq  \CC [\log(K)+1] \quad \text{and} \quad \sum_{k=1}^{K-1} \epk =  \sum_{k=1}^K \frac{\CC}{\sqrt{k}} \ge 2 \CC   (\sqrt{K}-1).
$$
Then
$$ \rho_K\leq \CC \alpha^{-\frac{1}{(2+d)}} \left(\frac{  1}{ K}\right)^{\frac{1}{2(2+d)}} \log^{\frac{1}{2+d}}(K).
$$

\medskip

\noindent
\underline{Proof of $iii)$:}
We first introduce the function $\bar f: \mathbb{N} \longrightarrow \mathbb{R}$ defined as 
\begin{equation}
\bar f(K) = \max \left( J(\nu_K) - J^\star - \CC \sum_{k=1}^{K-1} \varepsilon_k^2 \ ; \ 0 \right) \quad \forall K\in \mathbb{N},
\label{eq:newf}
\end{equation}
where $\CC$ is here related to Assumption $(\mathcal{H}_\varepsilon^{\mathrm{smooth},1})$ (see \eqref{eq:smooth_1}).  This function is non-increasing as, for any $K\in \mathbb{N}$,
\begin{align*}
\bar f(K+1) & = \max \left( J(\nu_{K+1}) - J^\star - \CC \sum_{k=1}^{K} \varepsilon_k^2 \ ; \ 0 \right) \\
 &=  \max \left( J(\nu_{K+1})-J(\nu_{K^+})+J(\nu_{K^+})-J(\nu_K)+J(\nu_K)-J^\star- \CC \sum_{k=1}^{K} \varepsilon_k^2 \ ; \ 0 \right) \\
 & \leq \max \left( \CC \varepsilon_K^2+J(\nu_K)- J^\star - \CC \sum_{k=1}^{K} \varepsilon_k^2 \ ; \ 0 \right) \\
 & = \bar f(K)
\end{align*}
where we applied Assumption $\Heps$ in the third line and the fact that $J(\nukp)-J(\nuk)\leq 0$ (see Proposition~\ref{prop:incre}). We first assume that 
\begin{equation}
\bar f(K) \geq 0.
\label{eq:minorbarf}
\end{equation}
Remark that \eqref{eq:minorbarf} together with the non-increasing property of $\bar f$ entails that $\bar f(k) \geq 0$ for any $k\in \lbrace 1,\dots , K \rbrace$. We use the discrete integration by part relationship: for $p=\bar f(k+1)^{-1}$ and $q=\bar f(k)^{-1}$, 
\begin{equation}\label{eq:sum_power_d}
p^{1+d}-q^{1+d} = (p-q) \sum_{i=0}^d p^i q^{d-i}.    
\end{equation}
The relationship $\bar f(k) \ge \bar f(k+1)$ then provides, since every term $p^i q^{d-i}$ in \eqref{eq:sum_power_d} is bounded below by $q^d = \bar f(k)^{-d}$ (resp.\ by $p\, q^d$ for the $i\geq 1$ terms),
\begin{equation}
\frac{1}{\bar f(k+1)^{1+d}} - \frac{1}{\bar f(k)^{1+d}} \geq \left( \bar f(k)-\bar f(k+1) \right)  \frac{d+1}{\bar f(k+1)\,\bar f(k)^{1+d}}.
\label{eq:draft1}
\end{equation}
Using \eqref{eq:minorbarf}, we have
\begin{eqnarray}
\bar f(k) - \bar f(k+1)
& \geq & J(\nu_{k}) - J(\nu_k^+) + J(\nu_k^+) - J(\nu_{k+1})   + \CC \varepsilon_k^2, \nonumber \\
& = & \Delta_{k^+} + J(\nu_k^+) - J(\nu_{k+1})   + \CC \varepsilon_k^2, \nonumber \\
& \geq & \Delta_{k^+},
\label{eq:draft2}
\end{eqnarray}
since 
$$ J(\nu_k^+) - J(\nu_{k+1}) \geq - \CC \varepsilon_k^2,$$
according to Assumption $(\mathcal{H}_\varepsilon^{\mathrm{smooth},1})$. Hence, \eqref{eq:draft1} together with \eqref{eq:draft2} leads to
\begin{equation}
\frac{1}{\bar f(k+1)^{1+d}} - \frac{1}{\bar f(k)^{1+d}} \geq \Delta_{k^+} \frac{d+1}{\bar f(k+1)\,\bar f(k)^{1+d}} \geq \Delta_{k^+} \frac{d+1}{\bar f(k)^{2+d}},
\label{eq:draft3}
\end{equation}
where the last inequality uses $\bar f(k+1)\leq \bar f(k)$.
Since for all $k\geq 0: \bar f(k) \leq J(\nu_k) - J^\star = f(k)$,  
\eqref{eq:tec-case1_bis} entails that
\begin{equation}
\label{eq:tec-case1-nomin}
\Delta_{k^+} \ge  
 \alpha \varepsilon_{k-1} \Lip^{-d} [A^{-1} \bar f(k)]^{2+d} -  \alpha  \Lip^{-d} \varepsilon_{k-1}^{2} ,
\end{equation}
This last inequality, together with \eqref{eq:draft3}, leads to 
\begin{eqnarray*}
\frac{1}{\bar f(k+1)^{1+d}} - \frac{1}{\bar f(k)^{1+d}} 
& \geq &  \left( \alpha \Lip^{-d} \varepsilon_{k-1} [A^{-1}\bar f(k)]^{2+d} - \alpha \Lip^{-d}  \varepsilon_{k-1}^{2}\right)  \frac{d+1}{\bar f(k)^{2+d}} , \\
& = &  \left( \alpha \Lip^{-d} \varepsilon_{k-1} A^{-(2+d)}%
- \alpha \Lip^{-d} \varepsilon_{k-1}^{2}\frac{1}{\bar f(k)^{2+d}} \right) (d+1) .
\end{eqnarray*}
We then use a telescopic sum argument to obtain:
\begin{eqnarray*}
\frac{1}{\bar f(K)^{1+d}} - \frac{1}{\bar f(1)^{1+d}} 
& \geq & \CC (d+1) \left[ \alpha \Lip^{-d} A^{-(2+d)} \sum_{k=1}^{K-1} \varepsilon_{k-1} - \alpha \Lip^{-d}  \sum_{k=1}^{K-1} \varepsilon_{k-1}^2 \frac{1}{\bar f(k)^{2+d}} \right], \\
& \geq &\CC (d+1)\Lip^{-d} \left[ \alpha  A^{-(2+d)} \sum_{k=1}^{K-1} \varepsilon_{k-1} - \alpha \frac{1}{\bar f(K)^{2+d}} \sum_{k=1}^{K-1} \varepsilon_{k-1}^2  \right],
\end{eqnarray*}
since $\bar f$ is non-increasing. This last inequality can be re-written as:
$$  \alpha \Lip^{-d}  (d+1) \frac{1}{\bar f(K)^{2+d}} \sum_{k=1}^{K-1} \varepsilon_{k-1}^2 + \frac{1}{\bar f(K)^{1+d}} \geq  (d+1) \alpha \Lip^{-d} A^{-(2+d)} \sum_{k=1}^{K-1} \varepsilon_{k-1} + \frac{1}{\bar f(1)^{1+d}}.$$
Using again the monotonicity of $\bar{f}$, we deduce that:
\begin{multline*}
\frac{1}{\bar f(K)^{2+d}} \left[ \alpha \Lip^{-d}  (d+1)  \sum_{k=1}^{K-1} \varepsilon_{k-1}^2 + \bar{f}(1)\right] \\
\geq \CC (d+1) \Lip^{-d} \alpha A^{-(2+d)} \sum_{k=1}^{K-1} \varepsilon_{k-1} + \frac{1}{\bar f(1)^{1+d}},
\end{multline*}
where the monotonicity $\bar f(K)\leq \bar f(1)$ was used to upper bound $\bar f(K)^{-(1+d)}$ by $\bar f(1)\,\bar f(K)^{-(2+d)}$ (the finiteness of $\bar f(1) = J(\nu_1)-J^\star$ itself being ensured by $\HTVC$). The last inequality can be rewritten as
\[
\bar f(K) \leq \CC \left[  \frac{\alpha \Lip^{-d} (d+1)  \sum_{k=1}^{K-1} \varepsilon_{k-1}^2 + \bar{f}(1)}{(d+1) \alpha \Lip^{-d} A^{-(2+d)} \sum_{k=1}^{K-1} \varepsilon_{k-1} + \bar f(1)^{-(1+d)}} \right]^{\frac{1}{2+d}}.
\]
Since $\bar{f}(1) = J(\nu_1)-J^* \ge 0$, this leads to 
\begin{equation}
J(\nu_K) - J^* \leq \CC \left[  \frac{\alpha \Lip^{-d} (d+1)  \sum_{k=1}^{K-1} \varepsilon_{k-1}^2 + (J(\nu_1)-J^*)}{(d+1) \alpha \Lip^{-d} A^{-(2+d)} \sum_{k=1}^{K-1} \varepsilon_{k-1}} \right]^{\frac{1}{2+d}} + \CC \sum_{k=1}^{K-1} \varepsilon_{k}^2.
\label{eq:bound_bar_mK}
\end{equation}
In this context, we consider a horizon-dependent strategy, namely we set $\varepsilon_k = \varepsilon$ for any $k \in \lbrace 1,\dots, K\rbrace$. The bound \eqref{eq:bound_bar_mK} becomes in this case
\begin{eqnarray*} 
J(\nu_K) - J^\star
& \leq & \CC \left[  \frac{\alpha \Lip^{-d} (d+1)(K-1) \varepsilon^2 + (J(\nu_1)-J^*)}{(d+1) \alpha \Lip^{-d} A^{-(2+d)} (K-1) \varepsilon } \right]^{\frac{1}{2+d}} + \CC K \varepsilon^2, \\
& \leq & \CC \left[ \frac{1}{ A^{-(2+d)}} \varepsilon + \frac{J(\nu_1)-J^*}{(d+1) \Lip^{-d} \alpha A^{-(2+d)}} \times \frac{1}{(K-1)\varepsilon}\right]^{\frac{1}{2+d}} + \CC K \varepsilon^2\\
&\leq& \CC A \left[ \varepsilon + \frac{J(\nu_1)-J^*}{(d+1) \Lip^{-d} \alpha(K-1)\varepsilon}\right]^{\frac{1}{2+d}} + \CC K \varepsilon^2.
\end{eqnarray*}
Then, choosing $\varepsilon$ such that 
$$ \left( \frac{1}{(d+1) \alpha \Lip^{-d} A^{-(2+d)}}\times  \frac{1}{K\varepsilon} \right)^{\frac{1}{2+d}} =  K\varepsilon^2 \quad \Leftrightarrow \quad \varepsilon =  \left( \frac{A^{2+d}}{(d+1) \alpha \Lip^{-d}} \right)^{\frac{1}{5+2d}} K^{-\frac{3+d}{5+2d}},$$
and defining $\CC$ large enough to absorb $(J(\nu_1)-J^*)^{\frac{1}{2+d}}$, we obtain that (the constant $\CC$ depends polynomially on the initial excess, specifically as $(J(\nu_1)-J^\star)^{1/(2+d)}$):
$$ J(\nu_K) - J^\star \leq \CC  \left( \frac{\Lip^d\,A^{2+d}}{(d+1) \alpha } \right)^{\frac{2}{5+2d}} K^{-\frac{1}{5+2d}}.$$
To conclude the proof, we have to investigate the case where \eqref{eq:minorbarf} does not hold. Remark that the latter entails that 
\[
J(\nu_K) - J^\star \leq \CC \sum_{k=1}^K \varepsilon_k^2 \leq \CC K \epsilon^2 \leq \CC \left( \frac{\Lip^d\,A^{2+d}}{(d+1) \alpha } \right)^{\frac{2}{5+2d}} K^{-\frac{1}{5+2d}},
\]
keeping the same choice for $\varepsilon$. We then use  $A=\CC \left(\|\nu^\star\|_{\mathrm{TV}} \vee \Lip \right)$ in the final results.
\end{proof}
\end{subequations}

\section{Proof of the stochastic results\label{sec:appendix_sto}}

\subsection{Almost sure total variation bound}

We begin with the study of the almost sure TV norm upper bound, that will be then used throughout the rest of the proofs.

\begin{proposition}
\label{prop:proof_TVsto}
Define 
\[ \widehat{\RR}= \frac{\mathbf{H}}{\mathbf{G}} e + \sqrt{\frac{e^3}{\mathbf{G}}} +1\] 
and assume that $\alpha \le \frac{1}{1+\widehat{\RR}}$ and $\epk \leq \alpha$ for all $k \ge 1$, then the sequence $(\nuks)_{k \ge 0}$ satisfies:
$$
\forall k \ge 0 \qquad \|\nuks\|_{TV} \leq \widehat{\RR} \quad \text{a.s.}
$$
\end{proposition}

\begin{proof}
Our proof follows essentially the same lines as the deterministic case, except that we have to take into account the randomness of our updates.
Let $k\in \mathbb{N}^*$ be fixed. According to \eqref{eq:stoscheme_del} and \eqref{eq:stoscheme}, we have
\[
\hat\nu_{k+1} =\nukps - \bm{1}_{\Pkp}(V_{k+1} )\nukps(V_{k+1})\delta_{V_{k+1}}  + \varepsilon_k \bm{1}_{\Nkp}(U_{k+1}) \delta_{U_{k+1}}
\]
Computing the total variation norm, we obtain:
\[
\| \hat\nu_{k+1} \|_{\mathrm{TV}} \leq \|\hat\nu_{k^+}\|_{\mathrm{TV}} + \epk \leq   \int_{\mathcal{X}} e^{-\alpha \widehat{J_{\nuks}'}(\vt)} d\nu_k(\vt) + \|\varepsilon \|_\infty.
\]
Next, observe that for any $\vt\in \mathcal{X}$, according to Assumption \eqref{A2}, we have 
\[
\widehat{J'_{\nuks}}(\vt) =   \frac{1}{m_k} \sum_{l=1}^{m_k} \widehat{J'_{\nuks}}(\vt,Z_{l,k}^+) \geq \bm{G} \| \nuks\|_{\mathrm{TV}} -\bm{H} + \kappa \ge  -\bm{H} + \kappa  \qquad \text{a.s.},
\]
so that we have the almost sure upper bound:
$$
\| \nukpps \|_{\mathrm{TV}} \leq  e^{- \alpha \bm{G} \| \nuks\|_{\mathrm{TV}}+\alpha \bm{H}}
\| \nuks\|_{\mathrm{TV}} + \alpha \qquad \text{a.s.}
$$
Then, the rest of the proof proceeds exactly following the same lines as in Proposition 
\ref{prop:hypdist}, $i)$, whose proof is located in Section \ref{s:proof_hypdist}.
\end{proof}

\subsection{Proof of Proposition \ref{prop:hypsto}}
\label{s:proof_hypsto}
The proof of Proposition \ref{prop:hypsto} is split into several parts, following all the  assumptions we need to verify accordingly. In particular, we establish below Propositions \ref{prop:hypsto1}, \ref{prop:hypsto2}, \ref{prop:hypsto3} and \ref{prop:hypsto4}.

\begin{proposition}[Assumption $\hat{\mathcal{H}}^+_{\varepsilon,a}$]
\label{prop:hypsto1}
For any $a>0$, set $c_a = \MM\sqrt{2a}$ in Equation~\eqref{eq:def_M_N_P}.
Then for any integer $k$, the iterate $(\nukps,\hat\nu_{k+1})$ of Algorithm
\ref{algo:SCPGD_pro}  satisfies for any $A\subset \mathcal{X}$:
$$
\mathbb{E}\left[  \hat\nu_{k+1}(A\cap \lbrace J'_{\hat\nu_{k^+}} < 0 \rbrace)\, \vert \mathfrak{F}_{k}^+\right]  \ge \varepsilon_{k}\frac{\lambda(A\cap \lbrace J'_{\hat\nu_{k^+}} < 0 \rbrace)}{\lambda(\mathcal{X})}
-  \varepsilon_k  m_k^{-a}
$$
\end{proposition}

\begin{proof}
Consider any integer $k\in \mathbb{N}$ and any measurable set $A \subset \mathcal{X}$, we use the definition of $\hat\nu_{k+1}$:
\begin{align*}
\mathbb{E}&\left[  \hat\nu_{k+1}(A\cap \lbrace J'_{\hat\nu_{k^+}} < 0 \rbrace)\, \vert \mathfrak{F}_{k}^+\right] \\
&
=  \nukkpps(A\cap \lbrace J'_{\hat\nu_{k^+}} < 0 \rbrace) +
\varepsilon_{k} \mathbb{E}\left[
\bm{1}_{\Nkp}(U_{k+1})\delta_{U_{k+1}}(A\cap \lbrace J'_{\hat\nu_{k^+}} < 0 \rbrace)  \, \vert \mathfrak{F}_{k}^+\right], \\
& \geq \varepsilon_{k}
\int_{A\cap \lbrace J'_{\hat\nu_{k^+}} < 0  \rbrace}
\lambda(\mathcal{X})^{-1}dx - \varepsilon_{k} \mathbb{E}\left[
\int_{A\cap \lbrace J'_{\hat\nu_{k^+}} < 0  \rbrace}
\bm{1}_{\Nkp^c}(x)  \lambda(\mathcal{X})^{-1} \text{d}x \right] , \\
& \ge    \varepsilon_{k}\frac{\lambda(A\cap \lbrace J'_{\hat\nu_{k^+}} < 0 \rbrace)}{\lambda(\mathcal{X})} -  \varepsilon_{k} \mathbb{P} \left( J'_{\hat\nu_{k^+}}(U_{k+1}) < 0 \, \text{and} \, \widehat{J'}_{\hat\nu_{k^+}}(U_{k+1}) > c_a \sqrt{\frac{\log m_k}{m_k}} \right),
\end{align*}
where we used the definition of $\Nkp$. Recall in particular that the constant $c_a$ is positive. The key observation is that the sampled point $U_{k+1} \sim \mathrm{Uniform}(\mathcal{X})$ is \emph{independent} of the mini-batch $Z^+_{k+1}$ used to construct the stochastic certificate $\widehat{J'}_{\hat\nu_{k^+}}$. Conditioning on $U_{k+1} = x$ reduces the problem to bounding the deviation of an empirical mean at a single fixed point, where standard Hoeffding's inequality applies without any covering argument. Specifically, for any $x \in \mathcal{X}$:
\begin{equation}\label{eq:pointwise_hoeffding_birth}
    \mathbb{P} \left( \widehat{J'}_{\hat\nu_{k^+}}(x) > c_a \sqrt{\frac{\log m_k}{m_k}} \, \Bigg| \, U_{k+1} = x \right) \leq \exp\left(- \frac{m_k}{2 \MM^2} \left(c_a \sqrt{\frac{\log m_k}{m_k}}\right)^2 \right) = m_k^{-a}.
\end{equation}
Integrating over $U_{k+1}$ and setting $a = \frac{c_a^2}{2\MM^2}$, we obtain:
\begin{align*}
\mathbb{E}\left[  \hat\nu_{k+1}(A\cap \lbrace J'_{\hat\nu_{k^+}} < 0 \rbrace)\, \vert \mathfrak{F}_{k}^+\right]
& \ge \varepsilon_{k}\frac{\lambda(A\cap \lbrace J'_{\hat\nu_{k^+}} < 0 \rbrace)}{\lambda(\mathcal{X})}
-  \varepsilon_k  m_k^{-a},
\end{align*}
with $a=\frac{c_a^2}{2\MM^2}$. No covering of $\mathcal{X}$ is needed, so the dimension $d$ disappears from the birth threshold. This ensures that $\hat{\mathcal{H}}^+_{\varepsilon,a}$ holds with $\cc = \lambda(\mathcal{X})^{-1}$ and $\CC = 1$.
\end{proof}

\begin{proposition}[Assumption $\hat{\mathcal{H}}^{\text{smooth},2}_{\varepsilon}$.]
\label{prop:hypsto3}
For any integer $k$, the iterate $(\nukkpps,\hat\nu_{k+1})$ of Algorithm
\ref{algo:SCPGD_pro}  satisfies:
$$\| J'_{\hat\nu_{k+1}} - J'_{\hat\nu_{k^+}} \|_\infty \leq \CC\varepsilon_{k}.$$
\end{proposition}
\begin{proof}
Using \eqref{eq:Frechet_J}, we have, for any $t\in \mathcal{X}$,
\[
J'_{\hat\nu_{k+1}}(t) - J'_{\hat\nu_{k^+}}(t) = \langle \varphi_t , \Phi (\hat\nu_{k+1} - \hat\nu_{k^+}) \rangle_\mathbb{H}.
\]
Hence, according to our update scheme,
\[
\sup_{t\in \mathcal{X}} |J'_{\hat\nu_{k+1}}(t) - J'_{\hat\nu_{k^+}}(t)| = \sup_{t\in \mathcal{X}} \left| \langle \varphi_t , \Phi \delta_{V_{k+1}}\rangle_\mathbb{H} \hat\nu_{k^+}(V_{k+1}) + \langle \varphi_t , \Phi \delta_{U_{k+1}}\rangle_\mathbb{H} \varepsilon_k \right| \leq \sqrt{2}\epk + \epk \leq \CC \epk,
\]
where we have used the Cauchy-Schwarz Inequality and Assumption ($\mathcal{H}_{\mathcal{P}}$).
\end{proof}

\begin{proposition}[Assumption $\hat{\mathcal{H}}^{\text{smooth},1}_{\varepsilon}$]
\label{prop:hypsto2}
For any integer $k$, the iterate $(\nukkpps,\hat\nu_{k+1})$ of Algorithm
\ref{algo:SCPGD_pro}  satisfies:
$$
\mathbb{E}\left[  J(\hat\nu_{k+1})- J(\nukps) \, \vert\mathfrak{F}_{k}^+ \right]\leq \CC \left( \epk \sqrt{\frac{\log m_k}{m_k}} + \epk^2\right).
$$
\end{proposition}

\begin{proof}
The proof is inspired by the deterministic case, see Section \ref{s:proof_delet}, but we still need to handle the randomness brought by $\Pkp$ and $\Nkp$. We decompose the evolution of $J$ into two terms:
$$
J(\hat\nu_{k+1})-J(\hat\nu_{k^+}) =
J(\hat\nu_{k+1})- J(\nukkpps) + J(\nukkpps) - J(\hat\nu_{k^+}).
$$

\paragraph{Death part: $J(\nukkpps) - J(\hat\nu_{k^+})$.}
Using $\nukkpps - \hat\nu_{k^+} = -\nukps(V_{k+1}) \delta_{V_{k+1}} \mathbf{1}_{\Pkp}(V_{k+1})$:
\begin{align*}
    \lefteqn{J(\nukkpps) - J(\hat\nu_{k^+})}\\
    & = \int J'_{\nukkpps}\, d[\nukkpps -\hat\nu_{k^+}] + \frac{1}{2} \| \Phi(\nukkpps -\hat\nu_{k^+}) \|_\mathbb{H}^2 \\
    & = \Big[ -\widehat{J'}_{\nukps}(V_{k+1})\nukps(V_{k+1}) + \big(\widehat{J'}_{\nukps} - J'_{\nukps}\big)(V_{k+1})\nukps(V_{k+1}) + \tfrac{1}{2}\nukps(V_{k+1})^2 \|\Phi\delta_{V_{k+1}}\|_\mathbb{H}^2 \Big]\\
    &\qquad\times \mathbf{1}_{\Pkp}(V_{k+1}) \\
    & \leq  \varepsilon_k^2 + \sqrt{2}\,\varepsilon_k \max_{1\le j \le p_k}\left|\widehat{J'}_{\nukps}(\hat t^{k^+}_j) - J'_{\nukps}(\hat t^{k^+}_j)\right|,
\end{align*}
since on $\Pkp$ we have $\widehat{J'}_{\nukps}(V_{k+1}) \ge 0$ (so the first term is $\le 0$ and dropped), $\nukps(V_{k+1}) \leq \sqrt{2}\,\varepsilon_k$, and $\|\Phi\delta_{V_{k+1}}\|_\mathbb{H}^2 \leq 1$. The death process only evaluates $\widehat{J'}_{\nukps}$ at the $p_k$ active particle locations $(\hat t^{k^+}_j)_{1\le j \le p_k}$, which are $\mathfrak{F}_k^+$-measurable. By Hoeffding's inequality applied at each fixed particle position and a sub-Gaussian maximal bound over the $p_k \le m_k + p_0$ points:
$$
\mathbb{E}\left[\max_{1\le j \le p_k}\left|\widehat{J'}_{\nukps}(\hat t^{k^+}_j) - J'_{\nukps}(\hat t^{k^+}_j)\right| \, \Big\vert\mathfrak{F}_{k}^+ \right]\leq \CC \sqrt{\frac{\log m_k}{m_k}},
$$
since $\log(p_k) \le \log(m_k + p_0) = \mathcal{O}(\log m_k)$. Hence:
$$ \mathbb{E}\left[ J(\nukkpps) - J(\hat\nu_{k^+}) \, \vert\mathfrak{F}_{k}^+ \right]\leq  \epk^2+ \CC \epk  \sqrt{\frac{\log m_k}{m_k}}.$$

\paragraph{Birth part: $J(\hat\nu_{k+1}) - J(\nukkpps)$.}
Since $\hat\nu_{k+1} - \nukkpps = \epk \mathbf{1}_{\Nkp}(U_{k+1})\,\delta_{U_{k+1}}$:
\begin{align*}
\lefteqn{J(\hat\nu_{k+1})- J(\nukkpps)}\\
& = \int J'_{\nukkpps}\, d[\hat\nu_{k+1} -\nukkpps]  + \tfrac{1}{2} \| \Phi(\hat\nu_{k+1}-\nukkpps) \|_\mathbb{H}^2 \\
& = \epk \mathbf{1}_{\Nkp}(U_{k+1}) \Big[\widehat{J'}_{\nukps}(U_{k+1}) + \big(J'_{\nukps} - \widehat{J'}_{\nukps}\big)(U_{k+1}) + \big(J'_{\nukkpps} - J'_{\nukps}\big)(U_{k+1})\Big]\\
&\qquad + \tfrac{\epk^2}{2}\|\Phi\delta_{U_{k+1}}\|_\mathbb{H}^2 \mathbf{1}_{\Nkp}(U_{k+1}).
\end{align*}
We bound each contribution. By the definition of $\Nkp$ in~\eqref{eq:def_M_N_P}, it holds that 
\[
\widehat{J'}_{\nukps}(U_{k+1})\,\mathbf{1}_{\Nkp}(U_{k+1}) \le c_a \sqrt{(\log m_k)/m_k}\,.
\]
Proposition~\ref{prop:hypsto3} yields $\|J'_{\nukkpps} - J'_{\nukps}\|_\infty \le \CC \epk$, and $\|\Phi\delta_{U_{k+1}}\|_\mathbb{H}^2 \le 1$. Since $U_{k+1}\sim\mathrm{Uniform}(\mathcal{X})$ is independent of the mini-batch, conditioning on $U_{k+1}=x$ and applying pointwise Hoeffding gives
$$\mathbb{E}\big[\, \big|\widehat{J'}_{\nukps}(U_{k+1}) - J'_{\nukps}(U_{k+1})\big| \,\big|\, \mathfrak{F}_k^+\big] \le \frac{\CC}{\sqrt{m_k}}.$$
Taking the conditional expectation and using $1/\sqrt{m_k} \le \sqrt{(\log m_k)/m_k}$ for $m_k \ge 3$:
$$
\mathbb{E}\left[ J(\hat\nu_{k+1})- J(\nukkpps) \, \vert\mathfrak{F}_{k}^+ \right] \leq \CC \left( \epk \sqrt{\frac{\log m_k}{m_k}} + \epk^2 \right).
$$
The finitely many iterates with $m_k \in \{1,2\}$ contribute an $O(1)$ constant absorbed into $\mathfrak{C}$. Combining the birth and death parts establishes $\hat{\mathcal{H}}^{\text{smooth},1}_{\varepsilon}$.
\end{proof}

\begin{proposition}[Assumption $(\hat{\mathcal{H}}_{D})$]
\label{prop:hypsto4}
There exists a large enough constant $\mathfrak{C}$ such that, for any integer $k$, if $\alpha \leq (4\mathfrak{C})^{-1} \wedge \sqrt{8\log 8}\, \MM^{-1}$, then:
$$
\mathbb{E}[J(\nukps) \, \vert \mathfrak{F}_k] \leq J(\nuks) - \frac{\alpha}{2} \|J'_{\nuks}\|_{\nuks}^2 + \mathfrak{C} \left( \frac{\alpha^2}{m_k} + \beta^2 + \frac{\beta}{ m_k}\right).
$$
\end{proposition}

\begin{subequations}
\begin{proof}
\underline{Step 1: One-step evolution and second order term.}
Let $k\in \mathbb{N}^\star$ be fixed. According to \eqref{eq:Frechet_Jprime}:
\begin{equation}
J(\hat\nu_{k^+}) - J(\hat\nu_k) = \int_\cX J_{\hat\nu_k}' d(\hat\nu_{k^+}-\hat\nu_k) + \frac{1}{2} \| \Phi(\hat\nu_{k^+}-\hat\nu_k)\|_\mathds{H}^2.
\label{eq:dvp_J}
\end{equation}
Writing $\hat\nu_{k} = \sum_{j=1}^{p_k} \hat\omega_j^{k} \delta_{\hat{t}_j^{k}}$ and $\hat\nu_{k^+} = \sum_{j=1}^{p_k} \hat\omega_j^{k^+} \delta_{\hat{t}_j^{k^+}}$, we introduce the following measure $\tilde \nu_{k^+} = \sum_{j=1}^{p_k}\hat\omega_{j}^{k^+} \delta_{\hat t_j^k}$. Then, we deduce that:
\begin{eqnarray*}
\| \Phi(\hat\nu_{k^+}-\hat\nu_k)\|_\mathds{H}^2
& = & \| \Phi(\hat\nu_{k^+}-\tilde \nu_{k^+} + \tilde \nu_{k^+} - \hat\nu_k)\|_\mathds{H}^2, \\
& \leq & 2 \| \Phi(\hat\nu_{k^+}-\tilde \nu_{k^+})\|_\mathds{H}^2 + 2\| \Phi(\tilde\nu_{k^+}-\hat\nu_k)\|_\mathds{H}^2.
\end{eqnarray*}
First remark that, 
\begin{eqnarray}
\|\Phi(\hat\nu_{k^+}-\tilde \nu_{k^+})\|_\mathds{H}^2 
& = &  \left\| \sum_{j=1}^{p_k} \hat\omega_j^{k^+} (\varphi_{\hat t_j^{k^+}}-\varphi_{\hat t_j^{k}})\right\|_\mathds{H}^2, \nonumber \\
& \leq & \sum_{j=1}^{p_k} \hat\omega_j^{k^+} \times \sum_{j=1}^{p_k} \hat\omega_j^{k^+} \| \varphi_{\hat t_j^{k^+}} - \varphi_{\hat t_j^{k}} \|_\mathds{H}^2, \nonumber \\
& \leq & \mathfrak{C}_\mathcal{P} \| \hat\nu_{k^+}\|_{\mathrm{TV}} \sum_{j=1}^{p_k} \hat\omega_j^{k^+}\| \hat t_j^{k^+} - \hat t_j^{k} \|^2, \nonumber \\
& \leq &  \mathfrak{C}_\mathcal{P} \| \hat\nu_{k^+}\|_{\mathrm{TV}} \ \beta^2 \sum_{j=1}^{p_k} \hat \omega_j^{k^+} \|\pi_\mathcal{X}(\hat t_j^k,\widehat{\mD_{k}}(\hat t_j^k),\beta)\|^2,
\label{eq:step1a}
\end{eqnarray}
where we have used \eqref{eq:control_kernel_distance_euclidean_distance} and \eqref{eq:up_w_pos_sto}. Similarly
\begin{eqnarray}
\| \Phi(\tilde\nu_{k^+}-\hat\nu_k)\|_\mathds{H}^2
& = & \left\| \sum_{j=1}^{p_k} (\hat \omega_j^{k^+} - \hat\omega_j^{k}) \varphi_{\hat t_j^k} \right\|_\mathds{H}^2, \nonumber \\
& = & \left\| \sum_{j=1}^{p_k} \hat\omega_j^k (e^{-\alpha \widehat{\mJ'_{k}}(\hat t_j^k) }-1) \varphi_{\hat t_j^k} \right\|_\mathds{H}^2, \nonumber \\
& \leq & \sum_{j=1}^{p_k} \hat \omega_{j}^{k} \times \sum_{j=1}^{p_k} \hat\omega_{j}^{k} (e^{-\alpha \widehat{\mJ'_{k}}(\hat t_j^k) }-1)^2, \nonumber
\end{eqnarray}
where the last line comes from the Jensen inequality and from the fact that $\|\varphi_{t} \|_\mathds{H}^2 = 1$ for any $t \in \mathcal{X}$. Moreover, since the random variable $\widehat{\mJ'_{k}}(\vt)$ is bounded for any $\vt \in \mathcal{X}$, 
\begin{align}
\| \Phi(\tilde\nu_{k^+}-\hat\nu_k)\|_\mathds{H}^2 & \leq \CC \| \hat\nu_k \|_{\mathrm{TV}} \alpha^2 \sum_{j=1}^{p_k} \hat\omega_j^k \widehat{\mJ'_{k}}(\hat t_j^k)^2 \nonumber \\
& \leq \CC \| \hat\nu_k \|_{\mathrm{TV}} \alpha^2 \|\widehat{\mJ'_{k}}\|_{\hat\nu_k}^2. \nonumber
\label{eq:step2a}
\end{align}
Taking the conditional expectation, since the particles' weights and positions $(\hat\omega_j^k,\hat t_j^k)_{j \in [p_k]}$ are $\mathfrak{F}_k$-measurable and independent of the $k$-th iteration mini-batch, we can push the expectation inside the sum over the particles:
$$
\mathbb{E}\Big[\|\widehat{\mJ'_{k}}\|_{\hat\nu_k}^2 \, \big\vert \mathfrak{F}_k\Big]
= \sum_{j=1}^{p_k} \hat\omega_j^k \mathbb{E}\left[ |\widehat{\mJ'_{k}}(\hat t_j^k)|^2 \, \big\vert \mathfrak{F}_k \right].
$$
Using the pointwise bias-variance decomposition at each fixed particle position $\hat t_j^k \in \mathfrak{F}_k$:
$$
\mathbb{E}\left[ |\widehat{\mJ'_{k}}(\hat t_j^k)|^2 \, \big\vert \mathfrak{F}_k \right] = |\mJ'_{\nuks}(\hat t_j^k)|^2 + \mathrm{Var}\left( \widehat{\mJ'_{k}}(\hat t_j^k) \, \big\vert \mathfrak{F}_k \right) \leq |\mJ'_{\nuks}(\hat t_j^k)|^2 + \frac{\MM^2}{m_k},
$$
where $\MM^2$ provides a uniform bound on the variance of the stochastic gradient evaluations due to assumption~\eqref{A2}. Plugging this into the second-order term and using the almost sure boundedness of $(\|\nuks\|_{TV})_{k \ge 0}$:
$$
\mathbb{E} \left[\| \Phi(\tilde\nu_{k^+}-\hat\nu_k)\|_\mathds{H}^2 \, \vert \mathfrak{F}_k \right] \leq \CC \alpha^2 \|J'_{\nuks}\|_{\nuks}^2 + \CC \frac{\alpha^2}{m_k}.
$$
Gathering with Equation \eqref{eq:step1a}, we deduce that:
\begin{equation}
\mathbb{E} \left[
\| \Phi(\hat\nu_{k^+}-\hat\nu_k)\|_\mathds{H}^2 \, \vert \mathfrak{F}_k \right]\leq
\CC \alpha^2 \|J'_{\nuks}\|_{\nuks}^2 + \CC \frac{\alpha^2}{m_k}+\CC \beta^2.
\label{eq:step1_final}
\end{equation}

\noindent
\underline{Step 2: Study of the drift.}
We expand the first order term in \eqref{eq:dvp_J} and observe that:
\begin{align*}
    \int_\cX J_{\hat\nu_k}' &d(\hat\nu_{k^+}-\hat\nu_k)  = \sum_{j=1}^{p_k} \left[(\hat\omega_j^{k^+} - \hat\omega_j^k) J_{\hat\nu_k}'(\hat t_j^k)  + \hat\omega_j^k (J_{\hat\nu_k}'(\hat t_j^{k^+})-J_{\hat\nu_k}'(\hat t_j^{k}))\right]  \\
    & \qquad +  \sum_{j=1}^{p_k} (\hat\omega_j^{k^+}-\hat\omega_j^{k})  (J_{\hat \nu_k}'(\hat t_j^{k^+})-J_{\hat\nu_k}'(\hat t_j^{k})), \\
    & =  \sum_{j=1}^{p_k} \left[(\hat \omega_j^{k^+} - \hat\omega_j^k) J_{\hat\nu_k}'(\hat t_j^k)  + \hat\omega_j^k \langle \hat t_j^{k^+} - \hat t_j^{k} , \nabla J_{\hat\nu_k}' (\hat t_j^{k}) \rangle \right] \\ 
    & \qquad +  \sum_{j=1}^{p_k} \left[\hat\omega_j^k \langle \hat t_j^{k^+} - \hat t_j^{k} , \nabla^2 J_{\hat\nu_k}' (\upsilon_j^{k})(\hat t_j^{k^+} - \hat t_j^{k}) \rangle  +
    (\hat\omega_j^{k^+}-\hat\omega_j^{k})  \langle \nabla J_{\hat\nu_k}'(\tilde{\upsilon}_j^{k}), \hat t_j^{k^+} -\hat t_j^{k} \rangle\right],
\end{align*}
where $\upsilon_j^{k}$ and $\tilde{\upsilon}_j^{k}$ are some auxiliary points that belong to $({t}_j^k,{t}_j^{k^+})$ obtained with the help of first and second order  Taylor expansions. Using Proposition C.1 in \cite{FastPartv1}, we get 
\begin{eqnarray*}
\int_\cX J_{\hat\nu_k}' d(\hat\nu_{k^+}-\hat\nu_k)
& \leq & \sum_{j=1}^{p_k} \left[(\hat\omega_j^{k^+} - \hat\omega_j^k) J_{\hat\nu_k}'(\hat t_j^k)  + \hat\omega_j^k \langle \hat t_j^{k^+} - \hat t_j^{k} , \nabla J_{\hat\nu_k}' (\hat t_j^{k}) \rangle \right] \\ 
& & +  \|\nabla^2 J_{\hat\nu_k}' \|_{\infty} \sum_{j=1}^{p_k} \hat\omega_j^k \| \hat t_j^{k^+} - \hat t_j^{k} \|^2 +  \sum_{j=1}^{p_k} |\hat\omega_j^{k^+}-\hat\omega_j^{k}|  \times \| \nabla J_{\hat\nu_k}'\| \| \hat t_j^{k^+} -\hat t_j^{k}\|, \\
& \leq & \sum_{j=1}^{p_k} \left[(\hat\omega_j^{k^+} - \hat\omega_j^k) J_{\hat\nu_k}'(\hat t_j^k)  + \hat \omega_j^k \langle \hat t_j^{k^+} - \hat t_j^{k} , \nabla J_{\hat\nu_k}' (\hat t_j^{k}) \rangle \right] \\
& & + A \sum_{j=1}^{p_k} \left[ \hat\omega_j^k \| \hat t_j^{k^+} - \hat t_j^{k} \|^2+ |\hat \omega_j^{k^+}-\hat\omega_j^{k}|   \| \hat t_j^{k^+} -\hat t_j^{k}\| \right],
\end{eqnarray*}
where 
\[
A := (\|\hat\nu_k\|_{\mathrm{TV}} +\|y\|_\mathbb{H}) \CC_\mathcal{P}.
\]
Using the weights update \eqref{eq:up_w_pos_sto}, we obtain
\begin{eqnarray}
\int_\cX J_{\hat\nu_k}' d(\hat\nu_{k^+}-\hat\nu_k)
& \leq  &\sum_{j=1}^{p_k}   \hat \omega_j^k (e^{-\alpha \widehat{\mJ'_{k}}(\hat t_j^k) }-1) J_{\nuks}'(\hat t_j^k) +   \hat\omega_j^k \langle \hat t_j^{k^+}-\hat t_j^k,  \nabla J_{\nuks}' (\hat t_j^{k}) \rangle \nonumber \\
&  & +  A \sum_{j=1}^{p_k}  \left[ \hat\omega_j^k \|\hat t_j^{k^+}-\hat t_j^k\|^2  +    \hat\omega_j^k \left| e^{-\alpha \widehat{\mJ'_{k}}(\hat t_j^k) }-1\right|  \|\hat t_j^{k^+}-\hat t_j^k\|\right].
\label{eq:step2_inter1}
\end{eqnarray}
The first term of the right hand side of Equation \eqref{eq:step2_inter1} is dealt thanks to Proposition \ref{prop:hoeffding}: a straightforward conditional expectation argument yields:

\begin{align*}
\mathbb{E}\left[
\sum_{j=1}^{p_k}   \hat\omega_j^k (e^{-\alpha \widehat{\mJ'_{k}}(\hat t_j^k) }-1) J_{\nuks}'(\hat t_j^k) \, \vert \mathfrak{F}_k \right]
& \leq - \alpha \|J'_{\nuks}\|_{\nuks}^2 \\
& + \frac{\alpha^2 \|J'_{\nuks}\|_{\infty} \|\nuks\|_{TV}\MM^2}{m_k} e^{\alpha \|J'_{\nuks}\|_{\infty}} + \frac{\alpha^2 \|J'_{\nuks}\|_{\nuks}^2  }{2}  e^{\alpha \|J'_{\nuks}\|_{\infty}}
\end{align*}

We pay a specific attention to the second term of the right hand side  of Equation \eqref{eq:step2_inter1}. Using the generalized projected gradient  and its related properties (\textit{e.g.} Lemma  \ref{lem:Ghadimi-Lan-Zhang}), we get for any $j\in \lbrace 1,\dots, p \rbrace$,

\begin{align*}
  \hat\omega_j^k &\langle \hat t_j^{k^+}-\hat t_j^k,  \nabla J_{\nuks}' (\hat t_j^{k}) \rangle \\
&= - \beta\hat\omega_j^k \left\langle \pi_{\cX}\left(\hat t_j^k, \widehat{\mD_{k}}(\hat t_j^k) ,\beta\right) ,  \widehat{\mD_{k}}(\hat t_j^k)\right\rangle  + \beta \hat\omega_j^k \left\langle \pi_{\cX}\left(\hat t_j^k,\widehat{\mD_{k}}(\hat t_j^k) ,\beta\right)  , \nabla J_{\nuks}' (\hat t_j^{k}) - \widehat{\mD_{k}}(\hat t_j^k) \right\rangle, \\
& \leq - \beta \hat\omega_j^k \left\|\pi_{\cX}\left(\hat t_j^k,\widehat{\mD_{k}}(\hat t_j^k) ,\beta\right) \right\|^2 + \beta \hat\omega_j^k \left\langle \pi_{\cX}\left(\hat t_j^k,\widehat{\mD_{k}}(\hat t_j^k) ,\beta\right)  , \nabla J_{\nuks}' (\hat t_j^{k}) - \widehat{\mD_{k}}(\hat t_j^k) \right\rangle.
\end{align*}
Using the Young inequality, we get  
\[
\hat \omega_j^k \langle \hat t_j^{k^+} -\hat t_j^k,  \nabla J_{\nuks}' (\hat t_j^{k}) \rangle 
   \leq - \frac{\beta}{2} \hat\omega_j^k \left\|\pi_{\cX}\left(\hat t_j^k,\widehat{\mD_{k}}(\hat t_j^k) ,\beta\right) \right\|^2  +2\beta \hat \omega_j^k \left\| \widehat{\mD_{k}}(\hat t_j^k)-\nabla J_{\nuks}' (\hat t_j^{k})\right\|^2, 
\]
where we have used again Lemma \ref{lem:Ghadimi-Lan-Zhang}. Taking the conditional expectation w.r.t. $\mathfrak{F}_k$, we get
\begin{eqnarray*}
\lefteqn{\mathbb{E}\left[\hat\omega_j^k \langle \hat t_j^{k^+} -\hat t_j^k,  \nabla J_{\nuks}' (\hat t_j^{k}) \rangle \big\vert \mathfrak{F}_k \right] }\\
& \leq & -  \frac{\beta}{2} \hat \omega_j^k \mathbb{E} \left[ \left\|\pi_{\cX}\left(\hat t_j^k,\widehat{\mD_{k}}(\hat t_j^k) ,\beta\right) \right\|^2 \big\vert \mathfrak{F}_k \right] + 2\beta\hat \omega_j^k \mathbb{E}\left[ \left\| \widehat{\mD_{k}}(\hat t_j^k)-\nabla J_{\nuks}' (\hat t_j^{k})\right\|^2 \big\vert \mathfrak{F}_k\right]  , \\
& \leq & - \frac{\beta}{2} \hat \omega_j^k \mathbb{E} \left[ \left\|\pi_{\cX}\left(\hat t_j^k,\nabla J_{\nuks}' (\hat t_j^{k}) ,\beta\right) \right\|^2 \big\vert \mathfrak{F}_k \right] + 2\beta\hat \omega_j^k \mathbb{E}\left[ \left\| \widehat{\mD_{k}}(\hat t_j^k)-\nabla J_{\nuks}' (\hat t_j^{k})\right\|^2 \big\vert \mathfrak{F}_k\right] \\
& & - \beta \hat\omega_j^k  \mathbb{E} \left[ \left\langle \pi_{\cX}\left(\hat t_j^k,\nabla J_{\nuks}' (\hat t_j^{k}) ,\beta\right)  , \pi_{\cX}\left(\hat t_j^k,\nabla J_{\nuks}' (\hat t_j^{k}) ,\beta\right)  - \pi_{\cX}\left(\hat t_j^k,\widehat{\mD_{k}}(\hat t_j^k) ,\beta\right) \right\rangle  \big\vert \mathfrak{F}_k \right] \\
& \leq & - \frac{\beta}{4} \hat \omega_j^k \mathbb{E} \left[ \left\|\pi_{\cX}\left(\hat t_j^k,\nabla J_{\nuks}' (\hat t_j^{k}) ,\beta\right) \right\|^2 \big\vert \mathfrak{F}_k \right] +4 \beta\hat \omega_j^k \mathbb{E}\left[ \left\| \widehat{\mD_{k}}(\hat t_j^k)-\nabla J_{\nuks}' (\hat t_j^{k})\right\|^2 \big\vert \mathfrak{F}_k\right].
\end{eqnarray*}
At this step, we can take advantage of the mini-batch step to control the second expectation in the previous inequality. Indeed, according to \eqref{eq:minibatch}, we have
\begin{eqnarray*}
\mathbb{E}\left[ \left\| \widehat{\mD_{k}}(\hat t_j^k)-\nabla J_{\nuks}' (\hat t_j^{k})\right\|^2 \big\vert \mathfrak{F}_k\right]
& = & \frac{1}{m_k^2} \sum_{l=1}^{m_k} \mathbb{E} \left[ \left\| \zeta_{\nuks}(\hat t_j^k,Z_l^k)\right\|^2 \big\vert \mathfrak{F}_k \right], \\
& \leq & \frac{\MM^2}{m_k} \leq  \frac{\CC}{m_k}.
\end{eqnarray*}
This leads to
\[
\mathbb{E}\left[\hat\omega_j^k \langle \hat t_j^{k^+} -\hat t_j^k,  \nabla J_{\nuks}' (\hat t_j^{k}) \rangle \big\vert \mathfrak{F}_k \right]  \leq - \frac{\beta}{4} \hat\omega_j^k \mathbb{E} \left[ \left\|\pi_{\cX}(\hat t_j^k,\nabla J_{\nuks}' (\hat t_j^{k}) ,\beta) \right\|^2 \big\vert \mathfrak{F}_k \right] + \frac{\CC}{m_k}.
\]
Plugging this expression in \eqref{eq:step2_inter1}, using the almost sure TV boundedness and taking the conditional expectation, we get:
\begin{eqnarray}
\mathbb{E}\left[\int_\cX J_{\hat\nu_k}' d(\hat\nu_{k^+}-\hat\nu_k) \big|  \mathfrak{F}_k\right]
& \leq  &\sum_{j=1}^{p_k}   \hat\omega_j^k \mathbb{E} \left[(e^{-\alpha \widehat{\mJ'_{k}}(\hat t_j^k) }-1) \big | \mathfrak{F}_k \right]J_{\nuks}'(\hat t_j^k) +  \mathfrak{C} \frac{\beta}{m_k}\nonumber \\
& & - \frac{\beta}{4} \sum_{j=1}^{p_k}\hat \omega_j^k \mathbb{E} \left[ \left\|\pi_{\cX}(\hat t_j^k,\nabla J_{\nuks}' (\hat t_j^{k}) ,\beta) \right\|^2 \big\vert \mathfrak{F}_k \right]\nonumber\\
&  & +  \CC\sum_{j=1}^p  \mathbb{E}\left[ \left(\hat \omega_j^k \|\hat t_j^{k^+}-\hat t_j^k\|^2  +    \hat \omega_j^k \left| e^{-\alpha \widehat{\mJ'_{k}}(\hat t_j^k) }-1\right|  \|\hat t_j^{k^+}-\hat t_j^k\|\right) \big | \mathfrak{F}_k \right], \nonumber\\
& \leq &  - \alpha \|J'_{\nuks}\|_{\nuks}^2
- \frac{\beta}{4} \sum_{j=1}^{p_k}\hat \omega_j^k \mathbb{E} \left[ \left\|\pi_{\cX}(\hat t_j^k,\nabla J_{\nuks}' (\hat t_j^{k}) ,\beta) \right\|^2 \big\vert \mathfrak{F}_k \right] \nonumber\\
& & + \mathfrak{C} \left( \alpha^2  \|J'_{\nuks}\|_{\nuks}^2 + \frac{\alpha^2}{m_k} + \beta^2 + \frac{\beta}{m_k}\right), \label{eq:step2_final}
\end{eqnarray}
where we have used the almost sure boundedness of $(\|\nuks\|_{TV})_{k \ge 0}$ and a large enough $\mathfrak{C}$.
Considering now Equations \eqref{eq:step1_final} and \eqref{eq:step2_final}, using $\alpha \le (4\mathfrak{C})^{-1}$, the quadratic drift term $\mathfrak{C}\alpha^2 \|J'_{\nuks}\|_{\nuks}^2$ is absorbed by half of the linear descent $-\alpha \|J'_{\nuks}\|_{\nuks}^2$, and we finally obtain that:
\begin{align*}
\mathbb{E}\left[
J(\hat\nu_{k^+}) - J(\hat\nu_k)\big|  \mathfrak{F}_k\right] &\leq
-\frac{\alpha}{2} \|J'_{\nuks}\|_{\nuks}^2 - \frac{\beta}{4} \sum_{j=1}^{p_k}\hat\omega_j^k \mathbb{E} \left[ \left\|\pi_{\cX}(\hat t_j^k,\nabla J_{\nuks}' (\hat t_j^{k}) ,\beta) \right\|^2 \big\vert \mathfrak{F}_k \right] \\ &+ \CC \left(\frac{\alpha^2}{m_k} + \beta^2 + \frac{\beta}{m_k} \right).
\end{align*}
This bound delivers a stronger conclusion than stated in the proposition, since it retains the negative projected-gradient term $-\tfrac{\beta}{4}\sum_j\hat\omega_j^k \mathbb{E}[\|\pi_{\cX}(\hat t_j^k,\nabla J_{\nuks}'(\hat t_j^k),\beta)\|^2 \mid \mathfrak{F}_k]$; dropping it recovers exactly the statement of Proposition~\ref{prop:hypsto4}.
\end{proof}
\end{subequations}

\subsection{One-step analysis \label{s:proof_upper-bound-sto}}

We introduce below the stochastic counterpart of the set $\Thetakm$ used in the deterministic approach and defined in Equation \eqref{eq:min}, that is denoted as $\Thetakms$ and is given by:

$$\Thetakms := \left\{ J'_{\hat\nu_{k-1^+}} \leq \frac{\hat v_{k-1^+}}{2}\right\} \quad \mathrm{with} \quad \hat v_{k-1^+} = \min (J'_{\hat\nu_{k-1^+}}) \wedge 0. $$

\begin{proposition}\label{prop:Decroissance_sto_1}
Assume that the sequence $(\nuks,\nukps)_{k \ge 1}$ satisfies     $(\hat{\mathcal{H}}_D)$, $(\hat{\mathcal{H}}^+_{\varepsilon,a})$, $(\hat{\mathcal{H}}^{\text{smooth},1}_{\varepsilon})$, $(\hat{\mathcal{H}}^{\text{smooth},2}_{\varepsilon})$ and $(\hat{\mathcal{H}}^\infty_{\mathrm{TV}})$.
Then, for any $k \ge 1$:
\begin{eqnarray*}
\lefteqn{\mathbb{E}  \left[ J(\hat{\nu}_{k^+}) \big\vert \mathfrak{F}_{k} \right] - J(\hat\nu_k)  }\\
  & \leq & -  \cc \alpha  \varepsilon_{k-1} \left(\left[  |\vks|^2 - \CC\varepsilon_{k-1}^2 \right]\vee 0 \right) \left(|\vks |^d - \CC m_{k-1}^{-a}\right) +\mathfrak{C}\left(
   \frac{\alpha^2}{m_k}
   +\beta^2 +\frac{\beta}{m_k}\right).
\end{eqnarray*}
\end{proposition}

\begin{proof}
First, we use Assumption $(\hat{\mathcal{H}}_D)$ and obtain that: 
\begin{eqnarray*}
\mathbb{E}  \left[ J(\hat{\nu}_{k^+}) \big\vert \mathfrak{F}_{k} \right] - J(\hat\nu_k)  
& \leq &  - \frac{\alpha}{2}  \|J'_{\hat\nu_k}\|^2_{\hat\nu_k} + \mathfrak{C}\left(\frac{\alpha^2}{m_k}+\beta^2 +\frac{\beta}{m_k}\right), \\
& = & - \frac{\alpha}{2}  \int_{\mathcal{X}} |J'_{\nuks}|^2 \text{d}\nuks + \mathfrak{C}\left(\frac{\alpha^2}{m_k}+\beta^2  +\frac{\beta}{m_k}\right), \\
& \leq & -  \frac{\alpha}{2}  \int_{\Thetakms} |J'_{\nuks}|^2 \text{d}\nuks + \mathfrak{C}\left(\frac{\alpha^2}{m_k}+\beta^2  +\frac{\beta}{m_k}\right)\\
\end{eqnarray*}
We apply the Young inequality $|J'_{\nuks}(t)|^2 \ge \frac{1}{2} |J'_{\nukms}(t)|^2 - |J'_{\nuks}(t)-J'_{\nukms}(t)|^2$ to get, according to $(\hat{\mathcal{H}}^{\text{smooth},2}_{\varepsilon})$,
$$
\forall t \in \Thetakms, \qquad  
|J'_{\nuks}(t)|^2 \ge \frac{1}{2}|J'_{\nukms}(t)|^2 - \CC\varepsilon_{k-1}^2 \ge \frac{|\vks|^2}{8} - \CC\varepsilon_{k-1}^2.
$$
Thanks to the positivity of $\alpha  \int_{\Thetakms} |J'_{\nuks}|^2 \text{d}\nuks$, we then get:
$$
\mathbb{E}  \left[ J(\hat{\nu}_{k^+}) \big\vert \mathfrak{F}_{k} \right]-J(\nuks) \leq -  \frac{\alpha}{2} \left(\left[  \frac{|\vks|^2}{8} - \CC\varepsilon_{k-1}^2 \right]\vee 0 \right) \nuks(\Thetakms) +\mathfrak{C}\left(\frac{\alpha^2}{m_k}+\beta^2 +\frac{\beta}{m_k}\right).
$$
The term $\vks$ is $\mathfrak{F}_{k-1}^+$ measurable, and since $\mathfrak{F}_{k-1}^+ \subset \mathfrak{F}_k$, we get
\begin{align}
\lefteqn{\mathbb{E}\left[ J(\nukps)-J(\nuks) \, \vert \mathfrak{F}_{k-1}^+\right] }\nonumber \\
& = \mathbb{E} \left[ \mathbb{E}\left[ J(\nukps)-J(\nuks) \, \vert \mathfrak{F}_{k}\right] \, \vert \mathfrak{F}_{k-1}^+ \right] \nonumber\\
& \leq -  \frac{\alpha}{2} \left(\left[  \frac{|\vks|^2}{8} - \CC \varepsilon_{k-1}^2 \right]\vee 0 \right) \mathbb{E}\left[  \nuks(\Thetakms)\, \vert \mathfrak{F}_{k-1}^+\right]+\mathfrak{C}\left(\frac{\alpha^2}{m_k}+\beta^2 +\frac{\beta}{m_k}\right) ,\nonumber\\
& \leq -   \frac{\alpha}{2} \varepsilon_{k-1} \left(\left[  \frac{|\vks|^2}{8} - \CC \varepsilon_{k-1}^2\right]\vee 0 \right) \left[ \cc \lambda(\Thetakms) - \CC m_{k-1}^{-a}\right] \nonumber\\
&+\mathfrak{C}\left(\frac{\alpha^2}{m_k}+\beta^2 +\frac{\beta}{m_k}\right),\label{eq:intermediaire}
\end{align}
where we have used Assumption $(\hat{\mathcal{H}}^+_{\varepsilon})$ at iteration $k-1$ for the last inequality.  The last part of the proof is very similar to the one displayed in the deterministic case. First, recall that  $J'_{\nu}$ is a $\mathfrak{L}(\nu)$-Lipschitz function and thanks to the boundedness of the TV norm stated by $\hat{\mathcal{H}}_{\mathrm{TV}}$, we know that for any $k$, $J'_{\hat\nu_{k-1^+}}$ is $\Lip$-Lipschitz.
We then get:
\begin{align*}
\forall k \ge 1: \qquad 
    |x-\argmin J'_{\nukms}| \leq \frac{|\vks|}{2 \Lip} & 
         \Longrightarrow |J'_{\nukms}(x)-\min(J'_{\nukms})| \leq \Lip \times \frac{|\vks|}{2\Lip} \\
         &  \Longrightarrow J'_{\nukms}(x) \leq \frac{\vks}{2}\\
& \Longrightarrow x \in \Thetakms.         
\end{align*}
This leads to the $\mathfrak{F}_{k-1}^+$-measurable inequality:
$$
\lambda(\Thetakms) \ge \left| \frac{\vks}{2 \Lip}\right|^d.
$$
Using this lower bound in Equation \eqref{eq:intermediaire}
\begin{eqnarray*}
\lefteqn{\mathbb{E}\left[ J(\nukps) - J(\nuks) \, \vert \mathfrak{F}_{k-1}^+\right]}\\
 & \leq & -  \cc \alpha  \varepsilon_{k-1} \left(\left[  |\vks|^2 - \CC\varepsilon_{k-1}^2 \right]\vee 0 \right) \left(|\vks |^d - \CC m_{k-1}^{-a}\right) +\mathfrak{C}\left(
 \frac{\alpha^2}{m_k}+\beta^2 +\frac{\beta}{m_k}\right).
\end{eqnarray*}
\end{proof}

We introduce $\Deltakps$ that quantifies the amount of decrease on $J$ through the transportation map:
$$
\Deltakps = J(\nuks)-J(\nukps).
$$
As the evolution is now randomized, we emphasize that $\Deltakps$ is not necessarily greater than~$0$.
\begin{proposition}
Assume that the sequence $(\nuks,\nukps)_{k \ge 1}$ satisfies     $(\hat{\mathcal{H}}_D)$, $(\hat{\mathcal{H}}^+_{\varepsilon,a})$, $(\hat{\mathcal{H}}^{\text{smooth},1}_{\varepsilon})$, $(\hat{\mathcal{H}}^{\text{smooth},2}_{\varepsilon})$ and $(\hat{\mathcal{H}}^\infty_{\mathrm{TV}})$.
Then, for any $k \ge 1$, we have the $\mathfrak{F}_k$-measurable inequality:
\begin{eqnarray*} 
J(\nuks)-J^\star
& \leq & \CC \max\left( \alpha^{-1/2}\left[  \frac{\mathbb{E}[\Deltakps \, \vert \mathfrak{F}_{k}]}{ \varepsilon_{k-1}} \right]^{\frac{1}{2+d}} ;
\sqrt{\frac{\alpha}{m_k}};
\frac{\beta}{\sqrt{\alpha}} ; \sqrt{\frac{\beta}{\alpha m_k}}; \varepsilon_{k-1}  ; m_k^{-a/d} ; \right.\\
& & \left.
\left(\frac{\alpha}{m_k\epkm}\right)^{\frac{1}{2+d}};
\left(\frac{\beta^2+\frac{\beta}{m_k}}{\alpha \epkm}\right)^{\frac{1}{2+d}} \right).
\end{eqnarray*}
\label{prop:enfer_sto}
\end{proposition}

\begin{subequations}
\begin{proof} 
Thanks to the convexity of $J$, we have:
\begin{equation}
J(\nuks) - J^\star \leq  \int_\mathcal{X} J_{\nuks}' \text{d}(\nuks - \nu^\star).
\label{eq:cvgs1}
\end{equation}

\noindent
We now bound separately the two terms in the r.h.s. of \eqref{eq:cvgs1} and begin with the 
term $\int J'_{\nuks} \text{d} \nuks$.
The Cauchy-Schwarz inequality and Assumption \eqref{eq:HTVC} yield:
$$\left| \int_\mathcal{X} J_{\nuks}' \text{d}\nuks\right| \leq \left[ \|\nuks\|_{\mathrm{TV}} \int_\mathcal{X} |J_{\nuks}'|^2 \text{d}\nuks \right]^{1/2} \leq \CC  \|J'_{\nuks}\|_{\nuks}.$$
Thanks to Assumption $(\hat{\mathcal{H}}_D)$, we have:
$$ \frac{\alpha}{2} \|J'_{\nuks}\|_{\nuks}^2 \leq \mathbb{E}[\Deltakps \, \vert \mathfrak{F}_k]  + \mathfrak{C} %
\left(\frac{\alpha^2}{m_k}+ \beta^2 + \frac{\beta}{ m_k}\right),$$
which in turn implies that
\begin{equation}
\left| \int_\mathcal{X} J_{\nuks}' \text{d}\nuks\right| \leq  \frac{\CC}{\sqrt{\alpha}} \sqrt{ \mathbb{E}[\Deltakps \, \vert \mathfrak{F}_k] + \mathfrak{C}\left(\frac{\alpha^2}{m_k}+\beta^2+\frac{\beta}{m_k}\right)}.
\label{eq:cvgs2}
\end{equation}

\noindent
The second integral in \eqref{eq:cvgs1} is dealt with Assumption $(\hat{\mathcal{H}}^{\text{smooth},2}_{\varepsilon})$ leading first to the bound
$$\int_\mathcal{X} J_{\nuks}'\text{d}\nu^\star = \int_\mathcal{X} J_{\nukms}'\text{d}\nu^\star + \int_\mathcal{X} (J_{\nuks}' - J_{\nukms}')\text{d}\nu^\star \geq (\vks - \CC \varepsilon_{k-1}) \| \nu^\star\|_{\mathrm{TV}}. $$
Then, two different situations may occur according to the value of $\vks$.

\begin{itemize}
\item \textbf{$1^\mathrm{st}$ case}: $\vks \leq - \left(2 \CC \varepsilon_{k-1}\vee (2 \CC)^{1/d} m_k^{-a/d}\right) $. Then, we get from the previous bound that
$$ \int_\mathcal{X} J_{\nuks}'\text{d}\nu^\star \geq \frac{3}{2} \vks \| \nu^\star\|_{\mathrm{TV}}.$$
Simultaneously, as soon as $\vks \leq - 2 \CC \varepsilon_{k-1} $ and $|\vks|^d \geq 2 \CC m_k^{-a}$, Proposition~\ref{prop:Decroissance_sto_1} implies that:
\begin{eqnarray}
\mathbb{E}\Big[ -\Deltakps \, \vert \mathfrak{F}_{k}\Big] 
& \leq & - \cc \alpha \varepsilon_{k-1} \left(\left[  |\vks|^2 - \CC\varepsilon_{k-1}^2 \right]\vee 0 \right) \left(|\vks |^d-\CC m_{k-1}^{-a}\right)\notag\\
& &\quad +\mathfrak{C}\left(\frac{\alpha^2}{m_k}+\beta^2+ \frac{\beta}{m_k}\right), \nonumber \\
& \leq &   - \cc \alpha \varepsilon_{k-1}  |\vks|^{2+d} +\mathfrak{C}\left(\frac{\alpha^2}{m_k}+\beta^2+\frac{\beta}{m_k} \right).
\label{eq:iter-sto-1}
\end{eqnarray}
The last inequality implies that 
\[
|\vks|^{2+d} \leq \CC \frac{\left(\mathbb{E}\left[ \Deltakps \, \vert \mathfrak{F}_{k}\right]+\mathfrak{C}\left(\frac{\alpha^2}{m_k}+\beta^2 +\frac{\beta}{m_k}  \right)\right)}{\alpha  \varepsilon_{k-1}}.
\]
Hence, we deduce from these computations that:

$$
\int_\mathcal{X} J_{\nuks}'\text{d}\nu^\star  \ge - \CC 
\left[\frac{\left(\mathbb{E}\left[ \Deltakps \, \vert \mathfrak{F}_{k}\right]+\mathfrak{C}\left(\frac{\alpha^2}{m_k}+\beta^2 +\frac{\beta}{m_k} \right)\right)}{\alpha \varepsilon_{k-1}} \right]^{\frac{1}{2+d}}.
$$
\item \textbf{$2^\mathrm{nd}$ case:} $\vks \geq
- \left(2 \CC \varepsilon_{k-1}\vee (2 \CC)^{1/d} m_k^{-a/d}\right)$. In such a situation, we immediately have from Assumption $\hat{\mathcal{H}}_{\varepsilon}^{\text{smooth},2}$ that:
$$
\int_\mathcal{X} J_{\nuks}'\text{d}\nu^\star \geq 
(\vks- \CC \varepsilon_{k-1}) \| \nu^\star\|_{\mathrm{TV}}
\geq - \CC  (\varepsilon_{k-1} \vee m_k^{-a/d}).
$$
\end{itemize}

\vspace{0.5cm}
\noindent
Regardless of the value of $\vks$, we then get the almost sure inequality:
\begin{equation}
\int_\mathcal{X} J_{\nuks}'\text{d}\nu^\star \geq - \CC \max \left( \varepsilon_{k-1} ; m_k^{-a/d} ;\left[\frac{\left(\mathbb{E}\left[ \Deltakps \, \vert \mathfrak{F}_{k}\right]+\mathfrak{C}\left(\frac{\alpha^2}{m_k}+\beta^2 +\frac{\beta}{m_k}\right)\right)}{\alpha \varepsilon_{k-1}} \right]^{\frac{1}{2+d}}\right).
\label{eq:cvg3}
\end{equation}

Gathering Equations \eqref{eq:cvgs1}, \eqref{eq:cvgs2} and \eqref{eq:cvg3}, we deduce that:
\begin{eqnarray*} 
\lefteqn{J(\nuks)-J^\star }\\
& \leq &  \frac{\CC}{\sqrt{\alpha}} \left(\mathbb{E}[\Deltakps \, \vert \mathfrak{F}_{k}] + \mathfrak{C}\left(\frac{\alpha^2}{m_k}+\beta^2+\frac{\beta}{m_k}\right) \right)^{1/2}\\
& & \hspace{1cm} +\CC \max \left(  \varepsilon_{k-1} ;   m_k^{-a/d};\left[\frac{\left(\mathbb{E}\left[ \Deltakps \, \vert \mathfrak{F}_{k}\right]+\mathfrak{C}\left(\frac{\alpha^2}{m_k}+\beta^2 +\frac{\beta}{m_k}\right)\right)}{\alpha  \varepsilon_{k-1}} \right]^{\frac{1}{2+d}}\right)\\
& \leq & \CC \max\left( \left[  \frac{\mathbb{E}[\Deltakps \, \vert \mathfrak{F}_{k}]}{\alpha \varepsilon_{k-1}} \right]^{\frac{1}{2+d}} ; \left[   \frac{\mathbb{E}[ \Deltakps \, \vert \Fk]}{\alpha } \right]^{\frac{1}{2}} ;
\sqrt{\frac{\alpha}{m_k}};
\frac{\beta}{\sqrt{\alpha}} ; \sqrt{\frac{\beta}{\alpha m_k}}; \varepsilon_{k-1}  ; m_k^{-a/d} ; \right.\\
& & \left.
\left(\frac{\alpha}{m_k\epkm}\right)^{\frac{1}{2+d}};
\left(\frac{\beta^2+\frac{\beta}{m_k}}{\alpha \epkm}\right)^{\frac{1}{2+d}} \right).
\end{eqnarray*}
Since $(\Deltakps)_{k \ge 1}$ is a bounded sequence, we can verify that the last inequality is translated into:
\begin{eqnarray*} 
J(\nuks)-J^\star
& \leq & 
\CC \max\left( \alpha^{-1/2}\left[  \frac{\mathbb{E}[\Deltakps \, \vert \mathfrak{F}_{k}]}{ \varepsilon_{k-1}} \right]^{\frac{1}{2+d}} ;
\sqrt{\frac{\alpha}{m_k}};
\frac{\beta}{\sqrt{\alpha}} ; \sqrt{\frac{\beta}{\alpha m_k}}; \varepsilon_{k-1}  ; m_k^{-a/d} ; \right.\\
& & \left.
\left(\frac{\alpha}{m_k\epkm}\right)^{\frac{1}{2+d}};
\left(\frac{\beta^2+\frac{\beta}{m_k}}{\alpha \epkm}\right)^{\frac{1}{2+d}} \right).
\end{eqnarray*}
\end{proof}
\end{subequations}

\subsection{Proof of Theorem \ref{thm:cvgce_sto_all}}
\label{s:cvgce_sto_all}

\begin{proof}%
In what follows, we set:
$$ F(k) = J(\nuks) - J(\nu^\star) \quad \forall k\in \mathbb{N}^\star.$$
Let $k \geq 1$ be fixed. Then, we know from Assumption $(\widehat{\mathcal{H}}_{\varepsilon}^{\text{smooth},1})$
\begin{align}
\mathbb{E}\left[ 
F(k) - F(k+1) \, \vert \mathfrak{F}_k^+ \right]& =  \mathbb{E}\left[ 
J(\nuks) - J(\nukpps) \, \vert \mathfrak{F}_k^+ \right]  \nonumber\\
& =  J(\nuks) - J(\nukps) + 
\mathbb{E}\left[ J(\nukps) - J(\nukpps) \, \vert \mathfrak{F}_k^+ \right]  \nonumber\\
& \geq  \hat\Delta_{k^+} - \CC \left( \epk^2 + \epk \sqrt{\frac{\log m_k}{m_k}}\right).
\label{eq:mino_deltak}
\end{align}

Applying Proposition \ref{prop:enfer_sto}, we get
\begin{align*}
\cc F(k)^{2+d}
 &\leq  
  \alpha^{-\frac{2+d}{2}}\left[  \frac{\mathbb{E}[\Deltakps \, \vert \mathfrak{F}_{k}]}{ \varepsilon_{k-1}} \right] +
\left(\frac{\alpha}{m_k}\right)^{\frac{2+d}{2}} +
\left( \frac{\beta}{\sqrt{\alpha}}\right)^{2+d} + \left(\frac{\beta}{\alpha m_k}\right)^{\frac{2+d}{2}} + \epkm^{2+d}  + m_k^{-\frac{(2+d)a}{d}} +\\
&
\frac{\alpha}{m_k\epkm}+
\frac{\beta^2+\frac{\beta}{m_k}}{\alpha \epkm} .
\end{align*}
Then, using Equation \eqref{eq:mino_deltak}, the tower rule and then
multiplying each term by $\epkm \alpha^{\frac{2+d}{2}}$, we obtain:
\begin{align*}
\mathbb{E}[F(k) - F(k+1) \vert \mathfrak{F}_k ] &  = \mathbb{E}[\mathbb{E}[F(k) - F(k+1) \vert \mathfrak{F}_k^+ ] \, \vert \mathfrak{F}_k]\\
&\geq
\mathbb{E}[ \Deltakps \vert \mathfrak{F}_k ] - \CC \left( \epk^2 + \epk \sqrt{\frac{\log m_k}{m_k}}\right) \\
& \ge  \cc \alpha^{\frac{2+d}{2}} \epkm F(k)^{2+d} - s_k,
\end{align*}
where $s_k$ is given by:
\begin{align*}s_k& =
\frac{\alpha^{2+d} \epkm}{m_k^{\frac{2+d}{2}}}
+ \alpha^{\frac{2+d}{2}} \epkm \left(\frac{\beta}{\sqrt{\alpha}}\right)^{2+d}
+ \epkm \left(\frac{\beta}{m_k}\right)^{\frac{2+d}{2}}
+\epkm^{3+d} \alpha^{\frac{2+d}{2}}
+\epkm \alpha^{\frac{2+d}{2}} m_k^{\frac{-(2+d)a}{d}}\\
&
+\frac{\alpha^{\frac{4+d}{2}}}{m_k} + \alpha^{d/2}(\beta^2+\frac{\beta}{m_k})
+ \epk^2 +  \epk \sqrt{\frac{\log m_k}{m_k}}.
\end{align*}
Taking the global expectation on both side of the inequality, we get:
\[
\mathbb{E}[F(k)] - \mathbb{E}[F(k+1) ] \geq \cc \alpha^{\frac{2+d}{2}} \epkm \mathbb{E}[F(k)^{2+d}] - s_k.
\]
We  finally fix now a final horizon $K$ and use a telescopic sum argument to obtain
\[
\mathbb{E}[F(1)] - \mathbb{E}[F(K+1) ] \geq  \sum_{k=1}^{K} \left(\cc \alpha^{\frac{2+d}{2}}\epkm \mathbb{E}[F(k)^{2+d}] - s_k\right),
\]
which can be rewritten as
\[
\mathbb{E}[F(K+1) ] + \cc \alpha^{\frac{2+d}{2}} \sum_{k=1}^{K} \epkm \mathbb{E}[F(k)^{2+d}]  \leq \mathbb{E}[F(1)] + \sum_{k=1}^{K} s_k.
\]
Using $\hat{\rho}_{K}$ as the minimal value of $F(k)_{1 \leq k \leq K}$, this last inequality implies that 
\[
  \mathbb{E}[\hat{\rho}_{K}^{2+d}]  \leq \CC \frac{\mathbb{E}[F(1)] +  \sum_{k=1}^{K} s_k}{ 
  \alpha^{\frac{2+d}{2}}\sum_{k=1}^{K} \epkm}.
\]
Using the Jensen Inequality, we finally obtain 
\begin{equation}
  \mathbb{E}[\hat{\rho}_{K}]  \leq \CC \left(\frac{\mathbb{E}[F(1)] + \sum_{k=1}^{K} s_k}{ \alpha^{\frac{2+d}{2}} \sum_{k=1}^{K} \epkm}\right)^{\frac{1}{2+d}}.
  \label{eq:borne_cvgce_sto}
\end{equation}
We now consider the final tuning of our parameters to optimize the upper bound obtained in~\eqref{eq:borne_cvgce_sto}. For the sake of simplicity, we restrict our parametrization to constant step-size sequences that depend on the final horizon $K$ of simulation:
\textit{i.e.} $\alpha, \beta, \epk$ and $m_k$ are chosen constant with a value that only depends on $K$ (the final horizon of simulation).
A careful inspection of the terms involved in \eqref{eq:borne_cvgce_sto} shows that $\varepsilon$ needs to balance $\frac{1}{\alpha^{1+d/2} \varepsilon K}$ and $\frac{\varepsilon}{\alpha^{1+d/2}}$. A straightforward argument yields
$$
\forall k \in [1,K] \qquad \epk = \frac{1}{\sqrt{K}}.
$$
Then, we observe that the step-size $\alpha$ of the push-forward has to be chosen small enough (to guarantee the descent property stated in $\widehat{\mathcal{H}}_D$) but has to be lower bounded and cannot be chosen arbitrarily small.
Finally, $\beta$ has to be chosen small enough to make the biggest term $\frac{\beta^2}{\alpha \varepsilon}$ smaller than $\frac{1}{\alpha^{1+d/2} \sqrt{K}}$. We then deduce that:
$$
\beta \leq \frac{1}{\sqrt{K} \alpha^{d/4}}.
$$
At last, we setup the mini-batch size $m$. Thanks to the pointwise bounds established in Propositions~\ref{prop:hypsto1}, \ref{prop:hypsto2} and~\ref{prop:hypsto4}, the dominant mini-batch--dependent term in $s_k$ is now $\varepsilon \sqrt{\frac{\log m}{m}}$ (from the birth/death process), while the descent lemma contributes only $\frac{\alpha^{(4+d)/2}}{m}$. Setting
$$m=K$$
makes both terms of order $\mathcal{O}(\sqrt{\log K}/K)$ or smaller, so that $\sum_{k=1}^K s_k = \mathcal{O}(\sqrt{\log K})$.
According to these several choices, we then obtain the \textit{global} convergence rate of our horizon dependent sequence:
\begin{equation*}
  \mathbb{E}[\hat{\rho}_{K}]  \leq \CC \left( \alpha^{-(1+d/2)} \sqrt{\frac{\log K}{K}}\right)^{\frac{1}{2+d}} \leq \CC \alpha^{-1/2} \left(\frac{\log K}{K}\right)^{\frac{1}{2(2+d)}}.
\end{equation*}
This ends the proof of our final result.
\end{proof}

\begin{proof}[Proof of Theorem~\ref{thm:horizon_free}]
We start from the general bound~\eqref{eq:borne_cvgce_sto} with the iteration-dependent schedules $m_k=k\vee 1$, $\varepsilon_k = \min(\alpha,1/\sqrt{k\vee 1})$ and $\beta_k = 1/(k\vee 1)$. Let $k_\alpha := \lceil 1/\alpha^2\rceil$, so that $\varepsilon_k = \alpha$ for $k \le k_\alpha$ and $\varepsilon_k = 1/\sqrt{k}$ for $k > k_\alpha$. The first $k_\alpha$ iterations contribute only an additive $\mathcal{O}(1)$ to all sums and are absorbed into $\CC$.

\paragraph{The saturated regime $k\le k_\alpha$ is negligible in $K$.}
We make this $\mathcal{O}(1)$ statement precise. By Section~\ref{sec:role_of_alpha}, $\alpha$ is fixed once and for all from intrinsic problem data; in particular, $\alpha$ does not depend on $K$. Hence the threshold $k_\alpha = \lceil 1/\alpha^2\rceil$ is itself a constant independent of the horizon $K$, and although it may be large when $\alpha$ is small, it does not grow with~$K$.

On the saturated range $\{1,\dots,k_\alpha\}$ each summand entering $s_k$ in~\eqref{eq:borne_cvgce_sto} is uniformly bounded by a constant $\CC(\alpha,d,\bm H,\bm G,\MM)$ depending only on intrinsic data: indeed $\varepsilon_k=\alpha$ is constant, $\beta_k\leq \beta_1=1$, $m_k\geq 1$, and the assumed bounds in~\eqref{A2} together with $\HTVC$ provide $K$-independent uniform controls on every factor (compare the term-by-term estimates below). Therefore the total contribution of these terms is bounded by
\[
\sum_{k=1}^{k_\alpha} s_k \;\leq\; k_\alpha\,\CC(\alpha,d,\bm H,\bm G,\MM) \;=\; \mathcal{O}(1) \quad\text{in } K,
\]
which is additive, $K$-independent, and absorbed into the generic constant $\CC$ of~\eqref{eq:horizon_free_rate}. For the lower bound on $\sum_{k=1}^{K}\varepsilon_{k-1}$ established below, the saturated terms are nonnegative and may simply be dropped; the dominant $\sqrt{K}$ growth comes from the regime $k>k_\alpha$, which is precisely why the assertion requires $K\geq 4 k_\alpha$. Equivalently, the theorem assumes $K\geq 4/\alpha^2$, ensuring both that $K$ exceeds $k_\alpha$ enough to make the lower bound effective and that the saturated regime is dominated by the $\sqrt{K}$ contribution.

\paragraph{Lower bound on $\sum \varepsilon_{k-1}$.}
For $K \ge k_\alpha + 4$:
$$
\sum_{k=1}^{K} \varepsilon_{k-1} \;\geq\; \sum_{j=k_\alpha}^{K-1} \frac{1}{\sqrt{j}} \;\geq\; \int_{k_\alpha}^{K} x^{-1/2}\, \mathrm{d}x \;=\; 2\bigl(\sqrt{K}-\sqrt{k_\alpha}\bigr) \;\geq\; \sqrt{K},
$$
provided $K \ge 4 k_\alpha$.

\paragraph{Upper bound on $\sum s_k$.}
We bound each group of terms in $s_k$ separately. The dominant contributions are:
\begin{itemize}
\item \emph{Birth/death term:} $\displaystyle\sum_{k=1}^{K} \varepsilon_k \sqrt{\frac{\log m_k}{m_k}} \leq \sum_{k=1}^{K} \frac{\sqrt{\log k}}{k} \leq \frac{2}{3}(\log K)^{3/2} + \CC,$
using the integral bound $\int_1^K \frac{\sqrt{\log x}}{x}\,\mathrm{d}x = \frac{2}{3}(\log K)^{3/2}$.
\item \emph{Exploration term:} $\displaystyle\sum_{k=1}^{K} \varepsilon_k^2 \leq \sum_{k=1}^{K} \frac{1}{k} = \mathcal{O}(\log K)$.
\item \emph{Descent lemma term:}  $\displaystyle\sum_{k=1}^{K} \frac{\alpha^{(4+d)/2}}{m_k} = \alpha^{(4+d)/2}\sum_{k=1}^{K}\frac{1}{k} = \mathcal{O}(\log K)$.
\item \emph{Regularization term:} $\displaystyle\sum_{k=1}^{K} \alpha^{d/2}\left(\beta_k^2+\frac{\beta_k}{m_k}\right) = \alpha^{d/2}\sum_{k=1}^{K}\frac{2}{k^2} = \mathcal{O}(1)$.
\item \emph{Hoeffding term:} $\displaystyle\sum_{k=1}^{K} \epkm \alpha^{(2+d)/2} m_k^{-(2+d)a/d}$. With $a \ge \frac{d}{2(2+d)}$, we have $m_k^{-(2+d)a/d} \le 1/\sqrt{k}$, so the summand is at most $\alpha^{(2+d)/2}/k$, yielding $\mathcal{O}(\log K)$ at the boundary value of $a$ (and $\mathcal{O}(1)$ for any $a > \frac{d}{2(2+d)}$).
\item All remaining terms in $s_k$ involve strictly higher negative powers of $k$ (each $\le k^{-(3+d)/2}$) and yield convergent sums bounded by a constant.
\end{itemize}
Therefore, $\sum_{k=1}^K s_k = \mathcal{O}\left((\log K)^{3/2}\right)$.

\paragraph{Conclusion.}
Substituting into~\eqref{eq:borne_cvgce_sto}:
$$
\mathbb{E}[\hat{\rho}_{K}] \leq \CC \left(\frac{(\log K)^{3/2}}{\alpha^{(2+d)/2}\sqrt{K}}\right)^{\frac{1}{2+d}} = \CC \alpha^{-1/2}\left(\frac{(\log K)^3}{K}\right)^{\frac{1}{2(2+d)}}.
$$
Since $N = \sum_{k=1}^K k = K(K+1)/2$, we have $K = \Theta(\sqrt{N})$, and the sample complexity follows.
\end{proof}

\section{Extension to $\beta > 0$: position updates in the deterministic CPGD}
\label{sec:beta_positive}

In Section~\ref{s:continuous}, the deterministic analysis is carried out under the restriction $\beta = 0$, meaning that only the weights of the measure are updated at each iteration. In this appendix, we extend the analysis to the case $\beta > 0$, where the Push-Forward update (Definition~\ref{def:PF}) also moves the positions of the particles. Throughout this section, we assume that $(\alpha,\beta)$ satisfies condition~\eqref{eq:small_learning_rates}. The transition $\nuk \longmapsto \nukp$ now reads:
\begin{equation}
\label{eq:cgpd_beta}
\nukp = \Tnub^\sharp \Tnuka \nuk,
\end{equation}
where $\Tnub^\sharp$ denotes the push-forward by the proximal position update $\Tnub$ (Definition~\ref{def:PF}).

\subsection{Total Variation boundedness for $\beta > 0$}
\label{s:TV_beta}

The TV-norm boundedness established in Proposition~\ref{prop:hypdist}~$i)$ carries over to $\beta > 0$ without modification. Indeed, the push-forward map preserves the total variation norm: for any $\mu \in \cM_+(\cX)$,
\[
\| \Tnub^\sharp \mu \|_{\mathrm{TV}} = \| \mu \|_{\mathrm{TV}},
\]
since the push-forward merely redistributes mass without creating or destroying it. Therefore:
\[
\| \nukp \|_{\mathrm{TV}} = \| \Tnub^\sharp \Tnuka \nuk \|_{\mathrm{TV}} = \| \Tnuka \nuk \|_{\mathrm{TV}} = \int_\cX e^{-\alpha J'_{\nuk}(\vt)} d\nuk(\vt),
\]
which is the same starting point as in the proof of Proposition~\ref{prop:hypdist}~$i)$ (Section~\ref{s:proof_hypdist}). The remainder of that proof—the case analysis on whether $\|\nuk\|_{\mathrm{TV}}$ exceeds the threshold~$\mathfrak{M}$ defined there, leading to the bound $\|\nuk\|_{\mathrm{TV}} \leq \CTV$—applies verbatim, since it depends on $\nuk$ only through its TV-norm. %

\subsection{Smoothness assumptions for $\beta > 0$}
\label{s:smoothness_beta}

We verify that Assumptions~\eqref{eqs:H_eps_determinist} remain valid when the transition $\nukp \longrightarrow \nukpp$ is applied after a push-forward update with $\beta > 0$.

\paragraph{Assumption $(\mathcal{H}_\varepsilon^+)$.} This assumption concerns the birth process $\nukp \longrightarrow \nukpp$ only, and requires that
\[
\nukpp \bm{1}_{\{J'_{\nukp} \leq 0\}} \geq \lambda \epk \bm{1}_{\{J'_{\nukp} \leq 0\}}.
\]
The definition of the sets $\mathcal{N}_{\nukp} = \{J'_{\nukp} \leq 0\}$ depends on the post-update measure $\nukp$, which now incorporates the push-forward. This is the most delicate point and is discussed separately in Section~\ref{s:screening_beta} below.

\paragraph{Assumption $(\mathcal{H}_\varepsilon^{\mathrm{smooth},1})$.} The proof of Proposition~\ref{prop:hypdist}~$ii)$ (Section~\ref{s:proof_delet}) relies on the decomposition
\[
J(\nukpp) - J(\nukp) \leq \| \Phi(\nukp \bm{1}_{\mathcal{P}_{\nukp}}) \|_\mathbb{H}^2 + \| \Phi(\epk \bm{1}_{\mathcal{N}_{\nukp}} \lambda) \|_\mathbb{H}^2,
\]
where $\mathcal{P}_{\nukp} \subset \{J'_{\nukp} > 0\}$ and $\mathcal{N}_{\nukp} \subset \{J'_{\nukp} < 0\}$. These set inclusions hold by definition regardless of $\beta$. The subsequent bounds rely on the TV-norm of $\nukp$ (which is preserved by the push-forward, as shown above) and on the inclusion $\mathcal{P}_{\nukp}\subset\{J'_{\nuk}>-2\alpha^{-1}\log\varepsilon_k\}$ of Remark~\ref{rem:weight_update_perturbation}; the latter inclusion is sensitive to $\beta$ through the relationship between $\nuk$ and $\nukp$, and its $\beta>0$ counterpart is established in Section~\ref{s:screening_beta} below. Modulo this re-verification, Assumption $(\mathcal{H}_\varepsilon^{\mathrm{smooth},1})$ holds with the same constant $\CC$ up to a bounded multiplicative factor $e^{\alpha\beta\Lip^2}$.

\paragraph{Assumption $(\mathcal{H}_\varepsilon^{\mathrm{smooth},2})$.} The bound
\[
\| J'_{\nukpp} - J'_{\nukp} \|_\infty \leq \|\nuk\|_{\mathrm{TV}} \epk^2 + \epk \lambda(\cX)
\]
is obtained via the Cauchy--Schwarz inequality applied to the feature map $\Phi$, and depends on $\nukp$ only through its TV-norm. The argument is therefore identical to the $\beta = 0$ case.  %

\subsection{Extended one-step descent for $\beta > 0$}
\label{s:descent_beta}

\begin{proposition}
\label{prop:boundvk_beta}
Under Assumptions~\eqref{eqs:H_eps_determinist} and~\eqref{eq:HTVC}, if $(\alpha,\beta)$ satisfies~\eqref{eq:small_learning_rates} and $\vkd^2 \geq 24 \epkm^2 \CC^2$, then:
\[
J(\nukp) - J(\nuk) \leq -\frac{3\alpha}{2} (2\Lip)^{-d} |\vkd|^{2+d} \epkm.
\]
\end{proposition}

\begin{proof}
The proof follows that of Proposition~\ref{prop:boundvk} with a single modification: we apply Proposition~\ref{prop:incre} with $\beta > 0$ instead of $\beta = 0$. The descent property yields:
\[
J(\nukp) - J(\nuk) \leq -\frac{3}{4} \left( \alpha \int_\cX |J'_{\nuk}|^2 d\nuk + \beta \int_\cX \|\pi_\cX(\vt, \nabla J'_{\nuk}(\vt), \beta)\|^2 d\nuk \right).
\]
Since the second term is non-positive, dropping it only weakens the upper bound by an amount of additional descent, and we obtain:
\[
J(\nukp) - J(\nuk) \leq -\frac{3}{4} \alpha \int_\cX |J'_{\nuk}|^2 d\nuk = -\frac{3}{4} \|g_{\nuk}^\alpha\|^2_{L^2(\nuk)}.
\]
This is exactly inequality~\eqref{eq:inter-inter1}, and the remainder of the proof of Proposition~\ref{prop:boundvk} (Eqs~\eqref{eq:inter1}--\eqref{eq:inter3}) proceeds without change.
\end{proof}

\subsection{Convergence rates for $\beta > 0$}
\label{s:convergence_beta}

Since Proposition~\ref{prop:boundvk_beta} yields the same bound as Proposition~\ref{prop:boundvk}, the downstream results—namely Proposition~\ref{prop:upper-bound-increment} and Theorem~\ref{theo:convergence_deterministe}—extend to $\beta > 0$ with identical rates.

The proof of Theorem~\ref{theo:convergence_deterministe} (Section~\ref{s:proofconv1}) uses two ingredients: (i) Assumption $(\mathcal{H}_\varepsilon^{\mathrm{smooth},1})$, which is verified in Section~\ref{s:smoothness_beta}, and (ii) the descent property $J(\nukp) - J(\nuk) \leq 0$, which follows from Proposition~\ref{prop:incre} for any $(\alpha,\beta)$ satisfying~\eqref{eq:small_learning_rates}. The telescoping sum argument and the optimization over $\varepsilon_k$ are unchanged.

\subsection{The screening issue for $\beta > 0$}
\label{s:screening_beta}

We now address the most delicate point: verifying that the birth process construction of Section~\ref{s:birth-death-proc} still satisfies Assumption~$(\mathcal{H}_\varepsilon^+)$ and that the smoothness bounds~\eqref{eqs:H_eps_determinist} remain valid when $\nukp$ is computed with $\beta > 0$.

\paragraph{Assumption $(\mathcal{H}_\varepsilon^+)$ is satisfied by construction.}
Recall that the mass creation step~\eqref{def:nukpp} adds $\epk \bm{1}_{\mathcal{N}_{\nukp}} \lambda$ on $\mathcal{N}_{\nukp} = \{J'_{\nukp} \leq 0\}$. Since $\nukpp = \nukkp + \epk \bm{1}_{\mathcal{N}_{\nukp}} \lambda$ with $\nukkp \geq 0$, we obtain
\[
\nukpp \bm{1}_{J'_{\nukp} \leq 0} \geq \epk \lambda \bm{1}_{J'_{\nukp} \leq 0},
\]
which is precisely Assumption~$(\mathcal{H}_\varepsilon^+)$. This holds regardless of $\beta$, since the birth process is defined in terms of the actual measure $\nukp$, whatever its construction.

\paragraph{Perturbation of $J'_{\nukp}$ by the push-forward.}

The more subtle issue arises in the smoothness proofs (Section~\ref{s:proof_delet}), which use the specific relationship between $\nuk$, $\nukp$, and the set $\mathcal{P}_{\nukp}$. We denote by $\tilde{\nu}_{k^+} = \Tnuka\nuk$ the intermediate reweighted measure (before position update), so that $\nukp = \Tnub^\sharp \tilde{\nu}_{k^+}$.

We control the perturbation induced by the push-forward using a Taylor expansion. For any $\vt \in \cX$:
\begin{equation}
\label{eq:J_prime_pushforward_perturbation}
J'_{\nukp}(\vt) - J'_{\tilde{\nu}_{k^+}}(\vt) = \langle \varphi_\vt, \Phi(\nukp - \tilde{\nu}_{k^+}) \rangle_\mathbb{H}.
\end{equation}
Since $\nukp = \Tnub^\sharp \tilde{\nu}_{k^+}$, the change-of-variables formula gives:
\[
\Phi(\nukp - \tilde{\nu}_{k^+}) = \int_\cX \left[ \varphi_{\Tnub(\vt)} - \varphi_\vt \right] d\tilde{\nu}_{k^+}(\vt).
\]
By the Lipschitz property of the feature map (Lemma~\ref{lem:lipschitz_feature_map}), $\|\varphi_{\Tnub(\vt)} - \varphi_\vt\|_\mathbb{H} \leq \sqrt{\CC_{\mathcal{P}}} \|\Tnub(\vt) - \vt\|$. Recalling the definition $\Tnub(\vt) = \vt - \beta \pi_\cX(\vt, \nabla J'_{\nuk}(\vt), \beta)$ and the bound $\|\pi_\cX(\vt, \nabla J'_{\nuk}(\vt), \beta)\| \leq \|\nabla J'_{\nuk}\|_\infty \leq \Lip$ (Lemma~\ref{lem:lipschitz_J_prime}), we obtain:
\[
\|\Phi(\nukp - \tilde{\nu}_{k^+})\|_\mathbb{H} \leq \beta \sqrt{\CC_{\mathcal{P}}} \Lip \, \|\tilde{\nu}_{k^+}\|_{\mathrm{TV}} \leq \beta \sqrt{\CC_{\mathcal{P}}} \Lip \, \CTV.
\]
Combined with~\eqref{eq:J_prime_pushforward_perturbation} and $\|\varphi_\vt\|_\mathbb{H} = 1$:
\begin{equation}
\label{eq:beta_perturbation_bound}
\|J'_{\nukp} - J'_{\tilde{\nu}_{k^+}}\|_\infty \leq \beta \sqrt{\CC_{\mathcal{P}}} \Lip \, \CTV.
\end{equation}

\paragraph{Extension of the smoothness bounds.}
It remains to verify that Assumptions~$({\mathcal{H}_\varepsilon^{\mathrm{smooth},1}})$ and~$({\mathcal{H}_\varepsilon^{\mathrm{smooth},2}})$ still hold with the same order. In both cases, the key observation is the following: with $\beta > 0$, integrals against $\nukp$ over a set $\mathcal{A} \subset \cX$ are expressed via the push-forward as
\[
\int_{\mathcal{A}} d\nukp(\vt') = \int_{\Tnub^{-1}(\mathcal{A})} e^{-\alpha J'_{\nuk}(\vt)} d\nuk(\vt).
\]
On the preimage $\Tnub^{-1}(\mathcal{P}_{\nukp})$, since $\Tnub(\vt) \in \mathcal{P}_{\nukp} \subset \{J'_{\nuk} > -2\alpha^{-1} \log \varepsilon_k\}$ (by Remark~\ref{rem:weight_update_perturbation}), the Lipschitz continuity of $J'_{\nuk}$ (Lemma~\ref{lem:lipschitz_J_prime}) yields:
\begin{equation}
\label{eq:J_prime_lipschitz_shift}
J'_{\nuk}(\vt) \geq J'_{\nuk}(\Tnub(\vt)) - \Lip \|\Tnub(\vt) - \vt\| > -2\alpha^{-1} \log \varepsilon_k - \beta \Lip^2,
\end{equation}
using the bound $\|\Tnub(\vt) - \vt\| = \beta \|\pi_\cX(\vt, \nabla J'_{\nuk}(\vt), \beta)\| \leq \beta \Lip$ (Lemma~\ref{lem:lipschitz_J_prime}). Therefore:
\begin{equation}
\label{eq:exp_bound_beta}
\forall \vt \in \Tnub^{-1}(\mathcal{P}_{\nukp}): \qquad e^{-\alpha J'_{\nuk}(\vt)} < e^{\alpha \beta \Lip^2}\, \epk^2.
\end{equation}
Under condition~\eqref{eq:small_learning_rates}, the product $\alpha \beta \Lip^2$ is uniformly bounded by a constant depending only on $\CC_\mathcal{P}$, $\CTV$, and $\kappa$.

\medskip

\noindent \underline{Assumption $({\mathcal{H}_\varepsilon^{\mathrm{smooth},1}})$.} As in the proof of Proposition~\ref{prop:hypdist}~$ii)$ (Section~\ref{s:proof_delet}), the Cauchy--Schwarz inequality gives:
\[
\| \Phi(\nukp \bm{1}_{\mathcal{P}_{\nukp}}) \|_\mathbb{H}^2
= \left\| \int_{\mathcal{P}_{\nukp}} \varphi_{\vt'} d\nukp(\vt') \right\|_\mathbb{H}^2
\leq \nukp(\mathcal{P}_{\nukp}) \times \int_{\mathcal{P}_{\nukp}} \| \varphi_{\vt'} \|_\mathbb{H}^2 d\nukp(\vt').
\]
Since $\|\varphi_{\vt'}\|_\mathbb{H} = 1$, both factors equal $\nukp(\mathcal{P}_{\nukp})$. Using the push-forward and the bound~\eqref{eq:exp_bound_beta}:
\[
\nukp(\mathcal{P}_{\nukp}) = \int_{\Tnub^{-1}(\mathcal{P}_{\nukp})} e^{-\alpha J'_{\nuk}(\vt)} d\nuk(\vt) \leq e^{\alpha \beta \Lip^2}\, \epk^2 \|\nuk\|_{\mathrm{TV}}.
\]
Therefore $\| \Phi(\nukp \bm{1}_{\mathcal{P}_{\nukp}}) \|_\mathbb{H}^2 \leq e^{2\alpha \beta \Lip^2}\, \epk^4 \|\nuk\|_{\mathrm{TV}}^2 \leq \CC \epk^2 \|\nuk\|_{\mathrm{TV}}^2$,
where the last inequality uses $\epk \leq 1$. The bound on $\|\Phi(\epk \bm{1}_{\mathcal{N}_{\nukp}} \lambda)\|_\mathbb{H}^2 \leq \epk^2 \lambda(\cX)^2$ is unchanged (it depends only on $\epk$ and $\lambda(\cX)$). Therefore, Assumption~$(\mathcal{H}_\varepsilon^{\mathrm{smooth},1})$ holds with a modified constant~$\CC$.

\medskip

\noindent \underline{Assumption $({\mathcal{H}_\varepsilon^{\mathrm{smooth},2}})$.} Following the proof in Section~\ref{s:proof_delet}:
\begin{align*}
\left| J'_{\nukpp}(\vt) - J'_{\nukp}(\vt) \right|
& \leq \int_{\mathcal{P}_{\nukp}} \left| \langle \varphi_u, \varphi_\vt \rangle_\mathbb{H} \right| d\nukp(u) + \epk \int_{\mathcal{N}_{\nukp}} \left| \langle \varphi_u, \varphi_\vt \rangle_\mathbb{H} \right| d\lambda(u).
\end{align*}
The second term is bounded by $\epk \lambda(\cX)$ as before. For the first term, using $|\langle \varphi_u, \varphi_\vt \rangle| \leq 1$ and the push-forward:
\[
\int_{\mathcal{P}_{\nukp}} d\nukp(u) = \int_{\Tnub^{-1}(\mathcal{P}_{\nukp})} e^{-\alpha J'_{\nuk}(\vt)} d\nuk(\vt) \leq e^{\alpha \beta \Lip^2}\, \epk^2 \|\nuk\|_{\mathrm{TV}},
\]
where we used~\eqref{eq:exp_bound_beta}. This yields:
\[
\|J'_{\nukpp} - J'_{\nukp}\|_\infty \leq e^{\alpha \beta \Lip^2}\, \|\nuk\|_{\mathrm{TV}} \epk^2 + \epk \lambda(\cX) \leq \CC \epk,
\]
since $e^{\alpha \beta \Lip^2}$ is bounded under~\eqref{eq:small_learning_rates} and $\epk \leq 1$.

\subsection{Global convergence with position updates}
\label{s:global_convergence_beta}

We now state the main result of this section, which extends Theorem~\ref{theo:convergence_deterministe} to the case $\beta > 0$. The transition $\nuk \longmapsto \nukp$ is given by the full Weight \& Push-Forward update~\eqref{eq:cgpd_beta}, and the birth process $\nukp \longrightarrow \nukpp$ follows the construction of Section~\ref{s:birth-death-proc}.

\begin{theo}[Global convergence for $\beta > 0$]
\label{theo:convergence_beta}
Assume~\eqref{eqs:hyp_HP} with $\kappa\geq 0$. Let the sequence $(\nuk)_{k \geq 1}$ be generated by the update $\nukp = \Tnub^\sharp \Tnuka \nuk$ followed by the birth process~\eqref{def:nukpp_alt1}--\eqref{def:nukpp}, with $(\epk)_{k\ge 0}$ satisfying $\epk\le\UU\alpha$ for every $k\ge 0$. If $(\alpha, \beta)$ satisfies condition~\eqref{eq:small_learning_rates}, then Assumptions~\eqref{eqs:H_eps_determinist} and~\eqref{eq:HTVC} hold (with constants $\CC$ depending on $\alpha$, $\beta$, $\CC_\mathcal{P}$, $\CTV$, and $\kappa$, and differing from the $\beta=0$ case by bounded multiplicative factors of order $e^{\alpha\beta\Lip^2}$), and for any final horizon $K \geq 2$:
\begin{itemize}
\item[$i)$] If $(\epk)_{k\ge 0}$ is \textit{non-adaptive} and $\epk = \varepsilon = \sqrt{\CC / K}\le\alpha$ for all $k \in \{1, \ldots, K\}$, then:
\[
\min_{1 \leq k \leq K} \left\{ J(\nuk) - J(\nu^\star) \right\} \leq \CC\, \Lip^{\frac{2+2d}{2+d}}\, \alpha^{-\frac{1}{2+d}}\, K^{-\frac{1}{2(2+d)}}.
\]
\item[$ii)$] If $(\epk)_{k\ge 0}$ is \textit{horizon-free} and $\epk = \sqrt{\CC / (k+1)}\le\alpha$, then:
\[
\min_{1 \leq k \leq K} \left\{ J(\nuk) - J(\nu^\star) \right\} \leq \CC\, \Lip^{\frac{2+2d}{2+d}}\, \alpha^{-\frac{1}{2+d}}\, K^{-\frac{1}{2(2+d)}}\, \log(K)^{\frac{1}{(2+d)}}.
\]
\item[$iii)$] If $\epk = \varepsilon = \CC \left( \frac{\Lip^{2+2d}}{(d+1)\alpha} \right)^{\frac{1}{5+2d}} K^{-\frac{3+d}{5+2d}}$, then:
\[
J(\nu_K) - J(\nu^\star) \leq \CC \left( \frac{\Lip^{2+2d}}{(d+1)\alpha} \right)^{\frac{2}{5+2d}} K^{-\frac{1}{5+2d}}.
\]
\end{itemize}
In all three items, the generic constant $\CC$ may depend polynomially on $(\|\nu^\star\|_{\mathrm{TV}}/\Lip)$; in item~$iii)$, $\CC$ also depends polynomially on the initial excess $J(\nu_1)-J^\star$ (which is finite under~$\HTVC$).
\end{theo}

\begin{proof}
We have established in Sections~\ref{s:TV_beta}--\ref{s:screening_beta} that Assumptions~\eqref{eqs:H_eps_determinist} and~\eqref{eq:HTVC} hold under the update~\eqref{eq:cgpd_beta} with $\beta > 0$, with constants $\CC$ that may differ from the $\beta = 0$ case by bounded multiplicative factors (specifically, factors of $e^{\alpha \beta \Lip^2}$, which are uniformly bounded under~\eqref{eq:small_learning_rates}).

With these assumptions verified, Proposition~\ref{prop:boundvk_beta} yields the same one-step descent bound as Proposition~\ref{prop:boundvk} (up to the modified constants), and the proof of Theorem~\ref{theo:convergence_deterministe} (Section~\ref{s:proofconv1}) applies verbatim: the telescoping sum argument~\eqref{eq:inegalite_importante} and the optimization over $\varepsilon_k$ depend only on the structure of Assumptions~\eqref{eqs:H_eps_determinist}, not on the specific value of~$\beta$. The three convergence rates therefore coincide with those of Theorem~\ref{theo:convergence_deterministe}.
\end{proof}

\section{Statistical guarantees}
\label{sec:statistical_guarantees}
\begin{defi}[Sparse target measures]
Let $\Delta^0$ be positive and let $s^0$ be greater than $1$. We define the class of $s^0$-sparse measures with minimal Euclidean separation $\Delta^0$ as
	\begin{equation}
   \label{eq:class_model}
	\mathbb M_{s^0, \Delta^0} 
       := 
       \Bigg\{ \mu^0\,:\,
           \mu^0 = \sum_{k=1}^{s^0} a_k^0 \delta_{x^0_k}
              \textnormal{ and }
           \min_{k\neq l} 
           \|x^0_k-x^0_l\|_2
           \geq \Delta^0
       \Bigg\}\,,
	\end{equation}
where $\delta_{x}$ denotes the Dirac mass at point $x\in\mathcal X$ and at least one $a_k^0\in\mathds R$ is non-zero. 
\end{defi}

In statistical learning theory, one considers a target sparse measure $\mu^0\in\mathbb M_{s^0, \Delta^0}$ and one defines the noise term and, respectively, the noise level as
\begin{subequations}\label{eq:noise_level_term}
\begin{align}
   \Gamma &:= y-\Phi\mu^0\label{eq:noise}\\
   \text{resp. }\gamma &:=\|\Gamma\|_{\mathbb{H}}\,.\label{eq:noise_level}
\end{align}
\end{subequations}
The statistical estimation error bounds of \eqref{eq:blasso} are defined by means of the so-called far and near regions. 
\begin{subequations}\label{eq:near_far_def}
\begin{defi}[Far and Near regions]
\label{def:Far_and_Near_regions}
	Let ${\mu^0} \in {\mathbb M}_{{s^0}, \Delta^0}$ and let $r > 0$. Define the near region of $x^0_k$ of radius~$r$ as
	\begin{equation}
       \label{eq:near_regions}
	\mathbb{N}_k(r) := \left\{ x \in {\mathcal X}, \quad 
       \|x-x^0_k\|_2
       \leq r \right\},
	\end{equation}
	and the far region as
	\begin{equation}
        \label{eq:far_regions}
	\mathbb{F}(r) := {\mathcal X} \setminus \mathbb{N}(r), \quad \textnormal{with: } \mathbb{N}(r) := \bigcup_{k=1}^{{s^0}} \mathbb{N}_k(r).
	\end{equation}
\end{defi}
\end{subequations}
In the literature \citep{azais2015spike,candes2014towards,poon2023geometry,de2025effective}, the estimation errors (with respect to the Euclidean metric) are proven to be, under some conditions, for some radius $r>0$ such that $r<\Delta^0/2$, 
\begin{subequations}\label{eq:error_bounds}
\begin{itemize}
   \item Control of the far region: 
   \begin{equation}
   \label{eq:control_far}
       |\mu^\star|(\mathbb{F}(r))\lesssim_d \gamma \sqrt{{s^0}}\,,
   \end{equation}
   \item Control of all the near regions:
   \begin{equation}
   \label{eq:control_near}
       |\mu^\star(\mathbb{N}_k(r))-a^0_k|\lesssim_d \gamma \sqrt{{s^0}}\,,
   \end{equation}
   \item Detection level: For every Borel set $A\subset{\mathcal X}$ such that $|\mu^\star|(A)\gtrsim_d \gamma \sqrt{{s^0}}$, there exists $x^0_k$ such that
   \begin{equation}
   \label{eq:detection_near}
       \min_{t\in A}\|t-x^0_k\|_2\lesssim_d r\,,
   \end{equation}    
\end{itemize}
\end{subequations}
where $\lesssim_d$ denotes the inequality up to a multiplicative constant that may depend on the dimension $d$, $\mu^\star$ is a solution to \eqref{eq:blasso} with regularization parameter $\kappa\sim \gamma /\sqrt{{s^0}}$ and $|\mu^\star|$ denotes the absolute part of $\mu^\star$. These statistical estimation error bounds hold for all $\mu\in\mathcal M(\mathcal X)$ such that
\begin{equation}
   \label{eq:objective_range}
   0\leq J(\mu) - J(\mu^\star)\leq J(\mu^0) - J(\mu^\star)=:\varepsilon^\star\,,
\end{equation}
see for instance \cite[Theorem 1]{decastro2024four}, under some conditions (see \cite[Assumption~1]{poon2023geometry}). Extensions to Fisher-Rao metrics are given in \cite{poon2023geometry} and \cite{giard2025gaussian}.

We call an $\varepsilon$-solution any $\mu\in\mathcal{M}(\mathcal X)$ such that $0\leq J(\mu)-J(\mu^\star)\leq \varepsilon$. From an optimization point of view, Equation~\eqref{eq:objective_range} shows that there exists $ \varepsilon^\star>0$ such that any $\varepsilon$-solution satisfies the statistical error bounds~\eqref{eq:error_bounds}, for $0\leq \varepsilon\leq \varepsilon^\star$. While the parameter $\varepsilon^\star$ is not observed in practice, it shows that any gradient descent path converging towards a solution will satisfy the statistical error~\eqref{eq:error_bounds} after a finite number of steps.

\section{List of notation}
\label{app:list_notation}

\begin{table}[!ht]
	\centering
	{\small
    \begin{tabularx}{\linewidth}{@{}lX@{}}
        \toprule
        \multicolumn{2}{c}{\textit{General Notation \& Measure Spaces}} \\
        \midrule
        $d$, $\mathcal X$                                     & Dimension and domain ($\mathcal X \subset \bbR^d$ is compact convex)\\
        $\cc,\CC$                               & Generic positive constants\\
        $\lambda$                               & Lebesgue measure on $\mathcal{X}$\\
        $\mathcal{C}(\mathcal X)$                   & Space of continuous functions on $\mathcal X$ equipped with $\|\cdot\|_\infty$\\
        $\mathcal{M}(\mathcal X)$, $\mathcal{M}_+(\mathcal X)$                   & Spaces of signed and non-negative Radon measures on $\mathcal X$ equipped with $\|\cdot\|_{\mathrm{TV}}$\\
        
        \midrule
        \multicolumn{2}{c}{\textit{Continuous Sparse Regression Framework}} \\
        \midrule
        $\mathbb H, y$                                 & Separable Hilbert space and observation vector in $\mathbb H$\\
        $\Phi$                                      & Linear measurement operator $\mathcal{M}(\mathcal X) \to \mathbb H$\\
        $\varphi_t$                                 & Feature map $t \longmapsto \varphi_t \in \mathbb H$\\
        $K(\cdot, \cdot)$                           & Model kernel defined by $K(s,t) = \langle \varphi_s, \varphi_t \rangle_{\mathbb H}$\\
        $\kappa$                                    & Regularization parameter of problem~\eqref{eq:blasso}\\
        $J(\nu)$                                    & Objective function over measures\\
        $J'_\nu$                                    & Dual certificate (Fréchet derivative of $J$ at $\nu$)\\
        $\mu^\star, \nu^\star$                               & Optimal signed and non-negative measures (minimizers of $J$)\\
        $\rho_K$                                    & Minimum excess loss along the deterministic trajectory, $\min_{0\le k\le K}[J(\nu_k)-J(\nu^\star)]$\\

        \midrule
        \multicolumn{2}{c}{\textit{Regularity \& Complexity Constants}} \\
        \midrule
        $\cc_{\mathcal{P}},\CC_{\mathcal{P}}$       & Kernel smoothness bounding constants (Assumption~$\mathcal{H}_{\mathcal{P}}$)\\
        $\CTV$                                      & Uniform bound on TV norm along trajectories (Assumption~\eqref{eq:HTVC})\\
        $\Lip$                                      & Uniform Lipschitz constant for $J'_\nu$, see~\eqref{def:lip_uniform_J'}\\
        $\MM$                                       & Uniform $L^\infty$ bound on the per-sample dual-certificate estimator (Assumption~(A2))
        \\

        \midrule
        \multicolumn{2}{c}{\textit{Conic Particle Gradient Descent Algorithm}} \\
        \midrule
        $K$                                     & Total number of iterations (horizon)\\
        $p_k$, $p_0$                            & Number of active particles at iteration $k$ and at initialization\\
        $\alpha,\beta$                          & Learning rates for weight and position updates\\
        $\nu_k, \nu_{k^+}, \nu_{k+1}$           & Measure estimates at iteration $k$, post-weight update, and next iteration\\
        $\Tnua, \Tnub$                          & Weight exponential update and position push-forward update operators\\
        $\pi_{\mathcal X}$                            & Generalized gradient descent step on $\mathcal X$\\
        $\Delta_k, \Delta_{k^+}$                     & Potential energy descents along iterations\\

        \midrule
        \multicolumn{2}{c}{\textit{Exploration: Birth and Death Processes}} \\
        \midrule
        $\mathcal{P}_\nu$, $\mathcal{N}_\nu$         & Regions of positivity (particle death) and negativity (particle birth) of $J'_\nu$\\
        $\varepsilon_k$                         & Exploration schedule / birth rate parameter\\
        $\poschedule$, $\negschedule$           & Threshold parameters for particle death and birth processes\\
        $U_{k+1}, V_{k+1}$                      & Random candidate points proposed for birth and for death at iteration $k+1$\\

        \midrule
        \multicolumn{2}{c}{\textit{FS\&P Stochastic Estimators}} \\
        \midrule
        $\bm{W}, \bm{T}$                        & Vectors of particle weights ($p_k$ components) and positions ($p_k \times d$)\\
        $\hat{\nu}_k, \nukps, \nukkpps, \nukpps$ & Stochastic counterparts of $\nu_k, \nu_{k^+}, \nu_{k^{++}}, \nu_{k+1}$ (current iterate, post-weight update, post-death, and next iterate)\\
        $m_k$, $N$                              & Mini-batch size at iteration $k$ and total oracle count $N=\sum_{k=1}^K m_k$\\
        $\widehat{J'_k}, \widehat{D_k}$         & Stochastic estimators of the dual certificate and its gradient\\
        $\Pkp, \Nkp$  & Empirical regions of particle death and birth\\
        $\widehat{\Delta}_k, \widehat{\Delta}_{k^+}$ & Empirical potential energy descents\\
        $\mathfrak{F}_k,\mathfrak{F}_k^+$       & Filtrations generated by the algorithm up to step $k$ and post-update $k^+$\\
        $\hat\rho_K$                            & Minimum stochastic excess loss along $K$ iterations, see~\eqref{eq:global_conv_rate_sto} %
        \\
        \bottomrule
  \end{tabularx}
  }
\caption{List of notation}
\label{tab:notation}
\end{table}

\end{document}

%% file: maths_commands_jmlr.tex
\usepackage{amsmath,amsfonts,amssymb,bm}
\usepackage{dsfont}

\def\eq{\begin{equation}}
\def\qe{\end{equation}}

\def\1{\bm{1}}

\def\vs{{{s}}}
\def\vt{{{t}}}

\def\mD{{{D}}}

\def\mJ{{{J}}}

\DeclareMathAlphabet{\mathsfit}{\encodingdefault}{\sfdefault}{m}{sl}
\SetMathAlphabet{\mathsfit}{bold}{\encodingdefault}{\sfdefault}{bx}{n}

\newcommand{\R}{\mathds{R}}

\DeclareMathOperator*{\argmin}{arg\,min}